\documentclass[11pt]{article}

\usepackage[utf8]{inputenc}
\usepackage{cite}
\usepackage{color}
\usepackage{latexsym}
\usepackage{dsfont}
\usepackage{amssymb}
\usepackage{comment}
\usepackage{graphicx}
\usepackage{amsmath}
\usepackage{amsthm}
\usepackage{amssymb}
\usepackage{bm}

\usepackage{hyperref}
\usepackage{appendix}
\usepackage[T1]{fontenc}
\usepackage{babel}
\usepackage{bbm}
\usepackage{subfigure}
\usepackage{color}

\usepackage{stmaryrd}

\usepackage[algo2e,ruled,vlined]{algorithm2e} 
 \usepackage{algorithm}
 \usepackage{algorithmicx}
\usepackage{algpseudocode}
\usepackage{ marvosym }
\usepackage{csquotes}
\usepackage{hyperref}

\usepackage{comment}

\usepackage{tikz}
\usetikzlibrary{fit,matrix,chains,positioning,decorations.pathreplacing,arrows}

\usepackage{mathtools}
\mathtoolsset{showonlyrefs}

\usepackage[a4paper,margin=1in,heightrounded]{geometry}

\usepackage{algpseudocode}
\usepackage{algorithm}

\usepackage[normalem]{ulem}

\def \trans{^{\scriptscriptstyle{\intercal}}}

\def \Inf{\displaystyle\inf}

\def \b1{\bf{1}}
\def \tildJ{\tilde{J}}

\def \B{\mathbb{B}}
\def \I{\mathbb{I}}

\def \R{\mathbb{R}}

\def \M{\mathbb{M}}

\def \E{\mathbb{E}}
\def \F{\mathbb{F}}

\def \P{\mathbb{P}}

\def \S{\mathbb{S}}

\def \vec{\text{vec}}

\def \bpi{\boldsymbol{\pi}}

\def \bpi{\boldsymbol{\pi}}

\def \d{\mathrm{d}}

\def\esssup_#1{\underset{#1}{\mathrm{ess\,sup\, }}}

\def\argmin_#1{\underset{#1}{\mathrm{argmin\, }}}
\def\argmax_#1{\underset{#1}{\mathrm{argmax\, }}}
\def\Err{{\textit{Err}}}

\def \Ac{{\cal A}}

\def \Dc{{\cal D}}

\def \Fc{{\cal F}}
\def \Gc{{\cal G}}

\def \Jc{{\cal J}}
\def \Kc{{\cal K}}

\def \Pc{{\cal P}}
\def \Rc{{\cal R}}

\def \Nc{{\cal N}}

\def \Sc{{\cal S}}

\def\Brm{\mathrm{B}}

\def \ep{\hbox{ }\hfill$\Box$}
\def\xib{{\boldsymbol{\xi}}}

\def\bpi{{\boldsymbol \pi}}

\def\Bd{\textbf{Bd}}
\def\diag{\text{diag}}

\def\bC{{\bf C}}

\def\bop{{\boldsymbol p}}

\def\bw{{\bf w}}

\def\xb{{\bf x}}
\def\yb{{\bf y}}
\def\zb{{\bf z}}

\def \Ob{{\bf O}}
\def \Var{\text{Var}}
\def\Gb{{\bf G}}
\def\Qc{\mathcal{Q}}

\def \d{\mathrm{d}} 
\def \h{\mathrm{h}}

\def \dt{\d t}
\def\Itm{\text{Itm}}
\def\b{\mathrm{b}}

\def \Tr{\text{Tr}}

\def\beqs{\begin{eqnarray*}}
\def\enqs{\end{eqnarray*}}
\def\beq{\begin{eqnarray}}
\def\enq{\end{eqnarray}}
\def\*{\times}

\def\Pone{{(1)}}
\def\Ptwo{{(2)}}

\newcommand{\red}[1]{\textcolor{red}{#1}}
\newcommand{\bl}[1]{\textcolor{blue}{#1}}

\def\red#1{{\color{red}#1}}

\def\tr{\text{tr}}

\newtheorem{Theorem}{Theorem}[section]
\newtheorem{Definition}{Definition}[section]
\newtheorem{Proposition}{Proposition}[section]
\newtheorem{Assumption}{Assumption}[section]
\newtheorem{Lemma}{Lemma}[section]

\newtheorem{Remark}{Remark}[section]

\numberwithin{equation}{section}

\title{Full error  analysis of policy gradient learning algorithms 
for exploratory linear quadratic mean-field control problem in continuous time with common noise}

\author{Noufel FRIKHA\footnote{
Universit\'e Paris 1 Panth\'eon-Sorbonne, Centre d'Economie de la Sorbonne (CES), 106 Boulevard de l’H\^opital, 75642 Paris Cedex 13, \sf  noufel.frikha at univ-paris1.fr.
The work of this author has benefited from the support of the Institut Europlace de Finance.
} 
\and Huy\^en PHAM \footnote{LPSM, Universit\'e Paris Cit\'e and Sorbonne University, \sf pham at lpsm.paris. 
The work of this author is  partially supported by the BNP-PAR Chair ``Futures of Quantitative Finance", and the 
Chair Finance \& Sustainable Development / the FiME Lab (Institut Europlace de Finance)} 
\and Xuanye SONG\footnote{LPSM,  Universit\'e Paris Cité, \sf \href{mailto:pham at lpsm.paris}{xsong at lpsm.paris}; Research assistant at Hong Kong Polytechnic University.}}

\date{}

\begin{document}

\maketitle

\begin{abstract}
We consider reinforcement learning (RL) methods for finding optimal policies in linear quadratic (LQ) mean field control (MFC) problems over an infinite horizon in continuous time, with common noise and entropy regularization. We study policy gradient (PG) learning and first demonstrate convergence in a model-based setting by establishing a suitable gradient domination condition.
Next, our main contribution is a comprehensive error analysis, where we prove the global linear convergence and sample complexity of the PG algorithm with two-point gradient estimates in a model-free setting with unknown parameters. In this setting, the parameterized optimal policies are learned from samples of the states and population distribution.
Finally, we provide numerical evidence supporting the convergence of our implemented algorithms.  
\end{abstract}



\vspace{5mm}

\noindent {\bf Key words}:  Mean-field control; 
reinforcement learning; linear-quadratic; two-point gradient estimation; Polyak-Lojasiewicz inequality; 
gradient descent; sample complexity.

\tableofcontents

\section{Introduction}


The last decade has seen significant advances in solving optimal control of dynamical systems in unknown environments using reinforcement learning (RL) methods. The essence of RL is to learn optimal decisions through trial and error, which involves repeatedly trying a policy, observing the state, receiving and evaluating the reward, and subsequently improving the policy. There are two main approaches in RL: (i) {$Q$-learning}, which is based on dynamic programming, and (ii) {policy gradient} (PG), which is based on the parametrization of policies. A key feature of RL is the \textit{exploration} of the unknown environment to broaden the search space, achievable through randomized policies. RL is a very active branch of machine learning. For an overview of this field in the discrete-time setting, we refer to the second edition of the monograph \cite{sutbar18}, and for recent advances in the continuous-time setting, see \cite{PFandAC} and \cite{jiazhou23}. 

\textit{Mean-field control} (MFC), also known as the \textit{McKean-Vlasov} (MKV) control problem, is a class of stochastic control problems that focuses on the study of large population models of interacting agents who cooperate and act for collective welfare according to a central decision-maker (or social planner). This field has attracted growing interest over the last decade, resulting in a substantial body of literature on both its theory and its various applications in economics, finance, population dynamics, social sciences, and herd behavior. For a detailed treatment of the topic, we refer to the seminal two-volume monograph \cite{cardel19a, cardel19b}.

RL for MFC has recently attracted attention in the research community, see e.g. \cite{carlautan19a}, \cite{guetal20}, \cite{angfoulau21}, \cite{frietal23}, \cite{phawar23}. The challenge lies in accurately learning optimal policies and value functions defined on the infinite-dimensional space of probability measures. The mathematical understanding and convergence analysis of these RL algorithms are still in their infancy.

In this paper, we aim to address questions surrounding convergence and sample complexity, focusing on policy gradient methods in RL within the context of infinite horizon linear quadratic (LQ) MFC with common noise for continuous time systems. The LQ problem is indeed the cornerstone of optimal control theory due to its tractability and can be viewed as an approximation of more general nonlinear control problems. To encourage exploration in unknown environments, we employ randomized policies and add entropy regularization, following the approach of recent papers \cite{PFandAC}, \cite{guolixu23}, \cite{RLinContinuousTime}, \cite{giereizha22}, \cite{YZhangLSzpruch2023}.

\paragraph{Our main contributions.}  
Our paper proposes and analyzes convergent PG algorithms to solve infinite horizon exploratory LQ MFC problems in a continuous time setting, with common noise and entropy regularization. 
\begin{itemize}
 \item Our first contribution is to derive the explicit form of the optimal solution using coupled algebraic Riccati equations, thereby generalizing the results in \cite{MartingApproLQ} to include entropy regularization for randomized policies (Theorem \ref{OptRelaxedControlPbl}). Motivated by the explicit form of the optimal randomized policy, we reformulate the LQ MFC problem into a minimization problem over Gaussian policies. The mean of each Gaussian policy being linear in the state and conditional mean with respect to the common noise is parameterised using two matrix-valued coefficients $\Theta=(\theta, \zeta)$. The parameterized cost function is shown to be smooth and satisfy a gradient domination condition, also known as the Polyak-Lojasiewicz inequality (Propositions \ref{GradJ} and \ref{GradDomCond}) following the approach of \cite{fazetal18} and \cite{GCPFLQMF}. Such inequality is known to be crucial to ensure the convergence of PG algorithms in non-convex landscape.
 
\item We then propose and study PG methods in both exact and model-free settings. Our work provides theoretical guarantees of convergence for the gradient descent (GD) algorithms with suitable step sizes (Theorems \ref{ConvExGradDes} and \ref{ThmConvSGD}). In the model-free case, where the exact gradient is unavailable, we adapt the two-point gradient estimation method of \cite{ConvSampGradMethod} to our mean-field setting by relying on samples of discrete-time trajectories and population distributions. For the first time to the best of our knowledge, we provide a comprehensive error analysis accounting respectively for the error of perturbation with respect to the exact expected functional cost, for the horizon truncation, for the time and particle discretizations, for the statistical error and finally for the optimization error from gradient iterations, demonstrating global linear convergence with polynomial computational sample complexities.

\end{itemize}

\vspace{-2mm} 

\paragraph{Related works.} The closest papers related to our work are \cite{carlautan19a} and \cite{GCPFLQMF}. In \cite{carlautan19a}, the authors consider an LQ MFC problem with common noise in a discrete time setting and prove the convergence of PG algorithms for deterministic policies in both model-based and model-free settings. The paper \cite{GCPFLQMF} addresses an infinite-horizon time average LQ MFC control problem in a continuous time setting without an entropy regularizer and demonstrates the convergence of PG with deterministic policies in the exact model-based setting, using a varying step size at each iteration. Finally, it is worth mentioning that our proofs extend the arguments presented in \cite{fazetal18} and \cite{hamxuyan20} for the discrete time setting, covering both finite and infinite horizons, as well as those in \cite{ConvSampGradMethod} and \cite{bumesmes20} for the continuous time setting. In these references, the authors demonstrated convergence results for standard LQ problems.

\vspace{-2mm} 

\paragraph{Outline.} The paper is organized as follows. In Section \ref{sec:problem:formulation}, we formulate the exploratory LQ MFC problem in continuous time with common noise, provide the theoretical optimal policy, and discuss parameterization in the model-free case. In Section \ref{sec:model:based:pg}, we demonstrate the convergence of the model-based gradient descent algorithm using the gradient domination condition. Section \ref{sec:model:free} presents the gradient estimation algorithm employing the population simulator and develops the convergence analysis of the model-free gradient descent algorithm. In Section \ref{sec:numerical:example}, we provide numerical experiments that illustrate our convergence results for both the model-based and model-free algorithms. The proofs of all results are included in the Appendix.

\paragraph{Notations.} 
\begin{itemize}
\item We denote by $x\cdot y$ the scalar product between the two vectors $x$, $y$, and by $M:N$ $=$ 
${\rm tr}(M N\trans)$ the inner product of the two matrices $M,N$ with compatible dimensions, where $N\trans$ is the transpose matrix of $N$. The Frobenius norm of a matrix $A$ is defined by 
$\lVert A\rVert_F:=\sqrt{A:A}$. 
\item $\S^d$ is the set of symmetric $\d\times d$ matrices, and $\S^d_+$ (resp. $\S^d_{>+})$ is the set of 
nonnegative (resp. positive definite)  matrices in $\S^d$. The partial order $\geq$ on $\S^d$ is defined as: $M$ $\geq$ $N$ if $M-N$ $\in$ $\S^d_+$. We also write $M$ $>$ $0$ to mean that $M$ $\in$ $\S^d_{>+}$. 
\end{itemize}

\section{Problem formulation}\label{sec:problem:formulation}

\subsection{Setup and preliminaries}

The linear dynamics of the mean-field state equation with randomized controls and common noise 
is described by 
\begin{align} \label{dynX}
\d X_t &= \;  \Big[ BX_t + \bar B \E_0[X_t] + D \int a \bpi_t(\d a)\Big] \dt    + \gamma \d W_t + \gamma_0 \d W_t^0, 
\end{align}
on a probability space $(\Omega,\Fc,\P)$ supporting two independent Brownian motions $W$ (the \textit{idiosyn\-cratic noise}), and $W^0$ (the \textit{common noise}), of dimension $d$ and $d_0$. For convenience, we choose the probability space in the product form $(\Omega^0\times\Omega^1,\Fc^0\otimes\Fc^1,\P^0\otimes\P^1)$, and denote by 
$\F^1=(\Fc^1_t)_{t\geq 0}$ the right-continuous $\P^1-$completion of the canonical filtration generated by $W$, and 
by $\F^0=(\Fc^0_t)_{t\geq 0}$ the  right-continuous $\P^0-$completion of the canonical filtration  generated by $W^0$. The initial condition $X_0$ is an $\R^d$-random variable, which is $\Gc$-measurable, where $\Gc$ is a $\sigma$-algebra independent of $(W,W^0)$. We denote by $\F$ $=$ $\F^0\vee\F^1\vee\Gc$.

Here, $\E_0[.]$ stands for the conditional expectation given $\F^0$, $B$, $\bar B$ are constant matrices in $\R^{d\times d}$, $D$ is a constant matrix in $\R^{d\times m}$, $\gamma$ is a constant matrix in $\R^{d\times d}$, 
$\gamma^0$ is a constant matrix in $\R^{d\times d_0}$, and in the sequel, we shall denote by $\hat B$ $:=$ $B+\bar B$. 

The randomized control $\bpi$ $=$ $(\bpi_t)_t$ is an $\F$-progressively measurable process in  $\Pc_2(\R^m)$, the set of probability measures on the action space $A$ $=$ $\R^m$ with a finite second order moment. We shall consider randomized controls $\bpi$ with densities 
$a$ $\mapsto$ $\bop_t(a)$, $t$ $\geq$ $0$.

The infinite horizon LQ MFC problem consists in minimizing over such randomized controls $\pi$ the quadratic cost functional with entropy regularizer of parameter $\lambda$ $>$ $0$:  
\begin{align}
J(\bpi;\lambda) &= \;  \E \Big[ \int_0^\infty e^{-\beta t} \Big( X_t\trans Q X_t + 
\E_0[X_t]\trans\bar Q \E_0[X_t]  \\
& \qquad \qquad + \;  \int a\trans R a \; \bpi_t(\d a) + \lambda \int \log \bop_t(a) \bpi_t(\d a) \Big) \d t    \Big].   
\end{align}
Here $Q$, $\bar Q$ are constant matrices in $\S^d$ such that $Q$  $>$ $0$, $\hat{Q} := Q+\bar Q$ $>$ $0$ 
and $R$ $\in$ $\S^m_{>+}$.  Notice that the cost functional is written equivalently as 
\begin{align}\label{defJ'}
J(\bpi;\lambda) &= \; \E\Big[ \int_0^\infty e^{-\beta t} 
\Big(  (X_t-\E_0[X_t])\trans Q (X_t-\E_0[X_t])+\E_0[X_t]\trans\hat Q \E_0[X_t] \\
& \qquad \qquad  + \; \int a\trans R a \;\bpi_t(\d a) + \lambda \int \log \bop_t(a) \bpi_t(\d a) \Big) \d t \Big]. 
\end{align}

\vspace{1mm}


\begin{Assumption}\label{AssumForKLambda}
    The two following Algebraic Riccati Equations (ARE) for $K\in\S^d$ and $\Lambda\in\S^d$ respectively admit a unique positive definite solutions:  
\begin{align} 
    -\beta K+KB+B\trans K+Q-K D R^{-1}D\trans K &= \; 0,  \label{AREForK} \\
   -\beta \Lambda+\Lambda\hat B+\hat B\trans \Lambda+\hat Q-\Lambda D R^{-1}D\trans \Lambda &= \; 0. \label{AREForLamb}
\end{align}
\end{Assumption}
\begin{Remark}
    According to Section 6 in \cite{MartingApproLQ}, the condition: $Q> 0,\hat Q=Q+\bar Q> 0,R>0$,  guarantees the existence of positive definite solutions to \eqref{AREForK}-\eqref{AREForLamb} and Assumption \ref{AssumForKLambda} is specifically for ensuring uniqueness.
\end{Remark}


Under Assumption \ref{AssumForKLambda},  the optimal randomized control, solution to \eqref{defJ'} is given in feedback policy form as $\bpi_t^*$ $=$ $\pi^*(.|X_t^*-\E_0[X_t^*],\E_0[X_t^*])$, 
where $\pi^*(.|y,z)$ is the normal distribution 
\begin{align} \label{piopt} 
\pi^*(.|y,z) &= \;  \Nc \Big( - R^{-1}D\trans  K y -R^{-1}D\trans \Lambda z ; \frac{\lambda}{2} R^{-1}  \Big), 
\quad y,z \in \R^d, 
\end{align}
 and where $(X^{*}_t)_{t\in[0,T]}$ is the state process with randomized control $\bpi^*$ and $(K,\Lambda)$ is the unique positive definite solution to \eqref{AREForK}-\eqref{AREForLamb}. 
Moreover, the optimal cost is given by 
\begin{align} 
J(\bpi^*;\lambda) 
& = \; K:M +  \Lambda : \hat M + \upsilon(\lambda)  \label{OptCostInfHor2}
\end{align}
where 
\begin{equation}\label{defupsilonlambda}
\begin{aligned} 
M & := \; {\rm Var}(X_0) + \frac{1}{\beta}\gamma\gamma\trans, \quad
\hat M \; := \; \E[X_0]\E[X_0]\trans +\frac{1}{\beta}\gamma_0\gamma_0\trans, \\
\upsilon(\lambda) & := \; \frac{1}{\beta}\Big(   - \frac{\lambda m}{2} \log (\pi\lambda) + \frac{\lambda}{2} \log\big| {\rm det}(R) \big| \Big).   
\end{aligned}
\end{equation}
Here, ${\rm Var}(X_0)$ denotes the covariance matrix of $X_0$. 
In Appendix \ref{ProofOptSol}, we state and prove this result in a more general case.

\vspace{2mm}

Throughout the remainder of this paper, we shall also assume that the following assumption is in force.

\begin{Assumption}\label{AssumForM}
$M,\hat M$ belongs to $\S^d_{>+}$, i.e. $\sigma_{\min}(M),\sigma_{\min}(\hat M)>0,$ where $\sigma_{\min}(\cdot)$ denotes the smallest eigenvalue of a square matrix.
\end{Assumption}

\begin{Remark}
The above assumption for $M$ is satisfied when ${\rm Var}(X_0)> 0$. As for $\hat M$, denoting by $\bar \gamma_0=[\E[X_0],\frac{1}{\sqrt{\beta}}\gamma_0]\in\R^{d\times (d_0+1)}$, the augmented matrix formed by $\E[X_0]$ and $\gamma_0$, we have $\hat M=\bar\gamma_0\bar\gamma_0\trans$. Thus, $\hat M> 0$ if and only if $\bar \gamma_0$ is of rank $d$, which requires that that  $d_0\geq d-1$.
\end{Remark}

\subsection{Model free perspective and reparametrization}

In this section, we are interested in the  model-free setting for the linear mean-field dynamics of state process, i.e., when the parameters $B,\bar B,D,\gamma,\gamma_0$ in \eqref{dynX} are unknown, and so the optimal policy in \eqref{piopt} cannot be implemented from the unique solution to the Riccati system \eqref{AREForK}-\eqref{AREForLamb}.  

Motivated by the Gaussian distribution of the optimal randomized policy, whose mean is a linear combination of $Y_t = X_t - \E_0[X_t]$ and $Z_t = \E_0[X_t]$, we propose the following parameterization of the randomized policy:   
\begin{align}
 \pi^{\Theta}(\cdot|y,z) &= \; \mathcal{N}\big(\theta y +\zeta z;\frac{\lambda}{2}R^{-1} \big), \quad y,z \in \R^d,  
\end{align}
where $\Theta = (\theta,\zeta) \in (\R^{m\times d})^2$ are the two-parameter matrices to be optimized. The density of the parametrized randomized policy
is explicitly given by 
\begin{align*}
a \in \R^m & \longmapsto    p^\Theta(y,z,a)
\;  = \; \sqrt{\frac{\det(R)}{(\pi\lambda)^m}}\exp\Big(-\frac{1}{\lambda} \big(a-(\theta y+\zeta z)\big)\trans R 
\big(a-(\theta y+\zeta z)\big) \Big). 
\end{align*}
The associated dynamics of the parametrized process $(X^\Theta_t)$, starting from $X_0^\Theta$ $=$ $X_0$, is given by
\begin{align}\label{DynXtheta}
\d X^{\Theta}_t & = \; 
\Big[ BX^\Theta_t + \bar B \E_0[X^\Theta_t] + D \int a \bpi_t^\Theta(\d a) \Big] \dt    + \gamma \d W_t + \gamma_0 \d W_t^0,
\end{align}
where $\bpi_t^\Theta$ $=$ $\pi^\Theta(\cdot|X^\Theta_t-\E_0[X^\Theta_t],\E_0[X^\Theta_t])$, 
with density $\bop_t^\Theta(a)$ $=$ $p^\Theta(X^\Theta_t-\E_0[X^\Theta_t],\E_0[X^\Theta_t],a)$.  
The corresponding cost function, now defined as a function on $(\R^{m \times d})^2$, is expressed (with a slight abuse of notation) as follows:
\begin{align}
J(\theta,\zeta;\lambda) & : = \;  J(\bpi^\Theta;\lambda) \\
& = \; \E_{} \Big[ \int_0^\infty e^{-\beta t} \big(  (X^\Theta_t - \E_0[X^\Theta_t])\trans Q(X^\Theta_t - \E_0[X^\Theta_t])+\E_0[X^\Theta_t]\trans \hat Q\E_0[X^\Theta_t]  \nonumber \\
& \qquad \qquad  + \; \int a\trans R a \;\bpi^\Theta_t(\d a) 
+ \lambda \int \log \bop^\Theta_t(a) \bpi^\Theta_t(\d a) \Big) \d t \Big] \label{defJtheta}
\end{align}
where $\hat{Q} := Q + \bar{Q}$, with the objective of minimizing the parameterized cost function $J(\theta, \zeta; \lambda)$ over $\Theta = (\theta, \zeta) \in (\R^{m \times d})^2$. It is clear from \eqref{piopt} that $\Inf_{\bpi}J(\bpi;\lambda)$ $=$ 
$\Inf_{\Theta} J(\theta,\zeta;\lambda)$, and the solution to this minimization problem is given by 
\begin{align} \label{opttheta}
\theta^* \; = \; -R^{-1} D\trans K, & \quad \zeta^* \; = \; -R^{-1}D\trans \Lambda,
\end{align} 
 where $(K,\Lambda)$ $\in$ $(\S^d_{>+})^2$ is the solution to the ARE \eqref{AREForK} and \eqref{AREForLamb}, ensuring that $J(\theta^*,\zeta^*;\lambda)$ $=$ $J(\bpi^*;\lambda)$.

\vspace{3mm}

Let us introduce the parametrized auxiliary processes
\begin{align}
Y^\Theta_t \; = \; X^\Theta_t-\E_0[X^\Theta_t], & \quad  Z^\Theta_t=\E_0[X^\Theta_t], \quad t \geq 0, \quad 
\Theta = (\theta,\zeta) \in (\R^{m\times d})^2.     
\end{align}
By observing from the definition of $\bpi_t^\Theta$ $=$ $\pi^\Theta(\cdot|Y_t^\Theta,Z_t^\Theta)$ that 
\begin{align}
\int a \bpi_t^\Theta(\d a) & = \; \theta Y_t^\Theta + \zeta Z_t^\Theta, \quad 
\E_0\Big[ \int a \bpi_t^\Theta(\d a) \Big] \; = \; \zeta Z_t^\Theta, 
\end{align}
 we see from \eqref{DynXtheta} that the dynamics of $(Y^\Theta,Z^\Theta)$ is decoupled and governed by 
\begin{equation}\label{dynamics:ytheta:ztheta}
\begin{aligned}
\begin{cases}
\d Y_t^\Theta &= \; (B+D\theta) Y_t^\Theta \dt + \gamma \d W_t,  \\
\d Z_t^\Theta &= \; (\hat B + D \zeta) Z_t^\Theta \ dt + \gamma^0 \d W_t^0,  
\end{cases}
\end{aligned} \quad t\geq0
\end{equation}
\noindent with $Y_0^\Theta := Y_0 =  
X_0 - \E[X_0]$ and $ Z_0^\Theta := Z_0 = \E[X_0]$. Notice that $Y^\Theta$ (resp. $Z^\Theta$) depends on $\Theta$ only via the first (resp. second) 
component parameter $\theta$ (resp. $\zeta$), and we shall then write $Y^\theta$ $=$ $Y^\Theta$, $Z^\zeta$ $=$ 
$Z^\Theta$. 

Moreover, by noting that 
\begin{align}
\int a\trans R a \, \bpi_t^\Theta(\d a) 
& = \; \theta Y_t^\theta + \zeta (Z_t^\zeta)\trans R \theta Y_t^\theta + \zeta Z_t^\zeta + \frac{\lambda m}{2} \\
\int \log \bop_t^\Theta(a) \, \bpi_t(d a) &= \;  -\frac{m}{2}(1+\log(2\pi))-\frac{m}{2}\log\lvert\frac{\lambda}{2\det(R)}\rvert\; 
\end{align}
the parametrized cost in \eqref{defJtheta} can be  written as a quadratic function of $(Y_t^\theta,Z_t^\zeta)$, namely:
\begin{align}
J(\theta,\zeta;\lambda) &= \;   
\E\Big[ \int_0^\infty e^{-\beta t} \Big(  (Y^\theta_t)\trans (Q+\theta\trans R\theta)Y^\theta_t 
+ (Z^\zeta_t)\trans(\hat Q+\zeta\trans R\zeta) Z^\zeta_t \Big)  \d t \Big] +\upsilon(\lambda), 
\end{align}

\noindent recalling that $v(\lambda)$ is given by \eqref{defupsilonlambda}.
An important observation is that it can be decomposed as follows
\begin{align}
J(\theta,\zeta;\lambda) &= \; J_1(\theta) + J_2(\zeta) + \upsilon(\lambda),       
\end{align}
where 
\begin{align} \label{defJ12} 
J_1(\theta) \; := \;   \big( Q+\theta\trans R\theta \big) : \Sigma_\theta,&  \quad 
J_2(\zeta) \; := \; \big( \hat Q+\zeta\trans R\zeta \big) : \hat\Sigma_\zeta, 
\end{align}
with 
\begin{align} \label{defSig} 
\Sigma_\theta \; := \; \int_0^\infty e^{-\beta t}\E\big[Y^\theta_t(Y^\theta_t)\trans \big]\d t, & \quad  
\hat\Sigma_\zeta \; := \; \int_0^\infty e^{-\beta t}\E\big[Z^\zeta_t(Z^\zeta_t)\trans \big]\d t,      
\end{align}
so that the minimization over $\Theta$ $=$ $(\theta,\zeta)$ amounts to separate minimization problems: 
\begin{align} \label{minJ12}
\inf_{\theta \in \R^{m\times d}} J_1(\theta), &  \qquad \mbox{ and } \qquad \inf_{\zeta \in \R^{m\times d}} J_2(\zeta). 
\end{align}

In fact, we now demonstrate that the two aforementioned minimization problems over $\theta$ and $\zeta$ can be reduced to suitable smaller sets defined by
\begin{align}
\Sc \; = \; \Big\{ \theta\in\R^{m\times d}:B-\frac{\beta}{2}I_d+D\theta\text{ is stable}\Big\}, & \;\;  
\hat\Sc \; = \; \Big\{\zeta\in\R^{m\times d}: \hat B-\frac{\beta}{2}I_d+D\zeta\text{ is stable}\Big\}, 
\end{align}
recalling that $\hat{B} = B + \bar{B}$, on which the problem is well-posed.

\begin{Lemma}\label{OptInS}
The minimizers in \eqref{opttheta} 
satisfy $\theta^*\in\Sc$, $\zeta^*\in\hat\Sc$,  and thus 
\begin{align}
\inf_{\theta \in \R^{m\times d}} J_1(\theta) \; = \; \inf_{\theta \in \Sc}  J_1(\theta), &
\qquad \mbox{ and } \qquad
\inf_{\zeta \in \R^{m\times d}} J_2(\theta) \; = \; \inf_{\zeta \in \hat\Sc}  J_2(\theta). 
\end{align}
\end{Lemma}
\begin{proof} 
\noindent The ARE  \eqref{AREForK} for $K$ rewrites as 
\begin{align} \label{Kinter}
 (B-\frac{\beta}{2}I_d-DR^{-1}D\trans K)\trans K+K(B-\frac{\beta}{2}I_d-DR^{-1}D\trans K) 
 &=  -(Q+KDR^{-1}D\trans K). 
\end{align} 
 Since $Q> 0$ then $Q+KDR^{-1}D\trans K > 0$ , 
 hence all the eigenvalues of the matrix on the r.h.s. of \eqref{Kinter} are strictly negative.
Moreover, since $K > 0$, if the matrix $(B-\frac{\beta}{2}I_d-DR^{-1}D\trans K)$ has a non-negative real part eigenvalue, then the largest eigenvalue of the matrix on the l.h.s. of \eqref{Kinter} is non-negative and therefore 
cannot be equal to the matrix on the r.h.s. It follows that $B-\frac{\beta}{2}I_d-DR^{-1}D\trans K$ is stable, which means tht  $\theta^*\in\Sc$.

Similarly, by using the ARE  \eqref{AREForLamb} for $\Lambda$, we show that 
$\hat B-\frac{\beta}{2}I_d-DR^{-1}D\trans \Lambda$ is also stable, i.e.,  $\zeta^*\in\hat\Sc$
\end{proof}

 \vspace{3mm}

We now state a characterization of the two matrices $\Sigma_\theta$ and $\hat\Sigma_\zeta$ defined in \eqref{defSig}.

\begin{Proposition}\label{PropSigmaTheta}
  For all $\theta\in\Sc,\zeta\in\hat\Sc$, the matrices $\Sigma_\theta$ and $\hat\Sigma_\zeta$ are well-defined in $\S^d$ and are the unique positive definite solution to the following Algebraic Lyapunov Equations (ALE):
\begin{equation}\label{ALEForSigmaTheta}
\begin{aligned} 
-\beta \Sigma_\theta+(B+D\theta)\Sigma_\theta+\Sigma_\theta(B+D\theta)\trans+M &= \; 0, \\
-\beta \hat\Sigma_\zeta+(\hat B+D\zeta)\hat\Sigma_\zeta+\hat\Sigma_\zeta(\hat B+D\zeta)\trans+\hat M &= \; 0. 
\end{aligned}
\end{equation}
\end{Proposition}
\begin{proof}
    Cf Appendix \ref{ProofSigmaTheta}
\end{proof}

\vspace{2mm}

We conclude this section by providing an alternate useful expression of the cost functions $J_1$ and $J_2$ defined by 
\eqref{defJ12} on which the minimization over $\theta$ $\in$ $\Sc$ and $\zeta$ $\in$ $\hat\Sc$ will be performed.

\begin{Proposition}\label{PropValueFuncForKtheta}
For all $\theta\in\Sc,\zeta\in\hat{\Sc}$, we have 
\begin{align}
J_1(\theta) \; = \;  
K_\theta : M, & \qquad 
J_2(\zeta) \; = \;  
\Lambda_\zeta : \hat M,
\end{align}
where $K_\theta$ and $\Lambda_\zeta$ are the unique elements in $\S^d_{>+}$ solutions to the ALE 
\begin{align} 
-\beta K_\theta+(B+D\theta)\trans K_\theta+K_\theta(B+D\theta)+Q+\theta\trans R\theta &= \; 0 \label{ALEForKtheta} \\    -\beta \Lambda_\zeta+(\hat B+D\zeta)\trans \Lambda_\zeta+\Lambda_\zeta(\hat B+D\zeta)+\hat Q+\zeta\trans R\zeta 
    & = \; 0. \label{ALEForLambdatheta}
\end{align}
\end{Proposition}
\begin{proof}
    Cf Appendix \ref{ProofKTheta}
\end{proof}

\begin{Remark}
    Note that, according to \eqref{OptCostInfHor2}, we have $K_{\theta^*}=K$ and $\Lambda_{\zeta^*}=\Lambda$.
\end{Remark}

\section{Model-based PG algorithm}\label{sec:model:based:pg}

In this section, we establish the convergence of the gradient descent algorithm towards the optimal parameters $(\theta^*,\zeta^*)$ assuming full knowledge of the model parameters, thus allowing for the exact computation of the gradient.  
This foundation will facilitate the learning of optimal parameters in the model-free setting in the subsequent section, despite the nonconvex optimization framework. Furthermore, the gradient descent method in the model-based scenario offers an alternative to solving the Riccati system \eqref{AREForK}-\eqref{AREForLamb}, which is computationally intensive, while the convergence result (Theorem \ref{ConvExGradDes}) remains dimension-free.

Given $(\theta_{[0]}, \zeta_{[0]})\in \Sc \times \hat\Sc $, the updating rule at step $k\geq 0$ for the exact gradient descent of the minimization problem \eqref{minJ12} is given by 
 \begin{equation}\label{exact:GD:algorithm}
\begin{cases}
\theta_{[k+1]} &= \; \theta_{[k]}-\rho \nabla J_1(\theta_{[k]}),  \\
\zeta_{[k+1]} &= \; \zeta_{[k]}-\rho \nabla J_2(\zeta_{[k]}),
\end{cases}
\end{equation}
where $\rho>0$ is the constant step size (learning rate), and $\nabla J_1(\theta)$ and $\nabla J_2(\zeta)$ are the gradients of $J_1$ and $J_2$ with respect to their respective parameters.

\subsection{Gradient domination condition}

We first provide an explicit formula for the gradient that will be used for the implementation of the 
(exact) gradient descent rule.

\begin{Proposition}[Expression of the gradients]\label{GradJ}
For $\theta\in\Sc, \, \zeta\in\hat\Sc$, it holds
\begin{align}
\nabla J_1(\theta) \; = \;  2E_{\theta}\Sigma_{\theta}, & \quad   \mbox{ and } \quad
\nabla J_2(\zeta) \; = \;   2\hat{E}_{\zeta}\hat\Sigma_{\zeta}, 
    \end{align}
where $E_\theta= R\theta+D\trans K_{\theta}$ and $\hat E_\zeta= R\zeta+D\trans \Lambda_{\zeta}$.
\end{Proposition}
\begin{proof}
    Cf. Appendix \ref{ProofGradE}
\end{proof}

Next, to assess the convergence of the aforementioned gradient descent algorithm in this nonconvex landscape, we will demonstrate that the cost functions $J_1$ and $J_2$ satisfy a Polyak-Lojasiewicz (PL) inequality on $\mathcal{S}$ and $\hat{\mathcal{S}}$, respectively, also known as the gradient domination condition.

\begin{Proposition}[Gradient domination] \label{GradDomCond}
There exist two positive constants $\kappa_1,\kappa_2>0$ such that for any $\theta\in\Sc,\zeta\in\hat\Sc$, 
\begin{align}
J_1(\theta)-J_1(\theta^*) \; \leq \;  \kappa_1\lVert\nabla J_1(\theta)\rVert^2_F, & \quad \mbox{ and } \quad  
J_2(\zeta)-J_2(\zeta^*) \; \leq \;  \kappa_2\lVert\nabla J_2(\zeta)\rVert^2_F. 
\end{align}
More precisely, the constants $\kappa_1$ and $\kappa_2$ are given by   
\begin{align} 
\kappa_1 \;= \;  \frac{\lVert \Sigma_{\theta^*}\rVert_F  }{4\sigma_{\text{min}}(R)\sigma_{\text{min}}^2(M)},& \qquad 
\kappa_2 \; = \; \frac{\lVert \hat\Sigma_{\zeta^*}\rVert_F  }{4\sigma_{\text{min}}(R)\sigma_{\text{min}}^2(\hat M)}. 
\end{align} 
\end{Proposition}
\begin{proof}
    Cf. Appendix \ref{ProofPL}
\end{proof}


\subsection{Global convergence of model-based PG algorithm}

 For fixed $\ell,  \hat \ell \in \R_+$, we define the level subsets of $\Sc$ and $\hat \Sc$ as follows:
\begin{align} \label{defSa}
\Sc(\ell) \; = \; \Big\{\theta\in\Sc: \;  J_1(\theta)\leq \ell\Big\} & \quad \mbox{ and } \quad  
\hat \Sc(\hat \ell) \; = \;  \Big\{\zeta\in\Sc: \;  J_2(\zeta)\leq \hat \ell\Big\}. 
\end{align}
In the sequel, we shall naturally restrict to  $\ell>J_1(\theta^*)$ and $\hat \ell>J_2(\zeta^*)$ to avoid empty level sets.

\vspace{2mm}

We now show that both gradient maps $\theta\mapsto\nabla J_1(\theta)$ and $\zeta\mapsto\nabla J_2(\zeta)$ are Lipschitz-continuous on $\Sc(a)$ and $\hat\Sc(\hat a)$ respectively. 

\begin{Proposition}\label{LJ}
There exist explicit positive constants 
\begin{align} \label{La}
L( \ell) &= \; L\big(\ell;\sigma_{\min}(Q),\sigma_{\min}(M),\lVert R\rVert_F,\lVert D\rVert_F,J_1(\theta^*)\big), \\
\hat L(\hat \ell) &= \; \hat L\big(\hat \ell;\sigma_{\min}(\hat Q),\sigma_{\min}(\hat M),\lVert R\rVert_F,\lVert D\rVert_F,J_2(\zeta^*)\big),
\end{align}
such that $\theta\mapsto\nabla J_1(\theta)$ is $L(\ell)$-Lipschitz continuous on $\Sc(\ell)$ 
and $\zeta\mapsto\nabla J_2(\zeta)$ is $\hat L(\hat \ell)$-Lipschitz continuous on $\hat\Sc(\hat \ell)$, that is, for all $\theta,\theta'\in\Sc(\ell)$ and all $\zeta,\zeta'\in\hat\Sc(\hat \ell)$
\begin{align}
\lVert \nabla J_1(\theta')-\nabla J_1(\theta)\rVert_F \; \leq \;  L(\ell)\lVert \theta'-\theta\rVert_F, & \qquad 
\lVert \nabla J_2(\zeta')-\nabla J_2(\zeta)\rVert_F \;  \leq \;  \hat L(\hat \ell)\lVert \zeta'-\zeta\rVert_F.          
\end{align}
\end{Proposition}
\begin{proof}
Cf. Appendix \ref{ProofLip}
\end{proof}

\vspace{2mm}

The following theorem is the main result of this section and states the linear convergence rate of the gradient descent method. A key point is to prove that the sequence $(\theta_{[k]}, \zeta_{[k]})_{k\geq 0}$ always lies in $\Sc(\ell) \times \hat\Sc(\hat \ell)$ provided that one starts from an initial point $(\theta_{[0]},\zeta_{[0]}) \in\Sc(\ell)\times \hat\Sc(\hat \ell)$ with a suitably chosen common step size $\rho_\theta=\rho_\zeta:=\rho>0$.

\begin{Theorem}[Global convergence of the exact gradient descent method] \label{ConvExGradDes}
Let $(\theta_{[0]},\zeta_{[0]})$ $\in$ $\Sc(\ell)\times\hat\Sc(\hat \ell)$ and select a constant step size $ \rho \in\big(0,\frac{2}{\check{L}(\ell,\hat \ell)}\big)$ where 
\begin{align}\label{checkLa}
\check{L}(\ell,\hat \ell)& = \; \max\big(L(\ell),\hat L(\hat \ell)\big).
\end{align}
Then, the sequence $(\theta_{[k]},\zeta_{[k]})_{k\geq 0}$ generated by the exact GD algorithm \eqref{exact:GD:algorithm} stays in $ \Sc(\ell)\times\hat\Sc(\hat \ell)$. Moreover, for any fixed accuracy $\varepsilon>0$, we achieve  
\begin{align}  
J(\theta_{[k]},\zeta_{[k]};\lambda)-  J(\theta^*,\zeta^*;\lambda) &\leq \;  \varepsilon
\end{align} 
with a number of iterations satisfying
\begin{align} 
k & \geq \;  \frac{\log\big(\frac{\varepsilon}{ J(\theta_{[0]},\zeta_{[0]})-  J(\theta^*,\zeta^*)}\big)}{\log\big(1-\frac{\rho(2-\rho \check{L}(\ell,\hat \ell))}{ 2\kappa(\ell,\hat \ell)}\big)},
\end{align} 
where 
\begin{align}
 \kappa(\ell,\hat \ell) &= \; \max\Big(\kappa_1,\kappa_2,\frac{2}{\check{L}(\ell,\hat \ell)}\Big)+\frac{1}{2}.       
\end{align}
\end{Theorem}

 \begin{Remark}
In the model-based case, since we can compute $J_1$ and $J_2$ explicitly, one can choose $a$ $=$ $J_1(\theta_{[0]})$ and 
$\hat a$ $=$ $J_2(\zeta_{[0]})$ for any $(\theta_{[0]}, \zeta_{[0]})\in \Sc \times \hat\Sc$. 
 \end{Remark}

\vspace{2mm}

The proof of the above convergence result is based on the following stability result which shows that both $\Sc$ and $\hat \Sc$ are stable by the one step transition of the GD algorithm.

\begin{Lemma}\label{PerbEGD}
If $\theta\in \Sc(\ell),\theta\neq\theta^*$ and $\zeta\in\hat\Sc(\hat \ell),\zeta\neq\zeta^*$,  
then  for all $\rho\in(0,\frac{2}{\check L(\ell,\hat \ell)})=(0,\min(\frac{2}{L(\ell)},\frac{2}{\hat L(\hat \ell)}))$, it holds 
\begin{align}
\theta_{\rho} \; := \; \theta-\rho\nabla  J_1(\theta) \; \in \;  \Sc(\ell), & \qquad     
\zeta_{\rho} \; := \; \zeta- \rho\nabla J_2(\zeta) \; \in \; \hat \Sc(\hat \ell).
\end{align}
\end{Lemma}
\begin{proof}
For a fixed $\theta\in\Sc(\ell)$, we let 
\begin{align} 
\rho_{\max} &= \; \sup\big\{\rho'\geq 0:  \; \theta-\rho\nabla J_1(\theta) \in\Sc(\ell),\forall \rho\in[0,\rho']\big\}. 
\end{align} 
From the first order Taylor expansion and the continuity of $\nabla J_1$, one has
$$
J_1(\theta_\rho)=J_1(\theta)-\rho\lVert \nabla J_1(\theta)\rVert_F^2+o(\rho).
$$
 Since $\theta$ $\neq$ $\theta^*$, we have $\lVert \nabla J_1(\theta)\rVert_F^2>0$ and then for all $\rho$ small enough, we have $J_1(\theta_\rho)\leq J_1(\theta)\leq \ell$ which implies that $\theta_\rho\in\Sc(\ell)$ and thus $\rho_{\max}>0$. Moreover, from Proposition \ref{Bds} where it is shown that $\Sc(\ell)$ is bounded, we deduce that $\rho_{\max}<+\infty$.

The definition of $\rho_{\max}$ implies that for any $\rho\in(0,\rho_{max}]$ and any $t\in [0,1]$, $t\theta_\rho+(1-t)\theta=\theta-(t\rho)\nabla J_1(\theta) = \theta_{t \rho} \in \Sc(\ell)$ since $t\rho\in (0,\rho_{\max}]$.

Now the second-order Taylor's expansion for $J_1(\theta)$ combined with the local Lipschitz continuity of $\nabla J_1$ stated in Proposition \ref{LJ} guarantees that
\begin{align} \label{interJ1}
J_1(\theta_\rho) & \leq \;  J_1(\theta)+\langle\nabla J_1(\theta),\theta_\rho-\theta\rangle+\frac{L(a)}{2} 
\lVert \theta_\rho-\theta\rVert_F^2. 
\end{align}
Suppose that $\rho_{\max}<\frac{2}{L(\ell)}$. From the continuity of $\rho\mapsto J_1(\theta-\rho \nabla J_1(\theta))$, and the definition of $\Sc(\ell)$, we have $ J_1\big(\theta-\rho_{\max}\nabla J_1(\theta)\big)=\ell$.  Then,    by \eqref{interJ1}, we get 
\begin{align} \label{interJ11}
J_1(\theta-\rho\nabla J_1(\theta)) &\leq \;   J_1(\theta)-\rho\lVert \nabla  J_1(\theta)\rVert_F^2+\frac{\rho^2 L(a)}{2}\lVert \nabla  J_1(\theta)\rVert_F^2, 
\end{align} 
which implies (as $2 -\rho_{max}L(\ell)$ $>$ $0$) that 
\begin{align}        
J_1\big(\theta-\rho_{\max}\nabla J_1(\theta)\big) &< \;  J_1(\theta) \; \leq \;  \ell. 
 \end{align} 
This contradicts the fact that $J_1\big(\theta-\rho_{\max}\nabla J_1(\theta)\big)=\ell$. Therefore, we conclude that $\rho_{\max}\geq\frac{2}{L(\ell)}$ and for all $\rho\in(0,\frac{2}{L(\ell)})$, $\theta_\rho=\theta-\rho\nabla J_1(\theta)\in\Sc(\ell)$. Similar arguments show the result for $\zeta_\rho$. In particular, one shows that for all $\rho\in(0,\frac{2}{\hat L(\hat \ell)})$, $\zeta_\rho\in\hat S(\hat \ell)$. In conclusion, for all $\rho\in (0,\min(\frac{2}{L(\ell)},\frac{2}{\hat L(\hat a)}))$,  it holds  at the same time
\begin{align}
\theta_{\rho} \; := \; \theta-\rho\nabla  J_1(\theta) \; \in \;  \Sc(\ell) & \quad \mbox{ and } \quad      
\zeta_{\rho} \; := \; \zeta- \rho\nabla J_2(\zeta) \; \in \; \hat \Sc(\hat \ell).
\end{align}
\end{proof}

\vspace{3mm}

\noindent {\bf  Proof of Theorem \ref{ConvExGradDes}.} 
Start from some fixed stabilizing initial parameter $(\theta_{[0]},\zeta_{[0]})\in \Sc(\ell) \times \hat{\Sc}(\hat \ell)$. 
Since $\check L(\ell,\hat \ell)=\max(L(\ell),\hat L(\hat \ell))$, by choosing a step size $\rho\in(0,\frac{2}{\check L(\ell,\hat \ell)})$, Lemma \ref{PerbEGD} guarantees that 
\begin{align} 
\theta_{[1]} \; = \; \theta_{[0]}-\rho\nabla J_1(\theta_{[0]}) \; \in \; \Sc(\ell), & \quad \mbox{ and } \quad 
\zeta_{[1]} \; = \; \zeta_{[0]}-\rho\nabla J_2(\zeta_{[0]}) \; \in \;  \hat\Sc(\hat \ell).
\end{align} 

 Moreover, by \eqref{interJ11}, the following inequalities hold
\begin{equation} \label{Jtheta}
\begin{cases}
J_1(\theta_{[1]})- J_1(\theta_{[0]})
& \leq \;  \frac{-\rho(2-\rho \check L(\ell,\hat \ell))}{2}\lVert \nabla  J_1(\theta_{[0]})\rVert_F^2 \\
J_2(\zeta_{[1]})- J_2(\zeta_{[0]})
& \leq \;  \frac{-\rho(2-\rho \check L(\ell,\hat \ell))}{2}\lVert \nabla  J_2(\zeta_{[0]})\rVert_F^2. 
\end{cases}
\end{equation} 
The definition of $\kappa(\ell,\hat \ell)$ together with the gradient domination condition of Proposition \ref{GradDomCond}, ensures that for all $\theta\in \Sc(\ell)$, $\zeta\in\hat{\Sc}(\hat \ell)$ 
\begin{align}
J_1(\theta)-J_1(\theta^*)
\; \leq \;  \kappa(\ell,\hat \ell)\lVert \nabla J_1(\theta)\rVert_F^2, & \qquad 
J_2(\zeta)-J_1(\zeta^*)
\; \leq \;  \kappa(\ell,\hat \ell)\lVert \nabla J_2(\zeta)\rVert_F^2, 
\end{align}
and that for all $\rho\in \big(0,\frac{2}{\check L (\ell,\hat \ell)}\big)$  
\begin{align} 
\frac{\rho(2-\rho \check L(\ell,\hat \ell))}{2 \kappa(\ell,\hat \ell)}\in(0,1).
\end{align} 

Hence, coming back to \eqref{Jtheta}, we deduce
\begin{align} 
J_1(\theta_{[1]})- J_1(\theta_{[0]}) 
\; \leq \;  \frac{-\rho(2-\rho \check L(\ell,\hat \ell))}{2\kappa(\ell,\hat \ell)} 
\big(J_1(\theta_{[0]})-J_1(\theta^*) \big), \\
J_2(\zeta_{[1]})- J_2(\zeta_{[0]}) 
\; \leq \;  \frac{-\rho(2-\rho \check L(\ell,\hat \ell))}{2\kappa(\ell,\hat \ell)}\big(J_2(\zeta_{[0]})-J_2(\zeta^*)\big). 
\end{align} 
Summing the two previous inequalities, we obtain
\begin{align*}
    J(\theta_{[1]},\zeta_{[1]};\lambda)-J(\theta_{[0]},\zeta_{[0]};\lambda)&= \; \big(J_1(\theta_{[1]})- J_1(\theta_{[0]})\big)+\big(J_2(\zeta_{[1]})- J_2(\zeta_{[0]})\big)\\
    & \leq \; \frac{-\rho(2-\rho \check L(\ell,\hat \ell))}{2\kappa(\ell,\hat \ell)}\big((J_1(\theta_{[0]})-J_1(\theta^*))+(J_2(\zeta_{[0]})-J_2(\zeta^*))\big)\\
    &= \; \frac{-\rho(2-\rho \check L(\ell,\hat \ell))}{2\kappa(\ell,\hat \ell)}(  J(\theta_{[0]},\zeta_{[0]};\lambda)-  J(\theta^*,\zeta^*;\lambda)). 
\end{align*}
By a direct induction argument, at each step $k\geq0$, it holds
\begin{align}  
J(\theta_{[k+1]},\zeta_{[k+1]};\lambda)- J(\theta_{[k]},\zeta_{[k]};\lambda) & \leq \; 
\frac{-\rho(2-\rho \check L(\ell,\hat \ell))}{2\kappa(\ell,\hat \ell)}
\big(J(\theta_{[k]},\zeta_{[k]};\lambda)-  J(\theta^*,\zeta^*;\lambda) \big), 
\end{align} 
\noindent so that
\begin{align}
J(\theta_{[k+1]},\zeta_{[k+1]};\lambda)-J(\theta^*,\zeta^*;\lambda) &\leq \;  
\Big(1-\frac{\rho(2-\rho \check L (\ell,\hat \ell))}{2 \kappa(\ell,\hat \ell)}\Big)
\big(J(\theta_{[k]},\zeta_{[k]};\lambda)-J(\theta^*,\zeta^*;\lambda)\big). 
\end{align} 
We thus conclude
\begin{align} 
 J(\theta_{[k]},\zeta_{[k]};\lambda)-J(\theta^*,\zeta^*;\lambda) &\leq \; 
 \Big(1-\frac{\rho(2-\rho \check L (\ell,\hat \ell))}{2 \kappa(\ell,\hat \ell)}\Big)^k
 \big(J(\theta_{[0]},\zeta_{[0]};\lambda)-J(\theta^*,\zeta^*;\lambda)\big). 
 \end{align} 
The above inequality readily leads to the rest of the assertions in the theorem.
\ep

 \vspace{1mm}

\begin{Remark} \label{remcomp} 
    The problem of the solvability of the Lyapunov equations \eqref{ALEForKtheta}-\eqref{ALEForLambdatheta} have been investigated in \cite{GCPFLQMF}. The main difference with our result is that the authors in \cite{GCPFLQMF} have shown that if $\theta\in\Sc$ and $\lVert \theta'-\theta\rVert\leq c(\theta)$ with the perturbation $c(\theta)$ depending on $\theta$ then $\theta'\in\Sc$. 
    Therefore, at each step  $\theta'=\theta-\rho\nabla J(\theta)$, they choose the step size $\rho>0$ according to $\nabla J(\theta)$ in order to make $\lVert \rho \nabla J(\theta)\rVert=\lVert \theta'-\theta\rVert$ the perturbation less than $c(\theta)$. 

    Here, we adopt a different approach: we truncate $\Sc$ and define a compact subset $\Sc(\ell)$. We demonstrate that $\nabla J$ is Lipschitz continuous on $\Sc(\ell)$ with a Lipschitz constant $L(\ell)$ and directly choose the step size $\rho > 0$ according to $L(\ell)$, ensuring that the sequence $(\theta_k)_{k \geq 0}$ remains in $\Sc(\ell)$. Thus, $\rho > 0$ remains constant throughout all iterations. We refer to \cite{ConvSampGradMethod} for a detailed description of the method we have extended to solve an LQ MFC problem. It is important to note that in \cite{ConvSampGradMethod}, the only source of randomness is the initial value.    
   \end{Remark}

\section{Model-free PG algorithm}\label{sec:model:free}

\subsection{Notations}
In the model-free setting, we do not have access to the values of the functions $J_1$ and $J_2$, or their gradients, as the model coefficients are unknown. Hence, we combine the two parameters $\Theta = (\theta, \zeta) \in H = \R^{m \times d} \times \R^{m \times d}$. The value function $J_1 + J_2$ is then regarded as a function defined on $H$. 

For all $\Theta=(\theta,\zeta)\in H$, we define the norm of $\Theta$ by 
\begin{align} 
\lVert\Theta\rVert_H &= \; \sqrt{\lVert \theta\rVert_F^2+\lVert\zeta\rVert_F^2}, 
\end{align} 
and for $\Theta_1=(\theta_1,\zeta_1),\Theta_2=(\theta_2,\zeta_2)\in H$, their inner product is defined by
\begin{align} 
\langle \Theta_1,\Theta_2\rangle_H & := \;  \theta_1 : \theta_2 \; + \;  \zeta_1 : \zeta_2. 
\end{align} 


We let
\begin{align} 
\Rc& := \; \Big\{\Theta=(\theta,\zeta)\in H\, : \, (\theta, \zeta) \in\Sc \times \hat\Sc\Big\} \subset H,
\end{align} 
and, for $b>0$, we define the level subset of $\Rc$ by
\begin{align}\label{Rcb}
\Rc(b) &:= \; \Big\{\Theta=(\theta,\zeta)\in \Rc \, : \, \check J(\Theta)\leq b\Big\},
\end{align}
where  
\begin{align}\label{checkJ}
\check J(\Theta) &:= \; J(\Theta;\lambda) - \upsilon(\lambda) \; = \; J_1(\theta)+J_2(\zeta).
\end{align}
Note that $\check{J}$ is obtained by removing the constant part that depends only on $\lambda$ from $J$.

\vspace{1mm}

Since $J_1$ and $J_2$ are non-negative functions, one has $\Rc(b)\subset\Sc(b)\times\hat \Sc(b)$. Having shown that both $\Sc(b)$ and $\hat \Sc(b)$ are compact subsets of $\R^{m\times d}$, we get that $\Rc(b)$ is bounded. The continuity of $\check J$ on $H$ also implies that $\Rc(b)$ is closed. Since $H$ is of finite dimension, we deduce that $\Rc(b)$ is compact.

\vspace{1mm}

Finally, one can define the total gradient of $J$ and $\check J$ as 
\begin{align} 
\nabla \check J(\Theta) = \nabla_\Theta J(\Theta;\lambda) \; := \; 
(\nabla_\theta J(\theta,\zeta;\lambda),\nabla_\zeta J(\theta,\zeta;\lambda)) \; = \; 
(\nabla J_1(\theta),\nabla J_2(\zeta)).
\end{align} 
According to Theorem \ref{GradDomCond}, the gradient domination inequality holds:
\begin{align}\label{GD}
J(\Theta;\lambda) - J(\Theta^*;\lambda) &= \;   J_1(\theta)-J_1(\theta^*) +    J_2(\zeta)-J_2(\zeta^*) 
\;    \leq \; \kappa\lVert\nabla J(\Theta)\rVert_H^2,
\end{align}
where $\kappa=\max\{\kappa_1,\kappa_2\}$. The same inequality also holds for  $\check J$.

\subsection{Model-free PG algorithm with population simulator}

In the current model-free setting, we do not explicitly know the gradient $\nabla \check{J}(\theta, \zeta)$. Therefore, the GD algorithm from the previous section cannot be directly applied. Instead, we rely on a stochastic PG algorithm based on a stochastic population simulator, which provides an approximation of the controlled MKV dynamics \eqref{DynXtheta} along with the associated cost $J$ \eqref{defJtheta}.

\vspace{1mm}

For  a finite terminal horizon $T>0$ and a positive integer $n$, we consider the uniform time grid of the interval $[0,T]$ given by $\Delta:=\Delta(T,n)=\{0=t_0<t_1<\dots<t_{n}=T\}$ where $t_l=l h$, $l=0,1,\dots n$, and $h:=\frac{T}{n}$. For a given positive integer $N$, the $N$ interacting agents system with states $(X^{\Theta,\Delta,(j)}_{t_l})_{1\leq j \leq N; 0\leq l \leq n}$ evolves according to the dynamics
\begin{equation}\label{XDeltThet}
\begin{aligned}
     X^{\Theta,\Delta,(j)}_{t_{l+1}}&=X^{\Theta,\Delta,(j)}_{t_l}+(BX^{\Theta,\Delta,(j)}_{t_l}+\bar B\hat\mu^{\Theta,\Delta,N}_{t_l}+D\alpha^{\Theta,\Delta,(j)}_{t_l})h\\
     & \quad + \; \gamma (W^{(j)}_{t_{l+1}}-W^{(j)}_{t_l})+\gamma_0 (W^0_{t_{l+1}}-W^0_{t_l}), \quad j=1,\dots,N, \, l=0, \cdots, n-1,
\end{aligned}
\end{equation}
where $X^{\Theta,\Delta,(j)}_{0}=X^{(j)}_0$, $\hat\mu^{\Theta,\Delta,N}_{t_l}:=\frac{1}{N}\sum_{j=1}^N X^{\Theta,\Delta,(j)}_{t_l}$,  
and the action that the $j$-th agent takes at time $t=t_l$ is drawn as $\alpha^{\Theta,\Delta,(j)}_{t_l}\sim\pi^\Theta(\cdot| X^{\Theta,\Delta,(i)}_{t_l}-\hat\mu^{\Theta,\Delta,N}_{t_l},\hat\mu^{\Theta,\Delta,N}_{t_l})$ independently of $((W^{(j)})_{j=1,\dots,N}, W^0)$. Here, $(W^{(j)})_{j=1,\dots,N}$ are $i.i.d.$ $d$-dimensional Brownian motions on $(\Omega^1,\Fc^1,\P^1)$ and $W^0$ is the common noise on $(\Omega^0,\Fc^0,\P^0)$, and $(X^{(j)}_0)_{1\leq j \leq N}$ is independent of 
$((W^{(j)})_{j=1,\dots,N}, W^0)$.

\vspace{1mm}

In the spirit of \cite{YZhangLSzpruch2023}, the population simulator samples the randomized actions $\alpha^{\Theta,\Delta,(j)}_{t_l}$ as follows
\begin{equation}\label{ExecActRandom}
\begin{aligned}
    \alpha^{\Theta,\Delta,(j)}_{t_l}& =\theta(X^{\Theta,\Delta,(j)}_{t_{l}}-\hat\mu^{\Theta,\Delta,N}_{t_l})\\
    & \quad +\zeta\hat\mu^{\Theta,\Delta,N}_{t_l}+\sqrt{\frac{\lambda}{2}R^{-1}}\xib^{(j)}_{t_l}, \quad j=1,\dots,N, \, l=0,\dots,n-1, 
\end{aligned}
\end{equation}
 where the \textit{action simulation noises} $(\xib^{(j)}_{t_l})_{1\leq j\leq N; 0\leq l\leq n-1}$ are \emph{i.i.d.} random variables independent of $((X_0^{(j)}, W^{(j)})_{j=1,\dots,N}, W^0)$ with law $\Nc(0,\mathbf{I}_m)$. In particular, one may assume that $\Fc_0$ is rich enough to support not only $(X_0^{(j)})_{1\leq j \leq N}$ but also the sequence $(\xib^{(j)}_{t_l})_{1\leq j \leq N; 0 \leq l\leq  n}$. 

 Introducing the notations $\bw^{(j)}_l=\frac{1}{\sqrt{h}}(W^{(j)}_{t_{l+1}}-W^{(j)}_{t_l}), \bw^{0}_l=\frac{1}{\sqrt{h}}(W^{0}_{t_{l+1}}-W^{0}_{t_l})$, the dynamics \eqref{XDeltThet} writes
\begin{equation}\label{XDeltThet2}
\begin{aligned}
      X^{\Theta,\Delta,(j)}_{t_{l+1}}&=X^{\Theta,\Delta,(j)}_{t_l}+\Big(BX^{\Theta,\Delta,(j)}_{t_l}+\bar B\hat\mu^{\Theta,\Delta,N}_{t_l}+D\big(\theta(X^{\Theta,\Delta,(j)}_{t_{l}}-\hat\mu^{\Theta,\Delta,N}_{t_l})+\zeta\hat\mu^{\Theta,\Delta,N}_{t_l}\\
      &\qquad +\sqrt{\frac{\lambda}{2}R^{-1}}\xib^{(j)}_{t_l} \big)\Big) h+\sqrt{h} \gamma\bw^{(j)}_l+\sqrt{h}\gamma_0 \bw^0_l, \quad j=1,\dots,N, \, l=0, \cdots, n-1, 
\end{aligned}
\end{equation}
and we notice that $(X^{\Theta,\Delta,(j)})_{j=1,\dots,N}$ are exchangeable in law.

The discounted running cost at time $t_l$ of the $j$-th agent is defined by
\begin{equation}\label{Runc}
    \begin{aligned}
         Runc^{\Theta,\Delta,(j)}_{t_l}&=e^{-\beta t_l}\big((X^{\Theta,\Delta,(j)}_{t_l}-\hat\mu^{\Theta,\Delta,N}_{t_l})\trans Q (X^{\Theta,\Delta,(j)}_{t_l}-\hat\mu^{\Theta,\Delta,N}_{t_l})+(\hat\mu^{\Theta,\Delta,N}_{t_l})\trans\hat Q\hat\mu^{\Theta,\Delta,N}_{t_l}\\
    & \qquad + \; (\alpha^{\Theta,\Delta,(j)}_{t_l})\trans R\alpha^{\Theta,\Delta,(j)}_{t_l}\big),    
    \end{aligned}
\end{equation}
and the average cost estimation of the population over the interval $[0,T]$ for a given $\Theta \in H$ is given by 
\begin{equation}\label{EstimpopDef}
    \begin{aligned}
         \Jc^{\Delta,N}_{pop}(\Theta) & := \frac{{h}}{N}\sum_{j=1}^N \sum_{l=0}^{n-1}\Bigg( Runc^{\Theta,\Delta,(j)}_{t_l}\\
         & +\lambda e^{-\beta t_l} \log p^\Theta( X^{\Theta,\Delta,(j)}_{t_l}-\hat\mu^{\Theta,\Delta,N}_{t_l}, \hat\mu^{\Theta,\Delta,N}_{t_l},\alpha^{\Theta,\Delta,(j)}_{t_l})\Bigg)\\
         &=\frac{{h}}{N}\sum_{j=1}^N \Bigg(\sum_{l=1}^{n} e^{-\beta t_l} \big((X^{\Theta,\Delta,(j)}_{t_l}-\hat\mu^{\Theta,\Delta,N}_{t_l})\trans Q(X^{\Theta,\Delta,(j)}_{t_l}-\hat\mu^{\Theta,\Delta,N}_{t_l})\\
        & +(\hat\mu^{\Theta,\Delta,N}_{t_l})\trans \hat Q\hat\mu^{\Theta,\Delta,N}_{t_l} + (\alpha^{\Theta,\Delta,(j)}_{t_l})\trans R\alpha^{\Theta,\Delta,(j)}_{t_l}\\
        & +\lambda\log p^\Theta( X^{\Theta,\Delta,(j)}_{t_l}-\hat\mu^{\Theta,\Delta,N}_{t_l}, \hat\mu^{\Theta,\Delta,N}_{t_l},\alpha^{\Theta,\Delta,(j)}_{t_l})\big)\Bigg).
    \end{aligned}
\end{equation}
Note that, since $\beta$, $\lambda>0$, $R$ and the density $p^\Theta$ are known, the last term in the above expression can be computed explicitly.

In the spirit of \cite{fazetal18}, see also \cite{carlautan19} and \cite{carlautan19a}, we compute in Algorithm \ref{gradient:estimation:algorithm}  a biased estimator of the true gradient $\nabla \check J$  $=$ 
$(\nabla_\theta \check J, \nabla_\zeta \check J)$ based solely on the average cost estimation \eqref{EstimpopDef}.
We let $\B_r\subset \R^{m\times d}$ be the ball of radius $r$ (with respect to the Frobenius norm) centered at the origin, and $\S_r=\partial \B_r$ be its boundary. The uniform distribution on $\S_r$ is denoted by $\mu_{\S_r}$.\\

\vspace{1mm}

\begin{algorithm2e}[H] \label{gradient:estimation:algorithm}
\DontPrintSemicolon 
\SetAlgoLined 
\vspace{1mm}
{\bf Input data}: Feedback parameters $(\theta, \zeta)\in\Sc \times \hat\Sc$ , number of agents $N$,  number of perturbations $\tilde N$, finite horizon $T>0$, number of period of discrete grid $n$, radius $r$. 

\For{ $i$ $=$ $1,\ldots,\tilde N$}
{\begin{itemize}
\item Sample $U_i$ and $V_i$ \, i.i.d. $\sim\mu_{\S_r}$
   \item  Define perturbed feedback parameters 
   $$
   \theta_{i}=\theta+ U_i\; ;\; \zeta_{i} = \zeta + V_i.
   $$
   \item Sample $ \Jc^{\Delta,N, i}_{pop} $ via the population simulator \eqref{EstimpopDef} for $\Theta= \Theta_i=(\theta_{i},\zeta_{i})$ .
\end{itemize}

}

{\bf Return: } The estimations of the gradient of the functional cost with respect to $\Theta =(\theta,\zeta)$:
\begin{align*}
    \tilde\nabla^{\Delta,N,pop}_{\theta} J(\Theta)&=\frac{d}{r^2}\frac{1}{\tilde N}\sum_{i=1}^{\tilde N} \Jc^{\Delta,N, i}_{pop} U_i, \\
     \tilde\nabla^{\Delta,N,pop}_{\zeta} J(\Theta)&=\frac{d}{r^2}\frac{1}{\tilde N}\sum_{i=1}^{\tilde N} \Jc^{\Delta,N, i}_{pop} V_i.
\end{align*}
\caption{Model Free Population-Based Gradient Estimation}
\label{Gradientestim}
\end{algorithm2e}


\vspace{7mm}

We are now in position to define our (stochastic) model-free PG algorithm. Given $(\theta_{[0]}, \zeta_{[0]}) \in \mathcal{R}(b)$, for some $b>0$, recalling \eqref{Rcb}, we update the parameters as follows
\begin{equation}\label{model:free:PG:algorithm}
\theta_{[k+1]}=\theta_{[k]} - \rho \tilde\nabla^{\Delta,N,pop}_{\theta} J(\theta_{[k]},\zeta_{[k]}), \quad \zeta_{[k+1]}=\zeta_{[k]} - \rho  \tilde\nabla^{\Delta,N,pop}_{\zeta} J(\theta_{[k]},\zeta_{[k]}) , \quad k\geq0,  
\end{equation}                     
for some well-chosen constant positive learning rate $\rho$. 

\vspace{7mm}

Our main contribution is the following convergence result of the model-free PG sequence $(\Theta_{[k]}:=(\theta_{[k]}, \zeta_{[k]}))_{k\geq0}$ as given by \eqref{model:free:PG:algorithm}.

\begin{Theorem}\label{ThmConvSGD}
    Let $\varepsilon>0$ be a prescribed accuracy and $\rho\in(0,\frac{4}{9\check L(b)})$ be the constant learning rate. Then, choosing $r$ small enough such that
    $$
    r \leq \breve{r}(\varepsilon) := \min\Big(\check r(b),\frac{1}{\h_r(b,1/(\frac{1}{10\sqrt{2}}\sqrt{\frac{\varepsilon}{\kappa}}))},\frac{1}{\h'_r(b,1/(\frac{1}{10\sqrt{2}}\sqrt{\frac{\varepsilon}{\kappa}}))}\Big),
    $$
    recalling that $\kappa$ is the coefficient in the gradient domination condition appearing in \eqref{GD}, choosing $T$ large enough such that 
    $$
    T \; \geq \;  { \breve{T}(\varepsilon) := }  \frac{1}{c_2(2b)}\log\big(\frac{10\sqrt{2}dc_1(2b)}{\breve{r}(\varepsilon)}\sqrt{\frac{\kappa}{\varepsilon}}),
    $$
    choosing the number of periods in the grid $n$ large enough such that 
    $$
    n\; \geq \;  {\breve{n}(\varepsilon) := } \frac{10\sqrt{2}d\breve{T}(\varepsilon) c_3(2b)}{\breve{r}(\varepsilon)}\sqrt{\frac{\kappa}{\varepsilon}},
    $$
    taking the number of particles $N$ large enough such that 
    $$
    N \; \geq \; { \breve{N}(\varepsilon) := } \frac{10\sqrt{2}d {c_4(2b,\breve{T}(\varepsilon),\breve{n}(\varepsilon))}}{\breve{r}(\varepsilon)}\sqrt{\frac{\kappa}{\varepsilon}},
    $$ 
    and finally taking the number of samples $\tilde N$ in Algorithm \ref{Gradientestim} large enough such that 
     $$
     \tilde N \; \geq \; {\breve{\tilde N}(\varepsilon) := } \max\left\{\h_{\tilde N}(b,\frac{1}{\breve{r}(\varepsilon)},1/(\frac{1}{10\sqrt{2}}\sqrt{\frac{\varepsilon}{\kappa}})), {\h'_{\tilde N}(b,\frac{1}{ \breve{r}(\varepsilon)},\breve{T}(\varepsilon),\breve{n}(\varepsilon),\breve{N}(\varepsilon),1/(\frac{1}{10\sqrt{2}}\sqrt{\frac{\varepsilon}{\kappa}}))}\right\},
     $$
     it holds
 $$
 J(\Theta_{[k]})-J(\Theta^*)\leq \varepsilon
 $$ 
 \noindent with probability at least $1-4k\big(\frac{10\sqrt{2}d\sqrt{\kappa}}{\sqrt{\varepsilon}}\big)^{-d}$ in at most $$k\geq {\log\big(\frac{\varepsilon}{J(\Theta_{[0]})-J(\Theta^*)}\big)}/{\log\big(1-\frac{\rho(4-9\rho \check L(b))}{8\bar\kappa(b)}\big)}$$ iterations.   
\end{Theorem}

The proof is deferred to Section \ref{proof:main:theorem:model:free} of the appendix. In the above convergence result, $\check r(b)$, $\h_r$, $\h_r'$, $c_1(2b)$, $ c_2(2b)$, $c_3(2b)$, $c_4(2b)$, $\h_{\tilde N}$ and $\h_{\tilde N}'$ are explicit functions whose definition is provided in the proof. In particular, $\check{r}(b)$ is defined in Lemma \ref{ra}, $\h_r,\h_r'$ are defined in Propositions \ref{PropPerb}, \ref{PropMCError}, $c_1(2b)$, $c_2(2b)$ are defined in Proposition \ref{PropTrunc}, $c_3(2b)$ is defined in Proposition \ref{PropDiscre}, $c_4(2b)$ is defined in Proposition \ref{PropN} and $\h_{\tilde N},\h_{\tilde N}'$ are defined in Propositions \ref{PropPerb}, \ref{PropMCError}. 

\vspace{3mm}
\begin{Remark}
     Regarding the choice of $b$, we recall the definition of the level set, in \eqref{Rcb}-\eqref{checkJ}- and $b>0$ is the superior limit of $\check J$, where $\check J$ is the expected function cost $J$ in which we removed the known constant $\upsilon(\lambda)$ (since $\beta>0,\lambda>0,R$ are known). Hence, a natural choice is to take $b\geq  \check J(\Theta_{[0]})=J(\Theta_{[0]})-\upsilon(\lambda)$ and one can replace $J(\Theta_{[0]})$ by its sample average approximation using the average cost estimation of the population \eqref{EstimpopDef}.
\end{Remark}

\section{Numerical example}\label{sec:numerical:example}
We demonstrate the convergence analysis of our algorithms using a one-dimensional example with the following model parameters:
\begin{table}[H]
\centering
\begin{tabular}{|c|c|c|c|c|c|c|c|c|c|}
\hline
  $Q$   & $\hat{Q}$ & $B$&$\hat B$&$\beta$&$\gamma$&$\gamma_0$& $D$ &$R$&$X_0$\\
\hline
  $0.1$   & $0.2$ & $0.1$ &  $0.2$& $20$& $0.05$ & $0.05$&$0.05$ & $0.2$ &$\Nc(1,1)$\\
  \hline
\end{tabular}
\caption{The parameters of the model
} \label{ModelPara} 
\end{table}
The optimal parameters \eqref{opttheta} are $(\theta^*,\zeta^*)=(-0.00126,-0.00255)$. Thus, the optimal cost is given by $J_1(\theta^*)=0.005051$, $J_2(\zeta^*)=0.010205$. Both model-based and model-free algorithms are initialized with $(\theta_{[0]},\zeta_{[0]})=(-2,-2)$. 
\subsection{Model-based algorithm}
In the model-based case, we can derive explicit formulas for $\theta \mapsto J_1(\theta)$ and $\zeta \mapsto J_2(\zeta)$, along with their derivatives $J_1'(\theta)$ and $J_2'(\zeta)$, using the results obtained in Sections \ref{sec:problem:formulation} and \ref{sec:model:based:pg}.

We evaluate the PG algorithm (Algorithm \ref{exact:GD:algorithm}) using three different step sizes: $\rho=0.5,0.9,1.2$. Figure \ref{fig:convergence:theta:zeta} illustrates the convergence of the sequence $(\theta_{[k]}, \zeta_{[k]})_{n \geq 0}$ towards $(\theta^{\star}, \zeta^{\star})$ in terms of the number of iterations.

\begin{figure}[H]
\centering 
   \begin{minipage}[b]{0.45\linewidth}
  \centering
 \includegraphics[width=\textwidth, height=4.5cm]{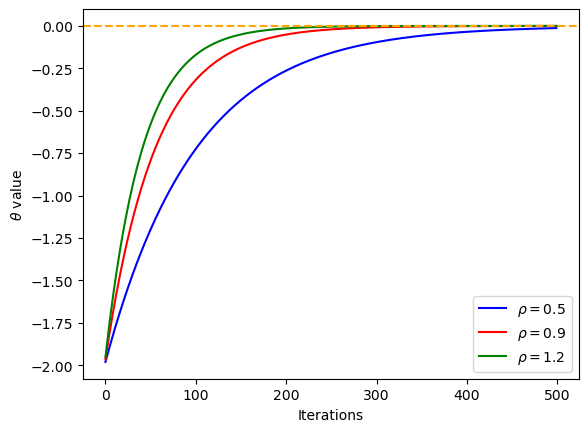}
 \end{minipage}
 \begin{minipage}[b]{0.45\linewidth}
  \centering
 \includegraphics[width=\textwidth, height=4.5cm]{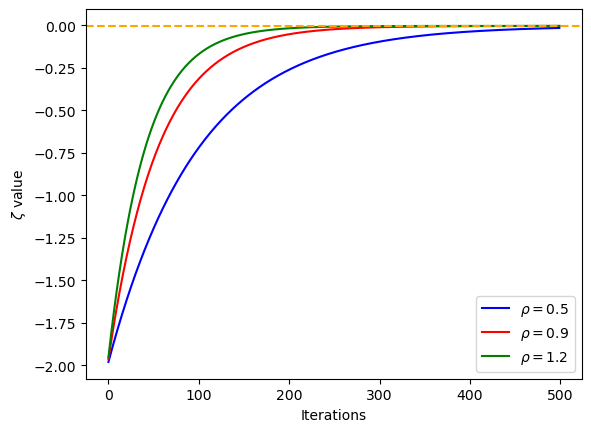}
 \end{minipage}
 \caption{Convergence of the sequence $(\theta_{[k]},\zeta_{[k]})_{k\geq0}$ towards $(\theta^{\star}, \, \zeta^{\star})$ (dashed lines).} \label{fig:convergence:theta:zeta} 
\end{figure}
Then, Figure 
\ref{fig:error:rate:cost:functions}
illustrates the convergence of the cost functions $J_1$ and $J_2$ by plotting the 
relative errors $k\mapsto (J_1(\theta_{[k]})-J_1(\theta^*))/J_1(\theta^*)$ and 
$k\mapsto (J_2(\zeta_{[k]})-J_2(\zeta^*))/J_2(\zeta^*)$. 

\begin{figure}[H] 
\centering 
   \begin{minipage}[b]{0.45\linewidth}
  \centering
 \includegraphics[width=\textwidth, height=4.5cm]{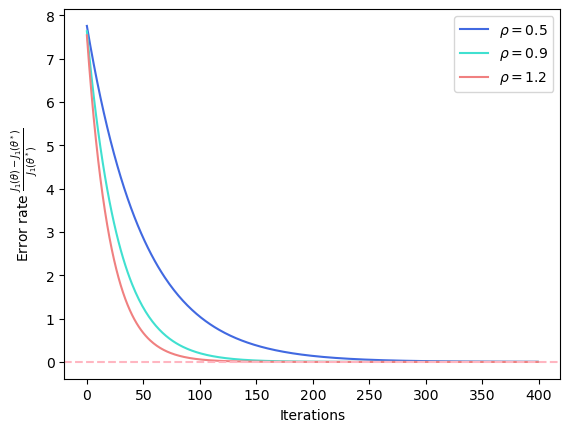}
 \end{minipage}
 \begin{minipage}[b]{0.45\linewidth}
  \centering
 \includegraphics[width=\textwidth, height=4.5cm]{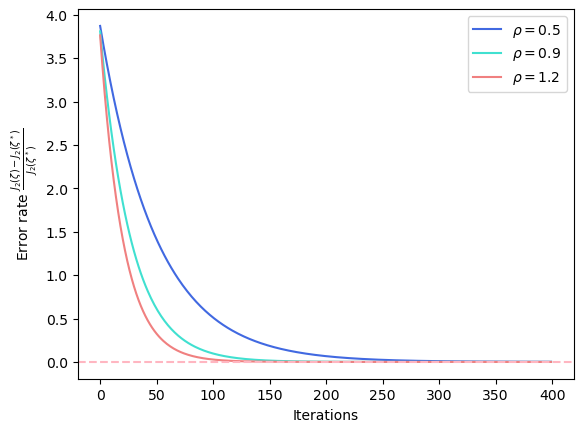}
 \end{minipage}
 \caption{Convergence of the error rates $(J_1(\theta_{[k]}))-J_1(\theta^{\star}))/J_1(\theta^{\star})$ and $(J_2(\zeta_{[k]})-J_2(\zeta^{\star}))/J_2(\theta^{\star})$.} \label{fig:error:rate:cost:functions} 
\end{figure}

 The PG algorithm \ref{exact:GD:algorithm} demonstrates very good performance. The error and the error rate of the value function converge to zero after roughly 200 iterations for different values of the learning rate.

\subsection{Model-free algorithm}
We use the following parameter values for the model-free PG algorithm:
\begin{table}[H]
\centering
\begin{tabular}{|c|c|c|c|c|c|}
\hline
  $T$   & $n$ & $N$&$\tilde N$&$r$&$\lambda$\\
\hline
  $1$   & $100$ & $100$ &  $100$& $0.05$& $0.001$ \\
  \hline
\end{tabular}
\caption{Parameters of the model-free PG algorithm.
} \label{ModelParaMF} 
\end{table}
The optimal cost corresponding to the parameters given in Table \ref{ModelPara} is $J(\theta^*,\zeta^*)=0.015360$.
First, we test whether the parameters in Table \ref{ModelParaMF} allow for accurate gradient estimations. In Figure \ref{fig:TestGradEstim}, we evaluate this for $(\theta_{[0]},\zeta_{[0]})=(-2,-2)$, by generating 100 gradient estimations using the gradient estimation algorithm (Algorithm \ref{Gradientestim}).
\begin{figure}[H] 
\centering 
   \begin{minipage}[b]{0.45\linewidth}
  \centering
 \includegraphics[width=\textwidth, height=4.5cm]{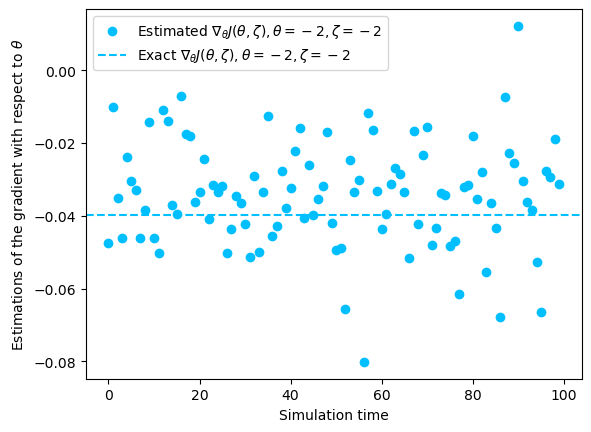}
 \end{minipage}
 \begin{minipage}[b]{0.45\linewidth}
  \centering
 \includegraphics[width=\textwidth, height=4.5cm]{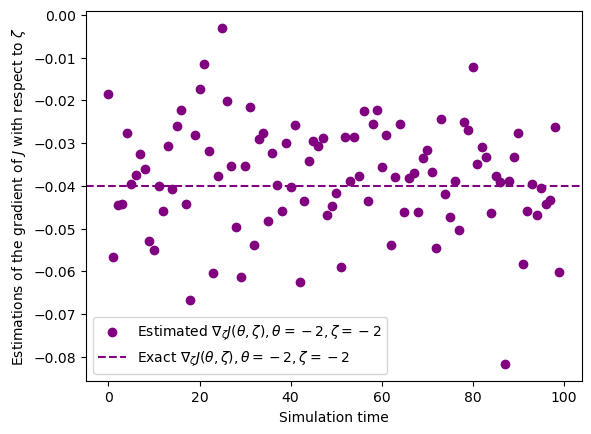}
 \end{minipage}
 \caption{100 gradient estimations vs exact values of $\nabla_\theta J(\theta, \zeta)$ and $\nabla_\zeta J(\theta, \zeta)$ (dashed lines).} \label{fig:TestGradEstim} 
\end{figure}
In the figures above, the vertical coordinates of the points represent the estimated gradients, with the dashed line indicating the exact gradients. We observe that most points cluster around the dashed line, remaining within $10^{-2}$ above and below it. This suggests that, with high probability, the error in gradient estimation is controlled within $10^{-2}$, indicating that the parameters in Table \ref{ModelParaMF} are well-chosen. 


Our main result (Theorem \ref{ThmConvSGD}) is established for a constant step size $\rho>0$. However, from a numerical standpoint, one might consider exploring adaptive selection of $\rho$ by adjusting it based on the observed behavior of the cost function $J$.
Specifically, if $\rho>0$ is in an acceptable interval (an open interval), the cost function $J$ should decrease. According to the findings of the model-based algorithm, increasing $\rho$ (while ensuring it remains within the acceptable interval) may accelerate convergence.

Therefore, we can initialize with a $\rho > 0$ and run the algorithm for several iterations (e.g., 100) initially. If we observe a clear downward trend in the (estimated) cost function $J$, it indicates that our chosen $\rho > 0$ falls within the acceptable interval. We can then consider increasing $\rho$ slightly to potentially accelerate convergence. Conversely, if we do not observe a decreasing trend, we should consider reducing $\rho > 0$.

The above results are based on the following choice of the step size as a function of the iteration index $k$:
$$
    \rho(k) = 
    \begin{cases}
        0.5 \hbox{ if }k \leq 100, \\
        0.9 \hbox{ if } 100< k \leq 200, \\
        1.2 \hbox { if } 200 < k   \leq 350
    \end{cases}
$$

In Figure \ref{ConvthetazetaMF}, we observe the convergence of both sequences $(\theta_{[k]})_{k\geq0}$ and $(\zeta_{[k]})_{k\geq0}$ towards $\theta^{\star}$ and $\zeta^{\star}$ respectively.
\begin{figure}[H] 
\centering 
   \begin{minipage}[b]{0.45\linewidth}
  \centering
 \includegraphics[width=\textwidth, height=4.5cm]{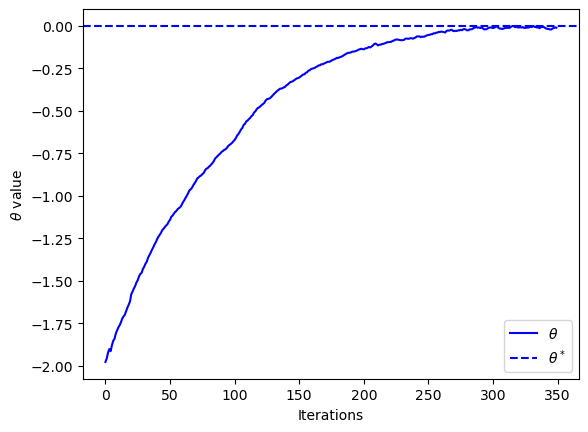}
 \end{minipage}
 \begin{minipage}[b]{0.45\linewidth}
  \centering
 \includegraphics[width=\textwidth, height=4.5cm]{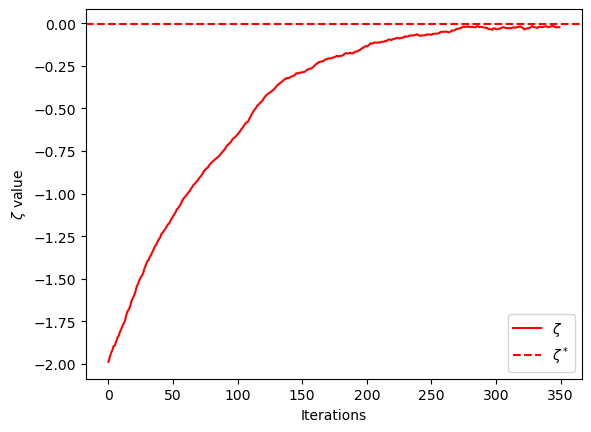}
 \end{minipage}
 \caption{Convergence of the sequences $(\theta_{[k]})_{k\geq0}$ and $(\zeta_{[k]})_{k\geq0}$ towards $\theta^{\star}$ and $\zeta^{\star}$ (dashed lines).} \label{ConvthetazetaMF} 
\end{figure}

In Figure \ref{ConvJMF}, we plot the error $(J(\theta_{[k]},\zeta_{[k]})-J(\theta^*,\zeta^*))_{k\geq0}$ (on the left side) and the error rate $(J(\theta_{[k]},\zeta_{[k]})-J(\theta^*,\zeta^*))/J(\theta^*,\zeta^*))_{k\geq0}$ (on the right side) as a function of the number of iterations. As in the model-based case, we observe excellent performance of the policy gradient algorithm. Both the error and the error rate on the value function vanish after approximately $200$ iterations.  
\begin{figure}[H] 
\centering 
   \begin{minipage}[b]{0.45\linewidth}
  \centering
 \includegraphics[width=\textwidth, height=4.5cm]{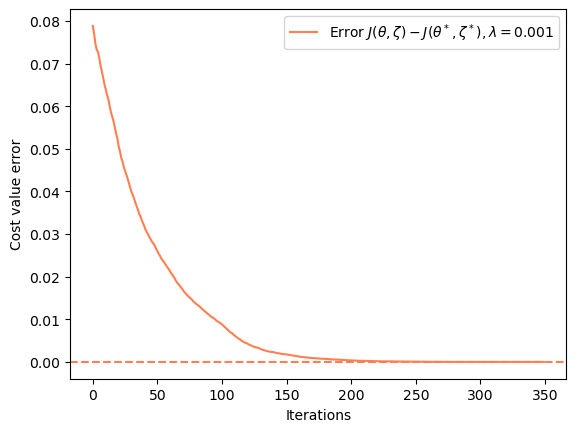}
 \end{minipage}
 \begin{minipage}[b]{0.45\linewidth}
  \centering
 \includegraphics[width=\textwidth, height=4.5cm]{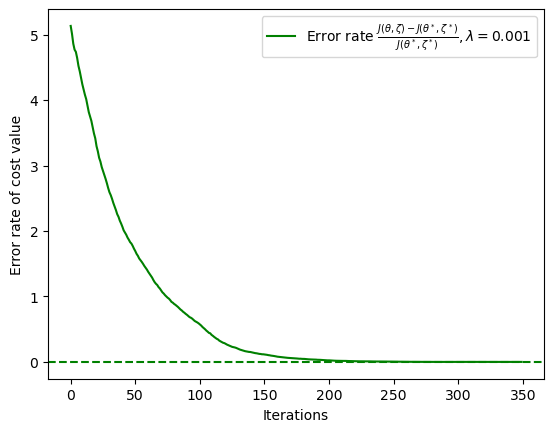}
 \end{minipage}
 \caption{\footnotesize{Convergence of the cost error $J(\theta_{[k]}, \zeta_{[k]})-J(\theta^\star, \zeta^\star)$ and error rate $(J(\theta_{[k]}, \zeta_{[n]})-J(\theta^\star, \zeta^\star))/J(\theta^\star, \zeta^\star)$.}} \label{ConvJMF} 
\end{figure}


\appendix
\section{Explicit solution to LQ MFC}\label{ProofOptSol}
Let us consider the more general linear mean-field dynamics 
\begin{equation}\label{Xpi}
\begin{cases}
\d X_t &= \; \big(B X_t + \bar B \E_0[X_t] + D \int a \bpi_t(\d a)  \big) \d t  \\
& \quad + \; \big( \gamma + FX_t + \bar F \E_0[X_t] \big) \d W_t + \; \big( \gamma_0 + F_0 X_t + \bar F_0 \E_0[X_t]  \big) \d W^0_t, \\
X_0 & \sim \; \mu, 
\end{cases}    
\end{equation}
and a quadratic cost with entropy regularizer 
\begin{align}
J(\bpi) &= \;  \E \Big[ \int_0^\infty e^{-\beta t} \Big( X_t\trans Q X_t + 
\E_0[X_t]\trans\bar Q \E_0[X_t]  \\
& \qquad \qquad + \;  \int a\trans R a \; \bpi_t(\d a) + \lambda \int \log \bop_t(a) \bpi_t(\d a) \Big) \d t    \Big]    
\end{align}
to be minimized over randomized controls $\bpi$ $=$ $(\bpi_t)_t$ with density $a$ $\in$ $\R^m$  $\mapsto$ $(\bop_t(a))_t$ of feedback form:  $\bpi_t$ $=$ $\pi(.|t,X_t,\E_0[X_t])$ for some randomized policy $(t,x,\mu)$ $\mapsto$ $\pi(.|t,x,\mu)$  $\in$ $\Pc_2(\R^m)$. 
For simplicity of presentation, we assume here that $W$ and $W^0$ are one-dimensional Brownian motions, see 
Remark \ref{remmulti} for the extension to the multi-dimensional case.

\begin{Theorem}\label{OptRelaxedControlPbl}
    Assume that the two following coupled Riccati equations for $K\in\S^d$ and $\Lambda\in\S^d$ have a unique positive definite solutions
      \begin{equation}\label{eqs:riccati:thm}
    \begin{cases}
         -\beta K+Q+KB+B\trans K+F\trans K F+F_0\trans KF_0-K\trans DR^{-1}D\trans K=0\\
         -\beta\Lambda+(Q+\bar Q)+\Lambda(B+\bar B)+(B+\bar B)\trans\Lambda+(F+\bar F)\trans K(F+\bar F)+(F_0+\bar F_0)\trans \Lambda(F_0+\bar F_0)\\
         -\Lambda\trans D R^{-1} D\trans \Lambda=0,
    \end{cases}
    \end{equation}
    Let  $(K,\Lambda)\in (\Sc^d_{>+})^2$ be the unique positive definite solutions to \eqref{eqs:riccati:thm} and  $Y\in\R^d$ be the solution to
    \begin{equation}\label{notations:vector:eq:thm}
           -\beta Y+(B+\bar B)\trans Y+(F+\bar F)\trans K\gamma\trans+(F_0+\bar F_0)\trans \Lambda\gamma_0\trans-\Lambda\trans DR^{-1}D\trans Y=0
    \end{equation}    
    Then the optimal randomised control is of feedback form with Gaussian distribution, namely for all 
    $t\geq 0$, $\pi^*_t(\cdot)=\pi^*(\cdot|X^{*}_t,\E_0[X^{*}_t])$ with
    $$
    \pi^*(\cdot|x,\bar x)=\mathcal{N}\Bigg(-R^{-1}D\trans K(x-\bar x)-R^{-1}D\trans \Lambda\bar x-R^{-1} D\trans Y;\frac{\lambda}{2}R^{-1}\Bigg)
    $$
    \noindent where $X^{{\star}}$ is the unique solution to \eqref{Xpi} with $\bpi=\bpi^*$. Moreover, the optimal functional cost satisfies
    \begin{align*}
        &J(\bpi^*)=\E\big[(X_0-\E[X_0])\trans K(X_0-\E[X_0])+\E[X_0]\trans \Lambda\E[X_0]+2Y\trans X_0\big]\\
        &+\frac{1}{\beta}\Bigg(\gamma\trans K\gamma+\gamma_0\trans\Lambda\gamma_0-Y\trans DR^{-1}D\trans Y-\frac{\lambda m}{2}\log(2\pi)-\frac{\lambda}{2}\log\Bigg| \frac{\lambda}{2\det(R)}\Bigg|\Bigg).
    \end{align*}
\end{Theorem}
\begin{proof}
We employ the same Martingale approach as in \cite{frietal23} and \cite{MartingApproLQ} to prove that $\bpi^{\star}$ is the optimal policy, noting that the value function is time-independent in the infinite horizon case. In what follows, we omit the dependence of the process w.r.t the control.

\vspace{1mm}

   \noindent \textit{Step 1.} Let us consider the function defined on $\R^d \times\Pc_2(\R^d)$ by $w(x,\mu)=\bar w(x,\bar \mu)$, recalling that $\bar \mu$ stands for the mean of $\mu$, where $\bar w$ is defined by
   \begin{align} 
   \bar w(x,\bar x) &= \; (x-\bar x)\trans K(x-\bar x)+\bar x\trans\Lambda\bar x+2Y\trans x+r, \quad  \R^d\times\R^d 
   \end{align} 
for some matrices/vectors $K,\Lambda\in\S^d_+$, $Y\in\R^d$ and $r\in\R$. Then, given $\bpi\in\Pi$ with density $\bop$ and $X:=X^{\bpi}$ the unique solution to \eqref{Xpi}, we introduce the following process
   \begin{equation*}
   \begin{aligned}
        \Sc^\bpi_t &:= e^{-\beta t}\bar w( X_t,\E_0[X_t])+ \int_0^t e^{-\beta  s} \Big(  (X_s)\trans Q X_s+\E_0[X_s]\trans\bar Q \E_0[X_s]\\ & \qquad + \; \int a\trans R a\bpi_s(\d a)  + \lambda \int \log \bop_s(a)\bpi_s(\d a) \Big)  \d s, \quad t\geq 0,
   \end{aligned}
   \end{equation*}
   where  we set $\bop_t(a)=\bpi_t(a)$ the density function of $\bm\pi$ for $t\geq 0$. Note that $\E_0[X_t]$ satisfies
   \begin{equation*}
       \d \E_0[X_t]=\big((B+\bar B)\E_0[X_t]+D\E_0[\int a\bpi_t(\d a)]\big)\d t+(\gamma_0+(F_0+\bar F_0)\E_0[X_t])\d W^0_t. 
   \end{equation*}
   
   \noindent\textit{Step 2.} The derivatives of $w$ are given by
   $$
   \partial_x \bar w(x,\bar x)=2K(x-\bar x)+2Y, \quad \partial_{\bar x}\bar w(x,\bar x)=-2K(x-\bar x)+2\Lambda\bar x,
   $$
   $$
   \partial^2_{xx}\bar w(x,\bar x)=2K, \quad \partial^2_{\bar x\bar x}\bar w(x,\bar x)=2(K+\Lambda), \quad 
   \partial^2_{x\bar x}\bar w(x,\bar x)=\partial^2_{\bar x  x}\bar w(x,\bar x)=-2K.
   $$
   Moreover, using the dynamics of $X_t$ and $\E_0[X_t]$, we obtain
\begin{equation*}
    \begin{aligned}
        \d \langle X,X\rangle_t&=\big((\gamma+F X_t+\bar F\E_0[X_t])(\gamma+F X_t+\bar F\E_0[X_t])\trans\\
       & +(\gamma_0+F_0 X_t+\bar F_0\E_0[X_t])(\gamma_0+F_0 X_t+\bar F_0\E_0[X_t])\trans\big) \, \d t,\\
        \d \langle \E_0[X],\E_0[X]\rangle_t&= (\gamma_0+(F_0+\bar F_0)\E_0[X_t])(\gamma_0+(F_0+\bar F_0)\E_0[X_t])\trans \,  \d t,\\
         \d \langle  X ,\E_0[X]\rangle_t&= (\gamma_0+(F_0+\bar F_0)\E_0[X_t])(\gamma_0+F_0 X_t+\bar F_0\E_0[X_t])\trans \, \d t.\\
    \end{aligned}
\end{equation*}
Applying Ito's rule to $(w(X_t,\E_0[X_t]))_{t\geq0}$, we obtain
   \begin{equation}\label{ItoForbarwEq1}
       \begin{aligned}
                & \d\bar w(X_t,\E_0[X_t])=\partial_x \bar w(X_t,\E_0[X_t])\cdot \d X_t+\partial_{\bar x } \bar w(X_t,\E_0[X_t])\cdot \d \E_0[X_t]+\frac{1}{2}\partial^2_{xx}  \bar w(X_t,\E_0[X_t]):\d \langle X,X\rangle_t\\
       &+\frac{1}{2}\partial^2_{\bar x\bar x} \bar w(X_t,\E_0[X_t]):\d \langle \E_0[X],\E_0[X]\rangle_t+\partial^2_{x\bar x }\bar w(X_t,\E_0[X_t]):\d \langle X,\E_0[X]\rangle_t\\
       &=\Bigg(2\big((X_t-\E_0[X_t])\trans K+Y\trans\big)(BX_t+\bar B\E_0[X_t]+D\int a\bpi_t(\d a))\\
       &+2\big(-(X_t-\E_0[X_t])\trans K+\E_0[X_t]\trans\Lambda\big)\big((B+\bar B)\E_0[X_t]+D\E_0[\int a\bpi_t(\d a)]\big)\\
       &+\langle K,\big((\gamma+F X_t+\bar F\E_0[X_t])(\gamma+F X_t+\bar F\E_0[X_t])\trans+(\gamma_0+F_0 X_t+\bar F_0\E_0[X_t])(\gamma_0+F_0 X_t+\bar F_0\E_0[X_t])\trans\big)\rangle\\
       &+\langle K+\Lambda,(\gamma_0+(F_0+\bar F_0)\E_0[X_t])(\gamma_0+(F_0+\bar F_0)\E_0[X_t])\trans \rangle\\
       &+\langle -2K,(\gamma_0+(F_0+\bar F_0)\E_0[X_t])(\gamma_0+F_0 X_t+\bar F_0\E_0[X_t])\trans\rangle\Bigg)\d t+\Bigg( \dots\Bigg)\d W_t+\Bigg( \dots\Bigg)\d W^0_t\\
        \end{aligned}
   \end{equation}
   and taking conditional expectation $\E_0$ on both sides of the preceding equality
    \begin{equation}\label{ItoForbarwEq2}
       \begin{aligned}
       &\d\E_0[\bar w(X_t,\E_0[X_t])]=\Bigg(\E_0\Bigg[2\big((X_t-\E_0[X_t])\trans K+Y\trans\big)(B(X_t-E_0[X_t])+(B+\bar B)\E_0[X_t]+D\int a\bpi_t(\d a))\\
       &+2\big(-(X_t-\E_0[X_t])\trans K+\E_0[X_t]\trans\Lambda\big)\big((B+\bar B)\E_0[X_t]+D\E_0[\int a\bpi_t(\d a)]\big)\\
       &+\langle K,\big((\gamma+F X_t+\bar F\E_0[X_t])(\gamma+F X_t+\bar F\E_0[X_t])\trans+(\gamma_0+F_0 X_t+\bar F_0\E_0[X_t])(\gamma_0+F_0 X_t+\bar F_0\E_0[X_t])\trans\big)\rangle\\
       &+\langle -K+\Lambda,(\gamma_0+(F_0+\bar F_0)\E_0[X_t])(\gamma_0+(F_0+\bar F_0)\E_0[X_t])\trans \rangle\Bigg]\Bigg)\d t +\Bigg( \dots\Bigg)\d W^0_t.
       \end{aligned}
   \end{equation}

   \noindent\textit{Step 3.} Referring to Lemma 6.3 in \cite{MartingApproLQ}, we try to show that $t\in\R_+\mapsto\E[\Sc^{\bpi}_t]$ is non-decreasing for all $\bpi\in\Pi$ and $t\in\R_+\mapsto\E[\Sc^{\bpi^*}_t]$ is constant for $\bpi^\star$. Applying Ito's formula to $\Sc^\bpi_t$ and taking expectation we get
   \begin{equation}
       \d \E[\Sc^\bpi_t]=e^{-\beta t}\E[\Dc^\bpi_t]\d t.
   \end{equation}
   with
   \begin{align*}
       \Dc^\bpi_t&=-\beta \bar w( X_t,\E_0[X_t])+\frac{\d}{\d t}\E[\bar w( X_t,\E_0[X_t])]+ X_t\trans Q X_t+\E_0[X_t]\trans\bar Q \E_0[X_t]\\
       & \qquad + \; \int a\trans R a\bpi_t(\d a)  + \lambda \int(\log \bop_t(a))\bpi_t(\d a). 
   \end{align*}
   Then, using \eqref{ItoForbarwEq1} and \eqref{ItoForbarwEq2}, we obtain
   \begin{equation}
       \begin{aligned}
         & \frac{\d}{\d t}\E[\bar w( X_t,\E_0[X_t])]= \frac{\d}{\d t}\E[\E_0[\bar w( X_t,\E_0[X_t])]]\\
         &=\E\Bigg[(X_t-\E_0[X_t])\trans(KB+B\trans K+F\trans K F+F_0\trans KF_0)(X_t-\E_0[X_t])\\
         &+\E_0[X_t]\trans (\Lambda(B+\bar B)+(B+\bar B)\trans\Lambda+(F+\bar F)\trans K(F+\bar F)+(F_0+\bar F_0)\trans\Lambda(F_0+\bar F_0))\E_0[Xt]\\
         &+2(Y\trans(B+\bar B)+\gamma\trans K (F+\bar F)+\gamma_0\trans \Lambda(F_0+\bar F_0))X_t\\
         &+2\big((X_t-\E_0[X_t])\trans KD+\E_0[X_t]\trans\Lambda D+Y\trans D\big)\int_A a\bpi_t(\d a)\Bigg]+\gamma\trans K\gamma+\gamma_0\trans\Lambda\gamma_0
       \end{aligned}
   \end{equation}
   \noindent so that
   \begin{equation}
       \begin{aligned}
          & \E[\Dc^\bpi_t]=\E\Bigg[(X_t-\E_0[X_t])\trans(-\beta K+KB+B\trans K+F\trans K F+F_0\trans KF_0+Q)(X_t-\E_0[X_t])\\
         &+\E_0[X_t]\trans (-\beta\Lambda+\Lambda(B+\bar B)+(B+\bar B)\trans\Lambda+(F+\bar F)\trans K(F+\bar F)+(F_0+\bar F_0)\trans \Lambda(F_0+\bar F_0)+Q+\bar Q)\E_0[X_t]\\
         &+2(-\beta Y\trans+Y\trans(B+\bar B)+\gamma\trans K(F+\bar F)+\gamma_0\trans \Lambda(F_0+\bar F_0))X_t\\
         &+\int a\trans R a\bpi_t(\d a)+2\big((X_t-\E_0[X_t])\trans KD+\E_0[X_t]\trans\Lambda D+Y\trans D\big)\int a\bpi_t(\d a)+\lambda \int(\log \bop_t(a))\bpi_t(\d a) \\
         &+\gamma\trans K\gamma+\gamma_0\trans\Lambda\gamma_0-\beta r\Bigg].
             \end{aligned}
   \end{equation}

   \noindent \textit{Step 4.} 
   We introduce the map
   \begin{equation}\label{DefIbpit}
       \begin{aligned} 
   \mathbb{I}(\bpi_t)& =\E_{0}\Bigg[\int a\trans R a\bpi_t(\d a)+2\big((X_t-\E_0[X_t])\trans KD+\E_0[X_t]\trans\Lambda D+Y\trans D\big)\int a\bpi_t(\d a)+\lambda \int(\log \bop_t(u))\bpi_t(\d u) \Bigg].
   \end{aligned}
   \end{equation}
   Following the approach outlined in Step 3 of the proof of Theorem B.1 in \cite{frietal23}, the optimal policy that minimizes $\mathbb{I}$ is given by:
   $$
   \bpi_t^*(\cdot|X_t,\E_0[X_t])=\mathcal{N}(-R^{-1}D\trans K(X_t-\E_0[X_t])-R^{-1}D\trans \Lambda\E_0[X_t]-R^{-1}D\trans Y;\frac{\lambda}{2}R^{-1})
   $$
and the infimum writes
\begin{equation}
\begin{aligned}
    \mathbb{I}(\bpi_t^*)&=\E_{0}\Bigg[-(X_t-\E_0[X_t])\trans KD R^{-1}D\trans K(X_t-\E_0[X_t])-\E_0[X_t]\trans \Lambda D R^{-1}D\trans\Lambda\E_0[X_t]\\
    &-2Y\trans DR^{-1} D\trans\Lambda X_t-Y\trans DR^{-1}D\trans Y\Bigg]-\frac{\lambda m}{2}\log(2\pi)-\frac{\lambda}{2}\log\Bigg|\frac{\lambda}{2\det(R)}\Bigg|.
\end{aligned}
\end{equation}

\noindent\textit{Step 5.} Noting that $R$ is strictly positive-definite, for all $\bpi\in\Pi$, we have
\small
\begin{equation}\label{Step5}
       \begin{aligned}
          & \E[\Dc^\bpi_t]=\E\Bigg[(X_t-\E_0[X_t])\trans(-\beta K+KB+B\trans K+F\trans K F+F_0\trans KF_0+Q)(X_t-\E_0[X_t])\\
         &+\E_0[X_t]\trans (-\beta\Lambda+\Lambda(B+\bar B)+(B+\bar B)\trans\Lambda+(F+\bar F)\trans K(F+\bar F)+(F_0+\bar F_0)\trans K(F_0+\bar F_0)+Q+\bar Q)\E_0[X_t]\\
         &+2(-\beta Y\trans+Y\trans(B+\bar B)+\gamma\trans K(F+\bar F)+\gamma_0\trans \Lambda(F_0+\bar F_0))X_t+\langle K,\gamma\gamma\trans\rangle+\langle\Lambda,\gamma_0\gamma_0\trans\rangle-\beta r+\mathbb{I}(\bpi_t) \Bigg]\\
         &\geq \E\Bigg[(X_t-\E_0[X_t])\trans(-\beta K+KB+B\trans K+F\trans K F+F_0\trans KF_0+Q- KDR^{-1}D\trans K) (X_t-\E_0[X_t])\\
         &+\E_0[X_t]\trans (-\beta\Lambda+\Lambda(B+\bar B)+(B+\bar B)\trans\Lambda+(F+\bar F)\trans K(F+\bar F)+(F_0+\bar F_0)\trans \Lambda(F_0+\bar F_0)\\
         &+Q+\bar Q-\Lambda DR^{-1}D\trans\Lambda )\E_0[X_t]+2(-\beta Y\trans+Y\trans(B+\bar B)+\gamma\trans K(F+\bar F)+\gamma_0\trans \Lambda(F_0+\bar F_0)\\
         &-Y\trans DR^{-1}D\trans \Lambda)X_t+\gamma\trans K\gamma+\gamma_0\trans\Lambda\gamma_0-\beta r-Y\trans DR^{-1}D\trans Y-\frac{\lambda m}{2}\log(2\pi)-\frac{\lambda}{2}\log\Bigg| \frac{\lambda}{2\det(R)}\Bigg|.
       \end{aligned}
   \end{equation}
\normalsize
   \noindent Now, taking $(K,\Lambda, Y)$ as the solutions to \eqref{eqs:riccati:thm}-\eqref{notations:vector:eq:thm} and letting 
   \begin{align*}
       r&=\frac{1}{\beta}\Big(\gamma\trans K\gamma+\gamma_0\trans\Lambda\gamma_0-Y\trans DR^{-1}D\trans Y-\frac{\lambda m}{2}\log(2\pi)-\frac{\lambda}{2}\log\Bigg| \frac{\lambda}{2\det(R)}\Bigg|\Big),
   \end{align*}
we observe that the right hand side of \eqref{Step5} vanishes, which means that for all $\bpi\in\Pi$, $\E[\Dc^\bpi_t]\geq 0$. Moreover, equality in \eqref{Step5} holds true for $\bpi=\bpi^*$ so that
\begin{align} 
   \inf_{\bpi\in\Pi}\E[\Dc^\bpi_t] &= \; \E[\Dc^{\bpi^*}_t]=0.
\end{align} 
We thus conclude that for any $\bpi \in \Pi$, $t\mapsto\E[\Sc^{\bpi}_t]$ is non-decreasing on $[0,+\infty)$ and $t\mapsto\E[\Sc^{\bpi^*}_t]$ is constant on $[0,+\infty)$ which eventually implies 
\begin{align} 
J(\bpi^*) &= \; \E[\Sc^{\bpi^*}_\infty]=\E[\Sc^{\bpi^*}_0]=\E[\bar w( X_0,\E_0[X_0])].
\end{align} 
The proof is now complete.
\end{proof}

\begin{Remark}
According to \textit{Section 6} \cite{MartingApproLQ}, if $R>0$, $Q>0$ and $\hat Q=Q+\bar Q>0$, the existence of a positive definite solution to the coupled Riccati equations \eqref{eqs:riccati:thm} is guaranteed and the assumption is specifically for the uniqueness part.
\end{Remark}

\begin{Remark} \label{remmulti} 
     As noted in Remark 5.2 in \cite{MartingApproLQ}, one may extend Theorem \ref{OptRelaxedControlPbl} to the case where $W$ and $W^0$ are multi-dimensional Brownian motions. One may consider $W=(W^{\{1\}}_t, \dots, W^{\{q\}}_t)_{t\geq 0} $ (resp. $W^0=(W^{0,\{1\}}_t,\dots,W^{0,\{q_0\}}_t)_{t\geq 0} $) is a $q$-dimensional (resp. $q_0$-dimensional) standard Brownian motion and the dynamics of the controlled state process writes as
\begin{equation*}
    \begin{cases}
        \d X^{\bpi}_t=(BX^{\bpi}_t+\bar B\E_0[X^{\bpi}_t]+D\int_A a\bpi_t(\d a))\d t+\sum_{i=1}^{q}(\gamma^{\{i\}}+F^{\{i\}} X^{\bpi}_t+\bar F^{\{i\}}\E_0[X^{\bpi}_t])\d W^{\{i\}}_t\\
       +\sum_{i=1}^{q_0} (\gamma^{\{i\}}_0+F^{\{i\}}_0 X^{\bpi}_t+\bar F^{\{i\}}_0\E_0[X^{\bpi}_t])\d W^{0,\{i\}}_t,\\
        X^{\bpi}_0\sim\mu.
    \end{cases}
\end{equation*}
\end{Remark}

\section{Proof of Propositions \ref{PropSigmaTheta} and \ref{PropValueFuncForKtheta}}
\subsection{Proof of Proposition \ref{PropSigmaTheta}}\label{ProofSigmaTheta}

We first provide two auxiliary technical results.
\begin{Lemma}\label{SoltoODEForMatrix}({Solution to the ODE for symmetric matrix-valued function})\\
Let $\Gamma:[0,T]\mapsto \S^d$ be the solution the following ODE
   \begin{equation}\label{ode:symmetric:matrix}
   \frac{\d \Gamma_t}{\d t}=\Gamma_t A\trans+A\Gamma_t+ O
   \end{equation}
   \noindent for some $A\in\R^{d\times d}$ and $O\in\S^d$. Then, it holds
   $$
   \Gamma_t=\exp(tA)\Gamma_0\exp(tA\trans)+\int_0^t \exp(-(s-t)A)O\exp(-(s-t)A\trans)\d s,
   $$
   where
   $\exp(A):=\sum_{n=0}^{+\infty}\frac{A^n}{n!}$.
\end{Lemma}
\begin{Lemma}\label{LemmaB2}
    If $(\theta,\zeta)\in\Sc \times \hat\Sc$, then it holds 
    \begin{equation}
            \lim_{t\rightarrow+\infty} e^{-\beta t}\E[ Y^\theta_t( Y^\theta_t)\trans]=0
           \; , \; \lim_{t\rightarrow+\infty}  e^{-\beta t} \E[ Z^\zeta_t(Z^\zeta_t)\trans] = 0.
    \end{equation}
\end{Lemma}
\begin{proof}
\noindent From the dynamics \eqref{dynamics:ytheta:ztheta} of $( Y^\theta_t, Z^\zeta_t)_{t\geq 0}$ and It\^o's rule for $t\mapsto\E[ Y^\theta_t( Y^\theta_t)\trans]$ and $t\mapsto \E[ Z^\zeta_t( Z^\zeta_t)\trans]$, we get
\begin{equation}\label{dynamics:expected:value:ytildey:ztildez}
\begin{aligned}
    \d\big(\E[ Y^\theta_t( Y^\theta_t)\trans]\big)&=\big(\E[ Y^\theta_t( Y^\theta_t)\trans](B+D\theta)\trans+(B+D\theta)\E[ Y^\theta_t( Y^\theta_t)\trans]+\gamma\gamma\trans\big)\d t\\
    \d\big(\E[ Z^\zeta_t( Z^\zeta_t)\trans]\big)&= \big(\E[ Z^\zeta_t( Z^\zeta_t)\trans](\hat B+D\zeta)\trans + (\hat B+D\zeta)\E[ Z^\zeta_t( Z^\zeta_t)\trans]+\gamma_0\gamma_0\trans\big)\d t.
\end{aligned}
\end{equation}

 Noting that both $t\mapsto\E[ Y^\theta_t( Y^\theta_t)\trans]$ and $t\mapsto \E[ Z^\zeta_t( Z^\zeta_t)\trans]$ are solution of an ODE of the form \eqref{ode:symmetric:matrix}, Lemma \ref{SoltoODEForMatrix} guarantees that 
\begin{align*}
   \E[ Y^\theta_t( Y^\theta_t)\trans]&=\exp(t(B+D\theta))\E[Y_0(Y_0)\trans]\exp(t(B+D\theta)\trans)\\
   &+\int_0^t \exp(-(s-t)(B+D\theta))\gamma\gamma\trans\exp(-(s-t)(B+D\theta)\trans)\d s,\\
   \E[ Z^\zeta_t( Z^\zeta_t)\trans]&=\exp(t(\hat B+D\zeta))\E[Z_0(Z_0)\trans]\exp(t(\hat B+D\zeta)\trans)\\
   &+\int_0^t \exp(-(s-t)(\hat B+D\zeta))\gamma_0\gamma_0\trans\exp(-(s-t)(\hat B+D\zeta)\trans)\d s
\end{align*}

\noindent so that 
\begin{align*}
    e^{-\beta t} \E[ Y^\theta_t( Y^\theta_t)\trans] &=\exp(t(B-\frac{\beta}{2}I_d+D\theta ))\E[Y_0(Y_0)\trans]\exp(t(B-\frac{\beta}{2}I_d+D\theta)\trans)\\
    &+\int_0^t \exp(-(s-t)(B-\frac{\beta}{2}I_d+D\theta))\gamma\gamma\trans\exp(-(s-t)(B-\frac{\beta}{2}I_d+D\theta)\trans) \, \d s,\\
    e^{-\beta t}\E[ Z^\zeta_t( Z^\zeta_t)\trans]&=\exp(t(\hat B-\frac{\beta}{2}I_d+D\zeta))\E[Z_0(Z_0)\trans]\exp(t(\hat B-\frac{\beta}{2}I_d+D\zeta)\trans)\\
    &+\int_0^t \exp(-(s-t)(\hat B-\frac{\beta}{2}I_d+D\zeta))\gamma_0\gamma_0\trans\exp(-(s-t)(\hat B-\frac{\beta}{2}I_d+D\zeta)\trans) \, \d s.
\end{align*}

\noindent Now, if $(\theta,\zeta)\in\Sc \times \hat\Sc$, then both  $B-\frac{\beta}{2}I_d+D\theta$ and $\hat B-\frac{\beta}{2}I_d+D\zeta$ are stable which allows to conclude.
\end{proof}

\noindent \textit{Proof of Proposition \ref{PropSigmaTheta}}:
Using the integration by parts, we have
$$
\int_0^\infty e^{-\beta t}\d \big(\E[ Y^\theta_t( Y^\theta_t)\trans]\big)+(-\beta)\int_0^\infty e^{-\beta t} \E[ Y^\theta_t( Y^\theta_t)\trans] \, \d t=e^{-\beta t} \E[ Y^\theta_t( Y^\theta_t)\trans]\bigg|^{t=+\infty}_{t = 0}=-\E[Y_0(Y_0)\trans].
$$
Using \eqref{dynamics:expected:value:ytildey:ztildez} together with the previous identity, we get
\begin{align*}
    \int_0^\infty e^{-\beta t}\d\big( \E[ Y^\theta_t( Y^\theta_t)\trans]\big) & = \Sigma_{\theta}(B+D\theta)\trans+(B+D\theta)\Sigma_{\theta}+\frac{1}{\beta}\gamma\gamma\trans \\
    &=-\E[Y_0(Y_0)\trans]+\beta\Sigma_{\theta}
\end{align*}

\noindent so that $\Sigma_\theta$ defined by \eqref{defSig} satisfies
$$
-\beta\Sigma_{\theta}+\Sigma_{\theta}(B+D\theta)\trans+(B+D\theta)\Sigma_{\theta}+\E[Y_0(Y_0)\trans]+\frac{1}{\beta}\gamma\gamma\trans = 0
$$

\noindent which in turn, recalling that $M= \E[Y_0(Y_0)\trans]+\frac{1}{\beta}\gamma\gamma\trans$, allows to conclude. The proof for $\hat\Sigma_\zeta$ is similar. We thus omit the remaining technical details.

 Expressing \eqref{ALEForSigmaTheta} in the standard form of the Lyapounov equation
$$
(B-\frac{\beta}{2}I_d+D\theta)\Sigma_\theta+\Sigma_\theta(B-\frac{\beta}{2}I_d+D\theta)\trans+M=0,
$$
$$
(\hat B-\frac{\beta}{2}I_d+D\zeta)\hat\Sigma_\zeta+\hat\Sigma_\zeta(\hat B-\frac{\beta}{2}I_d+D\zeta)\trans+\hat M=0
$$ 
\noindent and, recalling that $(\theta, \zeta)\in\Sc \times \hat\Sc$, we conclude that they admit a unique positive definite solution.
\ep 

\subsection{Proof of the Proposition \ref{PropValueFuncForKtheta}}\label{ProofKTheta}

We write \eqref{ALEForKtheta} and \eqref{ALEForLambdatheta} as
$$
(B-\frac{\beta}{2}I_d+D\theta)\trans K_\theta+K_\theta(B-\frac{\beta}{2}I_d+D\theta)+Q+\theta\trans R\theta=0,
$$
$$
(\hat B-\frac{\beta}{2}I_d+D\zeta)\trans \Lambda_\zeta+\Lambda_\zeta(\hat B-\frac{\beta}{2}I_d+D\zeta)+\hat Q+\zeta\trans R\zeta=0.
$$ 
Since $(\theta, \zeta)\in\Sc \times \hat\Sc$, the above equations admit a unique positive definite solutions.

From the dynamics \eqref{dynamics:ytheta:ztheta} of $( Y^\theta_t, Z^\zeta_t)_{t\geq 0}$ and It\^o's rule, for all positive definite symmetric matrix $\Gamma$, one has
\begin{equation}
    \begin{aligned}
        \d \big(e^{-\beta t} ({Y}^\theta_t )\trans \Gamma {Y}^\theta_t\big)&=e^{-\beta t}\Bigg(({Y}^\theta_t )\trans \big( -\beta \Gamma +(B+D\theta)\trans \Gamma+\Gamma(B+D\theta)\big){Y}^\theta_t+\gamma\trans \Gamma\gamma\Bigg)\d t\\
        &\qquad + \; e^{-\beta t}\big(({Y}^\theta_t )\trans \Gamma\gamma+\gamma\trans \Gamma {Y}^\theta_t\big)\d W_t, \\
                \d \big(e^{-\beta t}( ({Z}^\zeta_t)\trans \Gamma  {Z}^\zeta_t)\big) &=e^{-\beta t}\Bigg(({Z}^\zeta_t)\trans \big( -\beta \Gamma +(\hat B+D\zeta)\trans \Gamma+\Gamma(\hat B+D\zeta)\big){Z}^\zeta_t+\gamma_0\trans \Gamma\gamma_0\Bigg)\d t\\
        & \qquad + \; e^{-\beta t}\big(({Z}^\zeta_t)\trans \Gamma\gamma_0+\gamma_0\trans \Gamma  {Z}^\zeta_t\big)\d W^0_t.
    \end{aligned}
\end{equation}
Then, from the very definition \eqref{defJ12} of $J_1$ and $J_2$, denoting by $K_\theta$ and $\Lambda_\zeta$ the respective unique solutions to \eqref{ALEForKtheta} and \eqref{ALEForLambdatheta}, we get
\begin{equation*}
    \begin{aligned}
        J_1(\theta) &=  (Q+\theta\trans R\theta) :\Sigma_\theta\\
        &= \E\Big[ \int_0^\infty e^{-\beta t} \big(  ({Y}^\theta_t )\trans (Q+\theta\trans R\theta)){Y}^\theta_t \big)\d t\Big]\\
        &=\E\Big[ \int_0^\infty e^{-\beta t} \big(  ({Y}^\theta_t )\trans \big(\beta K_\theta-(B+D\theta)\trans K_\theta- K_\theta(B+D\theta)\big){Y}^\theta_t \d t\Big] \\
        &=\E\Big[ \int_0^\infty e^{-\beta t} \big(  ({Y}^\theta_t )\trans \big(\beta K_\theta-(B+D\theta)\trans K_\theta- K_\theta(B+D\theta)\big){Y}^\theta_t \d t +\big(\dots\big)\d W_t\Big] \\
        &=-\E\Bigg[\int_0^\infty \d \Bigg(e^{-\beta t}({Y}^\theta_t )\trans K_\theta {Y}^\theta_t \Bigg)\Bigg]+\int_0^\infty e^{-\beta t}(\gamma\trans K_\theta\gamma)\d t=\E [(Y_0)\trans K_\theta Y_0]+\frac{1}{\beta} \gamma\trans K_\theta\gamma,
    \end{aligned}
\end{equation*}
\noindent and    
\begin{equation*}
\begin{aligned}
J_2(\zeta)& = (\hat Q+\zeta\trans R\zeta) : \hat\Sigma_\zeta\\
&= \E\Big[ \int_0^\infty e^{-\beta t} \big( ({Z}^\zeta_t)\trans (\hat Q+\zeta\trans R\zeta){Z}^\zeta_t\big)  \d t \Big]\\
&=\E\Big[ \int_0^\infty e^{-\beta t} \big(  ({Z}^\zeta_t)\trans \big(\beta \Lambda_\zeta-(\hat B+D\zeta)\trans \Lambda_\zeta- \Lambda_\zeta(\hat B+D\zeta)\big){Z}^\zeta_t\big)  \d t \Big]\\
&=\E\Big[ \int_0^\infty e^{-\beta t} \big(  ({Z}^\zeta_t)\trans \big(\beta \Lambda_\zeta-(\hat B+D\zeta)\trans \Lambda_\zeta- \Lambda_\zeta(\hat B+D\zeta)\big){Z}^\zeta_t\big)  \d t +(\dots)\d W^0_t\Big]\\
&=-\E\Bigg[\int_0^\infty \d \Bigg(e^{-\beta t} ({Z}^\zeta_t)\trans \Lambda_\zeta  {Z}^\zeta_t\Bigg) \Bigg]+\int_0^\infty e^{-\beta t}\gamma_0\trans \Lambda_\zeta\gamma_0\d t =\E[(Z_0)\trans \Lambda_\zeta Z_0]+\frac{1}{\beta} \gamma_0\trans \Lambda_\zeta\gamma_0.
\end{aligned}
\end{equation*}
The proof of Proposition \ref{PropValueFuncForKtheta} is now complete.
\ep

\section{Proofs of results for model-based algorithms}
\subsection{Proof of Proposition \ref{GradJ}}\label{ProofGradE}

We prove Proposition \ref{GradJ} only for $J_1$, as the part concerning $J_2$ follows similar reasoning. First, we recall an auxiliary result regarding the exponential form of the solution to continuous Lyapunov equations.
\begin{Lemma}\label{SolALE}
(Solution of continuous Lyapunov equation). Let $W$ be a stable matrix and $Q$ a symmetric matrix. The following continuous Lyapunov equation 
$$
WY+YW\trans+Q=0
$$ 
\noindent admits a unique solution $Y$ which satisfies 
$$
Y=\int_0^\infty e^{Wt}Qe^{W\trans t}\d t.
$$
\end{Lemma}

\noindent Hence, recalling the algebraic Riccati equation \eqref{ALEForLambdatheta} for $K_{\theta}$, we deduce that
$$
K_{\theta}=\int_0^\infty e^{(B-\frac{\beta}{2}I_{d}+D\theta)\trans t}(Q+\theta\trans R\theta)e^{(B-\frac{\beta}{2}I_{d}+D\theta) t} \, \d t.
$$

In order to differentiate $J_1$, we will rely on the following definition of the derivatives of a matrix-valued map.

\begin{Definition}(Differentiability of the matrix applications)The map $\mathbf{M}:\R^{m_1\times n_1}\mapsto \R^{m_2\times n_2}$ is differentiable at $\mathbf{X}$ if there exists a matrix mapping: $\mathbf{Z}:\R^{m_1\times n_1}\mapsto\R^{m_2n_2\times m_1n_1}$ such that for any $\Delta \mathbf{X} \in \R^{m_1\times n_1}$
$$
\text{vec}(\mathbf{M}(\mathbf{X}+\Delta \mathbf{X})- \mathbf{M}(\mathbf{X}))=\mathbf{Z}(\mathbf{X})\text{vec}(\Delta\mathbf{X})+O(\textnormal{tr}((\Delta \mathbf{X})\trans(\Delta \mathbf{X})))
$$

\noindent where '$\text{vec}$' stands for the vectorization operator defined for $Y=(y_{ij})_{1\leq i\leq n,1\leq j\leq m}$ by
$$
\text{vec}(Y)=(y_{11},\dots,y_{n1},y_{12},\dots, y_{n2},\dots, y_{1m},\dots ,y_{nm})\trans.
$$
\end{Definition}
\begin{Lemma}\label{DiffKtheta}
 (Diffentiability of $K_{\theta}$ with respect to $\theta$) Denoting $E_{\theta}=R\theta+D\trans K_{\theta}$, for all $\theta$, $\theta'\in\Sc$, we have
 \begin{align*}
     K_{\theta'}-K_{\theta}&=\int_0^\infty e^{(B-\frac{\beta}{2}I_d+D\theta')\trans t}[(E_{\theta})\trans (\theta'-\theta)+(\theta'-\theta)\trans E_{\theta}+(\theta'-\theta)\trans R(\theta'-\theta)]e^{(B-\frac{\beta}{2}I_d+D\theta') t} \, \dt
 \end{align*}
 
 \noindent from which it readily follows that there exists a matrix $Z_{\theta'}$ depending only on $\theta'$ such that
 $$
 \text{vec}(K_{\theta}-K_{\theta'})=Z_{\theta'}\text{vec}(\theta-\theta')+O(\lVert \theta'-\theta\rVert_F^2).
 $$
\end{Lemma}
\begin{proof}
    Taking the difference between the two equations \eqref{ALEForKtheta} solved by $K_{\theta}$ and $K_{\theta'}$, we get
     \begin{align*}  0&=(B-\frac{\beta}{2}I_d+D\theta')\trans K_{\theta'}+K_{\theta'}(B-\frac{\beta}{2}I_d+D\theta')-((B-\frac{\beta}{2}I_d+D\theta)\trans K_{\theta}+K_{\theta}(B-\frac{\beta}{2}I_d+D\theta))\\
     &+(\theta')\trans R\theta'-\theta\trans R \theta\\
&=(B-\frac{\beta}{2}I_d+D\theta')\trans K_{\theta'}+K_{\theta'}(B-\frac{\beta}{2}I_d+D\theta')\\
&-(((B-\frac{\beta}{2}I_d+D\theta')-D(\theta'-\theta))\trans K_{\theta}+K_{\theta}((B-\frac{\beta}{2}I_d+D\theta')-D(\theta'-\theta))\\
&+(\theta'-\theta+\theta)\trans R(\theta'-\theta+\theta)-\theta\trans R\theta\\
&=(B-\frac{\beta}{2}I_d+D\theta')\trans(K_{\theta'}-K_\theta)+(K_{\theta'}-K_\theta)(B-\frac{\beta}{2}I_d+D\theta')\\
&+(\theta'-\theta)\trans D\trans K_\theta +K_\theta D(\theta'-\theta)+(\theta'-\theta+\theta)\trans R(\theta'-\theta+\theta)-\theta\trans R\theta\\
&=(B-\frac{\beta}{2}I_d+D\theta')\trans(K_{\theta'}-K_\theta)+(K_{\theta'}-K_\theta)(B-\frac{\beta}{2}I_d+D\theta')\\
&+(\theta'-\theta)\trans (R\theta+ D\trans K_\theta) + (R\theta+ D\trans K_\theta)\trans(\theta'-\theta)+(\theta'-\theta)\trans R(\theta'-\theta)\\
&=(B-\frac{\beta}{2}I_d+D\theta')\trans(K_{\theta'}-K_\theta)+(K_{\theta'}-K_\theta)(B-\frac{\beta}{2}I_d+D\theta')\\
&+(\theta'-\theta)\trans E_\theta + E_\theta\trans(\theta'-\theta)+(\theta'-\theta)\trans R(\theta'-\theta)
\end{align*}
     
\noindent so that $K_{\theta'}-K_\theta$ is the unique solution to the following Algebraic Lyapounov equation for $Y$
$$
(B-\frac{\beta}{2}I_d+D\theta')\trans Y+Y(B-\frac{\beta}{2}I_d+D\theta')+(\theta'-\theta)\trans E_\theta+E_{\theta}\trans (\theta'-\theta)+ (\theta'-\theta)\trans R(\theta'-\theta)=0.
$$

From Lemma \ref{SolALE}, we thus obtain
$$
     K_{\theta'}-K_{\theta}=\int_0^\infty e^{(B-\frac{\beta}{2}I_d+D\theta')\trans t}[(\theta'-\theta)\trans E_\theta+E_{\theta}\trans (\theta'-\theta)+ (\theta'-\theta)\trans R(\theta'-\theta)]e^{(B-\frac{\beta}{2}I_d+D\theta')t} \, \d t.
$$

The previous identity directly gives 
     \begin{align*}
        & vec(K_{\theta'}-K_{\theta})=\int_0^\infty vec(e^{(B-\frac{\beta}{2}I_d+D\theta')\trans t}[(\theta'-\theta)\trans E_\theta+E_{\theta}\trans (\theta'-\theta)+ (\theta'-\theta)\trans R(\theta'-\theta)]e^{(B-\frac{\beta}{2}I_d+D\theta')t})\d t\\
         &=\big(\int_0^\infty  e^{(B-\frac{\beta}{2}I_d+D\theta')\trans t}\otimes e^{(B-\frac{\beta}{2}I_d+D\theta')\trans t}\d t\big)vec\big((\theta'-\theta)\trans E_\theta+E_{\theta}\trans (\theta'-\theta)+ (\theta'-\theta)\trans R(\theta'-\theta)\big)\\
         &=\big(\int_0^\infty  e^{(B-\frac{\beta}{2}I_d+D\theta')\trans t}\otimes e^{(B-\frac{\beta}{2}I_d+D\theta')\trans t}\d t\big)vec\big((\theta'-\theta)\trans E_{\theta'}+(E_{\theta'})\trans (\theta'-\theta)+U\big)
     \end{align*}

    \noindent where
     \begin{align*}
         U&=(\theta'-\theta)\trans R(\theta'-\theta)+(E_{\theta}-E_{\theta'})\trans (\theta'-\theta)+(\theta'-\theta)\trans (E_{\theta}-E_{\theta'})\\
         &=-(\theta'-\theta)\trans R(\theta'-\theta)+(K_{\theta'}-K_\theta))D(\theta'-\theta)+(\theta'-\theta)\trans D\trans(K_{\theta'}-K_\theta)\\
         &={O}(\lVert\theta'-\theta\rVert_F^2).
     \end{align*}
     Note that in the last line we again used the expression of $K_{\theta'}-K_\theta$. Thus, there exists $Z_{\theta'}$ that depends on $B-\frac{\beta}{2}I_d+D\theta'$ and $E_{\theta'}$ such that 
$$
     vec(K_{\theta'}-K_\theta)=Z_{\theta'}(\theta-\theta')+O(\lVert\theta'-\theta\rVert_F^2).
$$ 
     
We conclude that $K_\theta$ is indeed differentiable with respect to $\theta$. 
\end{proof}
We recall the rules of the total differentiation of the matrices:
\begin{itemize}
    \item[(1)]$\mathrm{d}(\mathbf{X}\pm \mathbf{Y})=\mathrm{d}\mathbf{X}\pm\mathrm{d}\mathbf{Y}$
    \item[(2)]$\mathrm{d}(\mathbf{X} \mathbf{Y})=(\mathrm{d}\mathbf{X})\mathbf{Y}+\mathbf{X}(\mathrm{d}\mathbf{Y})$
    \item[(3)]$\mathrm{d}(\mathbf{X}\trans)=(\mathrm{d}\mathbf{X})\trans$
    \item[(4)]$\mathrm{d}(\mathbf{X}^{-1})=-\mathbf{X}^{-1}(\mathrm{d}\mathbf{X})\mathbf{X}^{-1}$
\end{itemize}

We are now in position to prove Proposition \ref{GradJ}. We only prove the identity for the gradient of $J_1$, as the gradient of $J_2$ can be  treated similarly.

By the definition of the differentiation
$$
\d J_1(\theta)=\tr(\nabla J_1(\theta)\trans \d \theta)= \nabla   J_1(\theta):\d\theta. 
$$
Since $J_1(\theta)= K_\theta:M $, we have $\mathrm{d} {J}_1(\theta)=\mathrm{d} K_{\theta}:M$ so that
$$
 \nabla   J_1(\theta):\d\theta = \mathrm{d} K_{\theta}:M.
$$
\noindent
Then, we differentiate totally the Lyapunov equation of $K_{\theta}$ and obtain
$$-\beta \mathrm{d}K_{\theta}+\mathrm{d}K_{\theta}(B+D\theta)+(B+D\theta)\trans \mathrm{d}K_{\theta}+K_{\theta}D\mathrm{d}\theta+\mathrm{d}\theta\trans D\trans K_{\theta}+\theta\trans R \mathrm{d}\theta+\mathrm{d}\theta\trans R\theta=0$$
\noindent which writes
$$
-\beta \mathrm{d}K_{\theta}+\mathrm{d}K_{\theta}(B+D\theta)+(B+D\theta)\trans \mathrm{d}K_{\theta}+E_{\theta}\trans \mathrm{d}\theta+\mathrm{d}\theta\trans E_{\theta} = 0.
$$

We then multiply from the right the above equation by $\Sigma_{\theta}$. We obtain
$$
-\beta \mathrm{d}K_{\theta}\Sigma_{\theta}+\mathrm{d}K_{\theta}(B+D\theta)\Sigma_{\theta}+(B+D\theta)\trans \mathrm{d}K_{\theta}\Sigma_{\theta}+(E_{\theta}\trans \mathrm{d}\theta+\mathrm{d}\theta\trans E_{\theta})\Sigma_{\theta} = 0.
$$

Using the fact that $\Sigma_{\theta}$ satisfies the equation
$$
-\beta\Sigma_{\theta}+(B+D\theta)\Sigma_{\theta}+\Sigma_{\theta}(B+D\theta)\trans+M=0
$$

\noindent
we get
\begin{align*}
\text{tr}\big((\mathrm{d}K_{\theta})M\big)&=\text{tr}\big((\mathrm{d}K_{\theta})(\beta\Sigma_{\theta}-(B+D\theta)\Sigma_{\theta}-\Sigma_{\theta}(B+D\theta)\trans)\big)\\
    &=\text{tr}\big(\beta(\mathrm{d}K_{\theta})\Sigma_{\theta}-\mathrm{d}K_{\theta}(B+D\theta)\Sigma_{\theta}-(B+D\theta)\trans \mathrm{d}K_{\theta}\Sigma_{\theta}\big)\\
    &=\text{tr}((E_{\theta} \mathrm{d}\theta+\mathrm{d}\theta\trans E_{\theta})\Sigma_{\theta})\\
    &=2\text{tr}((E_{\theta}\Sigma_{\theta})\trans \mathrm{d}\theta)
\end{align*}
where in the second and the fourth lines we used the commutative property of the trace operator. Comparing the left and the right sides of the previous identity, we eventually get
$$
\nabla J_1(\theta)=2E_{\theta}\Sigma_{\theta}.
$$
The proof of Proposition \ref{GradJ} is now complete.
\ep

\subsection{Proof of Theorem \ref{GradDomCond}}\label{ProofPL}
We here prove the Polyak-Lojasiewocz inequality stated in Theorem \ref{GradDomCond}. We will only prove it for $J_1$, as $J_2$ can be treated in a completely similar manner. We first need the following auxiliary result.

\begin{Lemma}\label{PerturbJ}
(Perturbation analysis of $ J_1$) For all $\theta,\theta'\in\Sc$, it holds
\begin{align*}
J_1(\theta')-J_1(\theta)& =\textnormal{tr}[\Sigma_{\theta'}((\theta'-\theta+R^{-1}E_{\theta})\trans R(\theta'-\theta+R^{-1}E_{\theta})-E_{\theta}\trans R^{-1}E_{\theta})].
\end{align*}
\end{Lemma}
\begin{proof}
We have
\begin{align*}
   & J_1(\theta')-J_1(\theta)=\text{tr}((K_{\theta'}-K_{\theta})M)\\
    &=\text{tr}\Big(\int_0^\infty e^{(B-\frac{\beta}{2}I_d+D\theta')\trans t}[(E_{\theta})\trans (\theta'-\theta)+(\theta'-\theta)\trans E_{\theta}+(\theta'-\theta)\trans R(\theta'-\theta)]e^{(B-\frac{\beta}{2}I_d+D\theta') t}M\d t\Big)\\
    &=\text{tr}\Big(\big(\int_0^\infty e^{(B-\frac{\beta}{2}I_d+D\theta')t}Me^{(B-\frac{\beta}{2}I_d+D\theta')\trans t}\d t\big)[(E_{\theta})\trans (\theta'-\theta)+(\theta'-\theta)\trans E_{\theta}\\
    &+(\theta'-\theta)\trans R(\theta'-\theta)]\Big)\\
    &=\text{tr}\Big(\Sigma_{\theta'}((\theta'-\theta+R^{-1}E_{\theta})\trans R(\theta'-\theta+R^{-1}E_{\theta})-E_{\theta}\trans R^{-1}E_{\theta})\Big),
\end{align*}

\noindent where for the last equality we used the algebraic Riccati equation of $\Sigma_{\theta}$.
\end{proof}

\noindent \textit{Proof of Theorem \ref{GradDomCond}.}
From Lemma \ref{PerturbJ}, we deduce
\begin{align*}
    J_1(\theta)-J_1(\theta^*)&=\text{tr}[\Sigma_{\theta^*}(E_{\theta}\trans R^{-1}E_{\theta}-(\theta^*-\theta+R^{-1}E_{\theta})\trans R(\theta^*-\theta+R^{-1}E_{\theta})]\\
    &\leq \text{tr}[\Sigma_{\theta^*} (E_{\theta}\trans R^{-1}E_{\theta})]\\
    &\leq \frac{\lVert\Sigma_{\theta^*}\rVert_F}{\sigma_{\min}(R)}\text{tr}(E_{\theta}\trans E_{\theta})\\
    &\overset{(\lozenge)}{\leq} \frac{\lVert\Sigma_{\theta^*}\rVert_F}{4\sigma_{\min}(R)\sigma^2_{\min}(\Sigma_{\theta})}\text{tr}(\nabla J_1(\theta)\trans \nabla J_1(\theta))\\
    &\overset{(\triangledown)}{\leq} \frac{\lVert \Sigma_{\theta^*}\rVert_F \text{tr}(\nabla J_1(\theta)\trans\nabla J_1(\theta)) }{4\sigma_{\text{min}}(R)\sigma_{\text{min}}^2(M)}.
\end{align*}
The inequality ($\lozenge$) is a consequence of the inequality
$$
\text{tr}(\nabla J_1(\theta)\trans \nabla J_1(\theta))=4\text{tr}(\Sigma_{\theta}E_{\theta}\trans E_{\theta}\Sigma_{\theta})\geq 4\sigma^2_{\min}(\Sigma_{\theta})\text{tr}(E_{\theta} E_{\theta})
$$

\noindent which stems from Proposition \ref{GradJ}. The inequality ($\triangledown$) follows from the fact that $B-\frac{\beta}{2}I_d+D\theta$ is stable which combined with Lemma \ref{SolALE} guarantees that
$$
\Sigma_{\theta}=\int_0^\infty e^{(B-\frac{\beta}{2}I_d+D\theta)t}Me^{(B-\frac{\beta}{2}I_d+D\theta)\trans t}\d t\succeq M.
$$
\ep 

\noindent
\subsection{Proof of Proposition \ref{LJ}}\label{ProofLip}

\noindent Let us recall that $\nabla J_1(\theta)$ and $\nabla J_2(\zeta)$ are given by
$$\nabla J_1(\theta)=2(R\theta+D\trans K_\theta)\Sigma_\theta=2E_\theta\Sigma_\theta,
$$
$$
\nabla J_2(\zeta)=2(R\zeta+D\trans\Lambda_\zeta)\hat{\Sigma}_\zeta=2\hat{E}_\zeta\hat{\Sigma}_\zeta,
$$

\noindent with $E_\theta=R\theta+D\trans K_\theta,\hat{E}_\zeta=R\zeta+D\trans\Lambda_\zeta$ and where $K_\theta,\Lambda_\zeta$ are the unique solutions to
$$
-\beta K_\theta+(B+D\theta)\trans K_\theta+K_\theta(B+D\theta)+Q+\theta\trans R\theta=0,
$$
$$-\beta \Lambda_\zeta+(\hat B+D\zeta)\trans \Lambda_\zeta+\Lambda_\zeta(\hat B+D\zeta)+\hat Q+\zeta\trans R\zeta=0,
$$

\noindent and $\Sigma_\theta,\hat\Sigma_\zeta$ are the unique solutions to 
$$
-\beta\Sigma_\theta+(B+D\theta)\Sigma_\theta+\Sigma_\theta(B+D\theta)\trans+M=0,
$$
$$
-\beta\hat\Sigma_\zeta+(\hat B+D\zeta)\hat\Sigma_\zeta+\hat\Sigma_\zeta(\hat B+D\zeta)\trans+\hat M=0.
$$


Here again, we will only prove the Lipschitz continuity of $\nabla J_1$, as $\nabla J_2$ can be treated in a similar manner.

\subsubsection{Auxiliary results}

We start by proving some useful bounds on $\lVert \theta\rVert_F,\lVert K_\theta\rVert_F,\lVert \Sigma_\theta\rVert_F,\lVert E_\theta\rVert_F$ when $\theta \in \Sc (\ell)$, and similar bounds for $\lVert \zeta\rVert_F,\lVert \Lambda_\zeta\rVert_F,\lVert \hat\Sigma_\zeta\rVert_F,\lVert \hat E_\zeta\rVert_F$ when $\zeta \in \hat\Sc (\hat \ell)$,

\begin{Proposition}\label{Bds}
For all $\theta \in \Sc (\ell )$, it holds
$$
\lVert K_\theta\rVert_F\leq \Bd_{K}(\ell)=\frac{\ell}{\sigma_{\min}(M)},
$$
$$
\lVert \Sigma_{\theta}\rVert_F\leq\Bd_{\Sigma}(\ell)=\frac{\ell}{\sigma_{\min}(Q)},
$$
$$
\lVert E_{\theta}\rVert_F\leq\Bd_{E}(\ell)= \sqrt{\frac{\lVert R\rVert_F(\ell-J_1(\theta^*))}{\sigma_{\min}(M)}},
$$
$$\lVert \theta\rVert_F\leq \Bd_\theta(\ell)=\frac{1}{\sigma_{\min}(R)}\Bigg(\sqrt{\frac{\lVert R\rVert_F(\ell-J_1(\theta^*))}{\sigma_{\min}(M)}}+\frac{\lVert D\rVert_F\ell}{\sigma_{\min}(M)}\Bigg).
$$
Similarly, for all $\zeta \in \hat\Sc (\hat \ell )$, it holds
$$
\lVert \Lambda_\zeta\rVert_F\leq \Bd_{\Lambda}(\hat \ell)=\frac{\hat \ell}{\sigma_{\min}(\hat M)},
$$
$$
\lVert \hat\Sigma_{\zeta}\rVert_F\leq\Bd_{\hat\Sigma}(\hat \ell)=\frac{\hat \ell}{\sigma_{\min}(\hat Q)},
$$
$$
\lVert \hat E_{\zeta}\rVert_F\leq\Bd_{\hat E}(\hat \ell)= \sqrt{\frac{\lVert R\rVert_F(\hat \ell-J_2(\zeta^*))}{\sigma_{\min}(\hat M)}},
$$

$$\lVert \zeta\rVert_F\leq \Bd_\zeta(\hat \ell)=\frac{1}{\sigma_{\min}(R)}\Bigg(\sqrt{\frac{\lVert R\rVert_F(\hat \ell-J_2(\zeta^*))}{\sigma_{\min}(\hat M)}}+\frac{\lVert D\rVert_F \hat \ell}{\sigma_{\min}(\hat M)}\Bigg).
$$

\end{Proposition}

Before proving the above proposition, we will need the following technical result.
\begin{Lemma}\label{LowerBoundJ}
    For all $\theta \in \Sc$, the following lower bound holds 
    $$
    J_1(\theta)- J_1(\theta^*)\geq\frac{\sigma_{\min}(M)}{\lVert R\rVert_F}\lVert E_\theta\rVert_F^2.
    $$
\end{Lemma}
\begin{proof}
    Applying Lemma \ref{PerturbJ} for ${\theta'}=\theta-R^{-1}E_\theta=-R^{-1}D\trans K_\theta$, we get
    $$ 
    J_1(\theta)- J_1({\theta'})=\textnormal{tr}(\Sigma_{{\theta'}} (E_\theta)\trans R^{-1}E_{\theta})
    $$
    so that
    \begin{equation*}
        \begin{split}
             J_1(\theta)- J_1(\theta^*)&\geq  J_1(\theta)- J_1({\theta'})\\&= \textnormal{tr}(\Sigma_{{\theta'}} E_\theta\trans R^{-1}E_{\theta})\\
            &\geq \frac{\sigma_{\min}(M)}{\lVert R\rVert_F}\textnormal{tr}(E_\theta\trans E_\theta)=\frac{\sigma_{\min}(M)}{\lVert R\rVert_F}\lVert E_\theta\rVert_F^2
        \end{split}
    \end{equation*}
    \noindent where for the last inequality we used the fact that $\Sigma_\theta\succeq M$. 
\end{proof}

\noindent  \textit{Proof of Proposition \ref{Bds}}

\noindent \emph{Step 1 (Bound for $K_\theta$).} Recalling that $J_1(\theta)=\Tr(K_\theta M)$, we directly deduce
    $$ J_1(\theta)\geq \Tr(K_\theta)\sigma_{\min}(M)\geq \lVert K_\theta\rVert_F \sigma_{\min}(M)
$$
\noindent so that
$$
    \lVert K_\theta\rVert_F\leq\frac{ J_1(\theta)}{\sigma_{\min}(M)}\leq \frac{\ell}{\sigma_{\min}(M)}.
$$

\noindent \emph{Step 2 (Bound for $\Sigma_\theta$).} We use the fact that $J_1(\theta)=\textnormal{tr}(\Sigma_\theta(Q+\theta\trans R\theta))$ to deduce
$$ 
    J_1(\theta)\geq \Tr(\Sigma_\theta)\sigma_{\min}(Q+\theta\trans R\theta)\geq \lVert \Sigma_\theta\rVert_F\sigma_{\min}(Q)
$$
\noindent which in turn implies
$$
\lVert \Sigma_\theta\rVert_F\leq\frac{ J_1(\theta)}{\sigma_{\min}(Q)}\leq \frac{\ell}{\sigma_{\min}(Q)}.
$$

\noindent \emph{Step 3 (Bound for $E_\theta$).} From Lemma \ref{LowerBoundJ}, we get
$$
    \lVert E_\theta\rVert_F\leq \sqrt{\frac{\lVert R\rVert_F ( J_1(\theta)- J_1(\theta^*))}{\sigma_{\min}(M)}}\leq \sqrt{\frac{\lVert R\rVert_F(\ell- J_1(\theta^*))}{\sigma_{\min}(M)}}.
$$

\noindent \emph{Step 4 (Bound for $\lVert \theta\rVert_F$).} We first write 
    \begin{equation*}
        \begin{split}
            \lVert\theta\rVert_F&\leq\lVert R\theta\rVert_F\lVert R^{-1}\rVert_F\\
            &\leq\frac{\lVert R\theta+D\trans K_\theta\rVert_F+\lVert -D\trans K_\theta\rVert_F}{\sigma_{\min}(R)}\\
            &\leq \frac{\lVert E_\theta\rVert_F+\lVert D\rVert_F\lVert K_\theta\rVert_F}{\sigma_{\min}(R)}\\
        \end{split}
    \end{equation*}

 \noindent which combined with Lemma \ref{LowerBoundJ} and the bound on $\lVert K_\theta\rVert_F$ clearly yields
    \begin{align*}
        \lVert\theta\rVert_F&\leq \frac{1}{\sigma_{\min}(R)}\Bigg(\sqrt{\frac{\lVert R\rVert_F(J_1(\theta)- J_1(\theta^*))}{\sigma_{\min}(M)}}+\frac{\lVert D\rVert_F J_1(\theta)}{\sigma_{\min}(M)}\Bigg)\\
        &\leq \frac{1}{\sigma_{\min}(R)}\Bigg(\sqrt{\frac{\lVert R\rVert_F(\ell- J_1(\theta^*))}{\sigma_{\min}(M)}}+\frac{\lVert D\rVert_F\ell}{\sigma_{\min}(M)}\Bigg).
    \end{align*}
\ep

\vspace{1mm}

\noindent\textbf{New operators.}
 \text{Taking a given $d\times d$ stable matrix $S$ and a symmetric matrix $X$,} we define the two operators $\mathcal{T}^\cdot$ and $\mathcal{F}^\cdot$ by
\begin{align} 
\mathcal{T}^S(X) \; = \; \int_0^\infty e^{Sr} X e^{S\trans r}\d r, & \qquad 
\mathcal{F}^S(X) \; = \; SX+XS\trans. 
\end{align} 

\noindent It is known that 
$$
\mathcal{F}^S \circ\mathcal{T}^S+I=0
$$ 
\noindent which means that $\mathcal{T}^S=-(\mathcal{F}^S)^{-1}$.
\begin{Lemma}\label{GE}
    For all $H \in \S^d_{>+}$, it holds 
    $$
    \interleave \mathcal{T}^S\interleave_F\leq \frac{\lVert \mathcal{T}^S(H)\rVert_F}{\sigma_{\min}(H)}.
    $$
\end{Lemma}
\begin{proof}
    For any unit vector $v\in\mathbb{R}^d$ and any unit spectral norm matrix $X$, it holds
\begin{equation*}
    \begin{split}
        v\trans \mathcal{T}^S(X) v&=\int_0^\infty \textnormal{tr}(Xe^{S\trans r}vv\trans e^{S r})\d r\\
        &\leq \int_0^\infty \textnormal{tr}(He^{S\trans r}vv\trans e^{Sr})\d r \lVert H^{-1/2}XH^{-1/2}\rVert_F\\
        &=v\trans \mathcal{T}^S(H) v\cdot\lVert H^{-1/2}XH^{-1/2}\rVert_F\\
        &\leq\lVert \mathcal{T}^S(H)\rVert_F \sigma_{\min}^{-1}(H)
    \end{split}
\end{equation*}

\noindent which clearly yields the conclusion.
\end{proof}
\begin{Lemma}
    It holds 
    $$
    \interleave \mathcal{T}^{S}\interleave_F=\interleave\mathcal{T}^{S\trans}\interleave_F.
    $$
\end{Lemma}
\begin{proof}
From the commutative property of the trace operator, for any unit vector $v\in\mathbb{R}^d$ and any unit spectral norm matrix $X$, one has
\begin{equation*}
    \begin{split}
        v\trans\mathcal{T}^{S\trans}(X)v&=\int_0^\infty \textnormal{tr}(Xe^{Sr}vv\trans e^{S\trans r})\d r\\
        &=\textnormal{tr}(X\mathcal{T}^S(vv\trans))\\
        &\leq \lVert X\rVert_F\lVert \mathcal{T}^S(vv\trans) \rVert_F \\
        &\leq \lVert X\rVert_F\interleave \mathcal{T}^{S\trans}\interleave_F\lVert vv\trans \rVert_F=\interleave \mathcal{T}^{S\trans}\interleave_F
    \end{split}
\end{equation*}

\noindent which yields 
 $$
 \interleave \mathcal{T}^{S\trans}\interleave_F\leq \interleave\mathcal{T}^{S}\interleave_F.
 $$
 Noting that $S=(S\trans)\trans$, the reverse inequality directly follows.
\end{proof}

\noindent In what follows, we introduce the notations 
$$
\Xi_\theta =B-\frac{\beta}{2}I_d+D\theta \quad \mbox{ and } \quad \hat\Xi_\zeta=\hat B-\frac{\beta}{2}I_d+D\zeta.
$$

\subsubsection{Lipschitz regularity of $\Sigma_\theta$, $\hat\Sigma_\zeta$, $K_\theta$ and $\Lambda_\zeta$.}

\noindent\textbf{Perturbation analysis for $\Sigma_\theta$ and $\hat\Sigma_\zeta$}\\

Recalling that $\Sigma_\theta$ and $\hat\Sigma_\zeta$ are the unique solution to the ALE \eqref{ALEForSigmaTheta}, from Lemma \ref{SolALE}, it holds 
$$
\Sigma_\theta=\mathcal{T}^{\Xi_\theta}(M),  \quad \hat\Sigma_\zeta=\mathcal{T}^{\hat{\Xi}_\zeta}(\hat M).
$$

Hence, from Lemma \ref{GE}, we directly get
\begin{equation}\label{IL}
\interleave\mathcal{T}^{\Xi_\theta}\interleave_F\leq\frac{\lVert \Sigma_\theta\rVert_F}{\sigma_{\min}(M)}, \quad \interleave\mathcal{T}^{\hat\Xi_\zeta}\interleave_F\leq\frac{\lVert \hat\Sigma_\zeta\rVert_F}{\sigma_{\min}(\hat M)}.
\end{equation}

\begin{Proposition}\label{PerturbationSigmaTheta}
For all $\theta , \, \theta' \in \Sc (\ell)$, it holds
$$
\lVert \Sigma_{\theta'}-\Sigma_{\theta}\rVert_F\leq Lip_{\Sigma}(\ell)\lVert\theta'-\theta\rVert_F,
$$ 
\noindent where $Lip_{\Sigma}(\ell)= \frac{2\lVert D\rVert_F(\Bd_{\Sigma}(\ell))^2}{\sigma_{\min}(M)}$.
Similarly, for all $\zeta ,\zeta' \in \hat\Sc (\hat \ell)$, it holds
$$
\lVert \hat\Sigma_{\zeta'}-\hat\Sigma_{\zeta}\rVert_F\leq Lip_{\hat\Sigma}(\hat \ell)\lVert\zeta'-\zeta\rVert_F,
$$ 
\noindent where $Lip_{\hat\Sigma}(\hat \ell)= \frac{2\lVert D\rVert_F(\Bd_{\hat\Sigma}(\hat \ell))^2}{\sigma_{\min}(\hat M)}$.
\end{Proposition}
\begin{proof}

From 
\begin{align*}
    \lVert\mathcal{F}^{\Xi_{\theta'}}(X)-\mathcal{F}^{\Xi_{\theta}}(X)\rVert_F&=\lVert D(\theta'-\theta)X+X(\theta'-\theta)\trans D\trans\rVert\\
    &\leq 2\lVert D\rVert_F\lVert X\rVert_F\lVert \theta'-\theta\rVert_F
\end{align*}
we get 
$$
\interleave\mathcal{F}^{\Xi_{\theta'}}-\mathcal{F}^{\Xi_{\theta}}\interleave_F \leq 2\lVert D\rVert_F \lVert \theta'-\theta\rVert_F.
$$

Then, using Proposition \ref{Bds} and the inequality \eqref{IL}, we get
$$
\interleave\mathcal{T}^{\Xi_\theta}\interleave_F\leq\frac{\lVert \Sigma_\theta\rVert_F}{\sigma_{\min}(M)}\leq\frac{\Bd_{\Sigma}(\ell)}{\sigma_{\min}(M)}.
$$

\noindent
Since $\mathcal{F}^{\Xi_\theta}(\Sigma_{\theta})=\mathcal{F}^{\Xi_{\theta'}}(\Sigma_{\theta'})=-M$
we have
\begin{align*}
    \lVert \Sigma_{\theta'}-\Sigma_{\theta}\rVert_F&=\lVert\mathcal{T}^{\Xi_\theta}(\mathcal{F}^{\Xi_\theta}(\Sigma_{\theta'}-\Sigma_{\theta}))\rVert_F\\
    &=\lVert\mathcal{T}^{\Xi_\theta}(\mathcal{F}^{\Xi_\theta}(\Sigma_{\theta'})-\mathcal{F}^{\Xi_\theta}(\Sigma_{\theta}))\rVert_F\\
    &=\lVert\mathcal{T}^{\Xi_\theta}(\mathcal{F}^{\Xi_\theta}(\Sigma_{\theta'})-\mathcal{F}^{\Xi_{\theta'}}(\Sigma_{\theta'}))\rVert_F\\
&\leq\interleave\mathcal{T}^{\Xi_\theta}\interleave_F\interleave \mathcal{F}^{\Xi_{\theta'}}-\mathcal{F}^{\Xi_{\theta}}\interleave_F\lVert\Sigma_{\theta'}\rVert_F\\
&\leq \frac{2\lVert D\rVert_F(\Bd_{\Sigma}(\ell))^2}{\sigma_{\min}(M)}\lVert \theta'-\theta\rVert_F.
\end{align*}

The Lipschitz regularity of $\hat\Sigma_{\zeta}$ is handled by similar arguments. The remaining technical details are omitted. 
\end{proof}

\noindent\textbf{Perturbation analysis of $K_\theta$ and $\Lambda_\zeta$.}
\begin{Proposition}\label{AnaPertKtheta}
    For all $\theta,\theta'\in\Sc(\ell)$, one has
    $$\lVert K_{\theta'}-K_\theta\rVert_F\leq Lip_{K}(\ell)\lVert \theta'-\theta\rVert_F$$ where 
    $$
    Lip_{K}(\ell)=2\big(\Bd_{E}(\ell)+\lVert R\rVert_F \Bd_\theta(\ell)\big) \sup_{\theta\in\Sc(\ell)} \lVert \int_0^\infty e^{\Xi_{\theta}t} e^{(\Xi_{\theta})\trans t}\d t\rVert_F<+\infty.
    $$
    Similarly, one can show that for all $\zeta,\zeta'\in\hat\Sc(\hat a)$, it holds
$$
\lVert \Lambda_{\zeta'}-\Lambda_\zeta\rVert_F\leq Lip_{\Lambda}(\hat \ell)\lVert \zeta'-\zeta\rVert_F$$ where 
    $$Lip_{\Lambda}(\hat \ell)=2\big(\Bd_{\hat E}(\hat \ell)+\lVert R\rVert_F \Bd_\zeta(\hat \ell)\big) \sup_{\zeta\in\hat\Sc(\hat \ell)} \lVert \int_0^\infty e^{\hat\Xi_{\zeta}t} e^{(\hat\Xi_{\zeta})\trans t}\d t\rVert_F<+\infty.
$$
   
\end{Proposition}
\begin{proof}

From Lemma \ref{DiffKtheta}, we get
$$
K_{\theta'}-K_{\theta}=\int_0^\infty e^{\Xi_{\theta'}\trans t}[(E_{\theta})\trans (\theta'-\theta)+(\theta'-\theta)\trans E_{\theta}+(\theta'-\theta)\trans R(\theta'-\theta)]e^{\Xi_{\theta'}t} \, \dt
$$

\noindent so that
$$
\lVert K_{\theta'}-K_{\theta}\rVert_F\leq \lVert (E_{\theta})\trans (\theta'-\theta)+(\theta'-\theta)\trans E_{\theta}+(\theta'-\theta)\trans R(\theta'-\theta)\rVert_F \lVert \int_0^\infty e^{\Xi_{\theta'}t} e^{\Xi_{\theta'}\trans t}\d t\rVert_F.
$$
\normalsize
The compactness of $\Sc(\ell)$ together with the continuity of $\theta\in \Sc(\ell)\mapsto \lVert \int_0^\infty e^{\Xi_{\theta}t} e^{\Xi_{\theta}\trans t}\d t\rVert_F$ yields $\sup_{\theta\in \Sc(\ell)} \lVert \int_0^\infty e^{\Xi_{\theta}t} e^{\Xi_{\theta}\trans t}\d t\rVert_F<+\infty$. The triangle inequality together with Proposition \ref{Bds} gives 
$$
\lVert \theta'-\theta\rVert_F\leq \lVert \theta'\rVert_F+\lVert \theta\rVert_F\leq 2\Bd_\theta(\ell)
$$
so that
$$
\lVert E_{\theta}\trans (\theta'-\theta)+(\theta'-\theta)\trans E_{\theta}+(\theta'-\theta)\trans R(\theta'-\theta)\rVert_F \leq 2\big(\Bd_{E}(\ell)+\lVert R\rVert_F \Bd_\theta(\ell)\big)\lVert \theta'-\theta\rVert_F.
$$

Gathering the previous bounds allows to conclude that $\theta\mapsto K_\theta$ is $Lip_K(\ell)$-Lipschitz continuous on $\Sc(\ell)$. Similar arguments yield the Lipschitz regularity of $\Lambda$. We omit the remaining technical details.

\end{proof}

\subsubsection{Proof of Proposition \ref{LJ}}
In order to establish the Lipschitz regularity of $\nabla   J_1$ and $\nabla J_2$ on $\Sc(\ell)$, we start from the representation provided by Proposition \ref{GradJ}. Namely, one has
$$
\nabla    J_1(\theta)=2E_\theta\Sigma_\theta
$$ 

\noindent where $E_\theta=R\theta+D\trans K_\theta$. Hence,
$$
\nabla J_1(\theta')-\nabla  J_1(\theta)=2(E_{\theta'}-E_\theta)\Sigma_{\theta'}+2E_{\theta}(\Sigma_{\theta'}-\Sigma_\theta).
$$

 Noting that
$$
E_{\theta'}-E_\theta=R(\theta'-\theta)+D\trans (K_{\theta'}-K_\theta)
$$
\noindent and using Proposition \ref{AnaPertKtheta}, we obtain 
\begin{align*}
    \lVert E_{\theta'}-E_\theta\rVert_F&\leq \lVert R\rVert_F\lVert\theta'-\theta\rVert_F+\lVert D\rVert_F \lVert K_{\theta'}-K_\theta\rVert_F\\
    &\leq (\lVert R\rVert_F+\lVert D\rVert_F Lip_{K}(\ell))\lVert \theta'-\theta\rVert_F.
\end{align*}

The previous bound, along again with Proposition \ref{AnaPertKtheta}, provides
\begin{align*}
    \lVert\nabla  J_1(\theta')-\nabla  J_1(\theta)\rVert_F&\leq 2\lVert E_{\theta'}-E_\theta\rVert_F\lVert\Sigma_{\theta'}\rVert_F+2\lVert E_{\theta}\rVert_F\lVert\Sigma_{\theta'}-\Sigma_\theta\rVert_F\\
    &\leq 2 ( (\lVert R\rVert_F+\lVert D\rVert_F Lip_{K}(\ell)) \Bd_{\Sigma}(\ell)+\Bd_{E}(\ell) Lip_{\Sigma}(\ell))\lVert \theta'-\theta\rVert_F.
\end{align*}


\noindent Using similar arguments, we find that for all $\zeta,\zeta'\in\hat\Sc(\hat \ell)$, 
$$ \lVert\nabla  J_2(\zeta')-\nabla  J_2(\zeta)\rVert_F\leq \hat L(\hat \ell)\lVert \zeta'-\zeta\rVert_F$$ 

\noindent where 
$$
\hat L(\hat \ell)=2 ( (\lVert R\rVert_F+\lVert D\rVert_F Lip_{\Lambda}(\hat \ell)) \Bd_{\hat \Sigma}(\hat \ell)+\Bd_{\hat E}(\hat \ell) Lip_{\hat \Sigma}(\hat \ell)).
$$


\section{Proofs of the results for the model-free algorithm}\label{proof:main:theorem:model:free}

In the first subsection, we describe the strategy of the proof of Theorem \ref{ThmConvSGD}. In Subsection \ref{main:proof:theorem:conv:sgd}, we provide its proof. Several auxiliary results are postponed to the next subsections.

\subsection{Strategy of proof of Theorem \ref{ThmConvSGD}}\label{strategy:proof:subsection}
We first solve the well-posedness problem of the perturbation of the parameters. The following lemma is directly taken from Lemma 4 of \cite{ConvSampGradMethod}.

\begin{Lemma}\label{ra}
    There exists $\check{r}(b)$ depending only upon $b\in\R_+$, $m$, $d$ and the model parameters such that for all $0<r\leq \check{r}(b)$, for any $\mathbf{U}=(U,V)\in H,$ such that $\lVert U\rVert_F\leq r,\lVert V\rVert_F\leq r$ and for all $\Theta=(\theta,\zeta)\in\Rc(b)$, we have 
    \begin{align} 
    \Theta+\mathbf U &= \; (\theta+  U,\zeta+ V) \; \in \;  \Rc(2b).
    \end{align} 
\end{Lemma}

From now on, we assume that $r>0$ is small enough ($r<\check{r}(b)$, $\check{r}(b)$ defined in Lemma \ref{ra}), so that for all $\Theta=(\theta,\zeta)\in\Rc(b)$, for all $U,V$ with $\lVert U\rVert_F=\lVert V\rVert_F=r, (\theta+U,\zeta+V)\in\Rc(2b)$. 

\vspace{1mm}

The main ingredients of the proof can be described as follows.  Denoting by $\tilde \nabla^{\Delta,N,pop} J(\Theta)=(\tilde \nabla_\theta^{\Delta,N,pop} J(\Theta),\tilde \nabla_\zeta^{\Delta,N,pop} J(\Theta))\in H$ where $\tilde \nabla_\theta^{\Delta,N,pop} J(\Theta)$ and $\tilde \nabla_\zeta^{\Delta,N,pop} J(\Theta)$ are the results of Algorithm \ref{Gradientestim}, the key idea is to show that the two following events 
\begin{equation}\label{EventsM1M2}
    \begin{aligned}
            \M_1&=\Bigg\{\langle \tilde \nabla^{\Delta,N,pop} J(\Theta),\nabla  J(\Theta)\rangle_H\geq \nu_1\lVert \nabla J(\Theta)\rVert_H^2\Bigg\}\\
  \M_2&=\Bigg\{\lVert \tilde \nabla^{\Delta,N,pop} J(\Theta)\rVert_H^2\leq \nu_2\lVert\nabla J(\Theta)\rVert_H^2\Bigg\}
    \end{aligned}
\end{equation}
occur together with high probability, for some $\nu_1,\nu_2>0$. More precisely, by writing
\begin{align*}
   \langle \tilde \nabla^{\Delta,N,pop} J(\Theta),\nabla  J(\Theta)\rangle_H&=\lVert \nabla J(\Theta)\rVert_H^2+\langle  (\tilde \nabla^{\Delta,N,pop}-\nabla)  J(\Theta),\nabla  J(\Theta)\rangle_H\\
    &{\geq}\lVert \nabla J(\Theta)\rVert_H^2 -\lVert ( \tilde \nabla^{\Delta,N,pop} - \nabla )  J(\Theta)\rVert_H\lVert \nabla  J(\Theta)\rVert_H
\end{align*}
\noindent and
\begin{align*}
    \lVert  \tilde \nabla^{\Delta,N,pop} J(\Theta)\rVert_H&\leq \lVert \nabla J(\Theta)\rVert_H+ \lVert  (\tilde \nabla^{\Delta,N,pop} -\nabla ) J(\Theta)\rVert_H\\
\end{align*}
we notice that if the event
\begin{equation}\label{DesiredErrorGradientE}
    \lVert  (\tilde \nabla^{\Delta,N,pop} -\nabla ) J(\Theta)\rVert_H\leq\frac{1}{2} \lVert \nabla  J(\Theta)\rVert_H
\end{equation}

\noindent occurs with high probability, then so does $\M_1 \cap \M_2$ with $\nu_1=\frac{1}{2}$ and $\nu_2=\frac{9}{4}$.

 Next, for any fixed $\varepsilon>0$, the P.-L. inequality \eqref{GD} guarantees that if $\Theta\in \Rc$ is such that $J(\Theta)-J(\Theta^*)>\varepsilon$, then \eqref{DesiredErrorGradientE} is satisfied as soon as 
 $$ 
 \lVert  (\tilde \nabla^{\Delta,N,pop} -\nabla)  J(\Theta)\rVert_H\leq\varepsilon'
 $$ 
 
 \noindent with $\varepsilon'=\frac{1}{2}\sqrt{\frac{\varepsilon}{\kappa}}$. Indeed, it suffices to notice that
$$
   \lVert  (\tilde \nabla^{\Delta,N,pop} -\nabla ) J(\Theta)\rVert_H\leq\varepsilon'=\frac{1}{2}\sqrt{\frac{\varepsilon}{\kappa}}\leq\frac{1}{2}\sqrt{\frac{J(\Theta)-J(\Theta^*)}{\kappa} }\leq\frac{1}{2}\lVert\nabla J(\Theta)\rVert_H.
$$

Noting that
$$
    \lVert  (\tilde \nabla^{\Delta,N,pop} -\nabla)  J(\Theta)\rVert_H=\sqrt{\lVert (\tilde \nabla^{\Delta,N,pop}_\theta -\nabla_\theta)  J(\Theta)\rVert_F^2+\lVert (\tilde \nabla^{\Delta,N,pop}_\zeta -\nabla_\zeta)  J(\Theta)\rVert_F^2 }
$$

\noindent we will show that for all $\varepsilon>0$, the event
$$
\lVert (\tilde\nabla^{\Delta,N,pop}_\theta -\nabla_\theta)  J(\Theta)\rVert_F\leq\varepsilon \quad \mbox{and} \quad\lVert (\tilde \nabla^{\Delta,N,pop}_\zeta -\nabla_\zeta)  J(\Theta)\rVert_F\leq \varepsilon
$$
\noindent holds with high probability as soon as the parameters $T,n,N,\tilde N$ are large enough and $r$ is small enough.\\

\subsection{Proof of Theorem \ref{ThmConvSGD}}\label{main:proof:theorem:conv:sgd}

\noindent \emph{Step 1: } The central idea is to use the following decomposition
    \begin{align*}
        (\tilde  \nabla^{\Delta,N,pop}_\theta -\nabla_\theta)  J(\Theta) & =(\tilde \nabla^{\Delta,N,pop}_\theta -\hat\nabla^{\Delta,N,pop}_\theta)  J(\Theta)+ (\hat\nabla^{\Delta,N,pop}_\theta -\hat\nabla^{(T),\Delta}_\theta)  J(\Theta)\\
        &+ (\hat\nabla^{(T),\Delta}_\theta -\hat\nabla^{(T)}_\theta)  J(\Theta)+(\hat\nabla^{(T)}_\theta -\hat\nabla_\theta) J(\Theta)\\
        &+ (\hat\nabla_\theta -\nabla_\theta)  J(\Theta)\\
        & =: \rm{E}^{(1)}_\theta+ \rm{E}^{(2)}_\theta+\rm{E}^{(3)}_\theta+\rm{E}^{(4)}_\theta+\rm{E}^{(5)}_\theta,
    \end{align*}
\noindent and similarly
\begin{align*}
           \lVert (\tilde \nabla^{\Delta,N,pop}_\zeta -\nabla_\zeta)  J(\Theta)\rVert_F& =  (\tilde \nabla^{\Delta,N,pop}_\zeta -\hat\nabla^{\Delta,N,pop}_\zeta)  J(\Theta)+ (\hat\nabla^{\Delta,N,pop}_\zeta -\hat\nabla^{(T),\Delta}_\zeta)  J(\Theta)\\
        & + (\hat\nabla^{(T),\Delta}_\zeta  -\hat\nabla^{(T)}_\zeta)  J(\Theta)+ (\hat\nabla^{(T)}_\zeta  -\hat\nabla_\zeta) J(\Theta)\\
        & +  (\hat\nabla_\zeta-\nabla_\zeta)  J(\Theta)\\
        & =: \rm{E}^{(1)}_\zeta + \rm{E}^{(2)}_\zeta+\rm{E}^{(3)}_\zeta+\rm{E}^{(4)}_\zeta+\rm{E}^{(5)}_\zeta,
    \end{align*}

\noindent where we introduced the following notations which will be useful for our convergence analysis:

\noindent $\bullet$ The perturbed policy gradients:
    $$
    \hat \nabla_\theta J(\Theta)=\frac{d}{r^2}\frac{1}{\tilde N}\sum_{i=1}^{\tilde N} J(\Theta_i) U_i \quad \mbox{ and } \quad \hat \nabla_\zeta J(\Theta)=\frac{d}{r^2}\frac{1}{\tilde N}\sum_{i=1}^{\tilde N} J(\Theta_i) V_i.
    $$
\noindent$\bullet$ The truncated policy gradients:
     $$
     \hat \nabla^{(T)}_\theta J(\Theta)=\frac{d}{r^2}\frac{1}{\tilde N}\sum_{i=1}^{\tilde N} J^{(T)}(\Theta_i) U_i \quad \mbox{ and }\quad \hat \nabla^{(T)}_\zeta J(\Theta)=\frac{d}{r^2}\frac{1}{\tilde N}\sum_{i=1}^{\tilde N} J^{(T)}(\Theta_i) V_i.
     $$
\noindent$\bullet$ The  policy gradients with time discretization:
     $$
     \hat \nabla^{(T),\Delta}_\theta J(\Theta)=\frac{d}{r^2}\frac{1}{\tilde N}\sum_{i=1}^{\tilde N} J^{(T),\Delta}(\Theta_i) U_i \quad \mbox{ and } \quad \hat \nabla^{(T)}_\zeta J(\Theta)=\frac{d}{r^2}\frac{1}{\tilde N}\sum_{i=1}^{\tilde N} J^{(T),\Delta}(\Theta_i) V_i.
     $$
\noindent $\bullet$ The output of Algorithm \ref{Gradientestim} averaged over the randomness of the state trajectories:
\begin{equation}\label{grad:hat:pop:expression}
\begin{aligned}
      \hat \nabla^{\Delta,N,pop}_\theta J(\Theta) & =\E[\tilde  \nabla^{\Delta,N,pop}_\theta J(\Theta)| \underline{U},\underline{V}], \\ 
      \hat \nabla^{\Delta,N,pop}_\zeta J(\Theta) & =\E[\tilde  \nabla^{\Delta,N,pop}_\zeta J(\Theta)| \underline{U},\underline{V}],
\end{aligned}
\end{equation}
      
\noindent where $\underline{U}=(U_1,\dots,U_{\tilde N})$ and $\underline{V}=(V_1,\dots,V_{\tilde N})$.

In the above notations, given a finite horizon $T>0$ (to be chosen later on) and $\Theta=(\theta,\zeta)\in H$, we introduced the following time-truncated expected cost functional
\begin{equation}\label{TruncFC}
    \begin{aligned}
         J^{(T)}(\Theta)
    & = \E\Bigg[\int_0^T e^{-\beta t} \big(  (Y^\Theta_t )\trans QY^\Theta_t + (Z^\Theta_t)\trans \hat Q Z^\Theta_t\\
    & +\int \big(a\trans Ra + \lambda \log p^\Theta(Y^\Theta_t, Z^\Theta_t,a)\big) p^\Theta(Y^\Theta_t, Z^\Theta_t,a)\d a\big)  \d t \Bigg],
    \end{aligned}
\end{equation}

\noindent and also introduced the corresponding time discretized cost functional, with action execution noise
\begin{equation}\label{DiscApproFC}
    \begin{aligned}
        J^{(T),\Delta}(\Theta)&={h}\E\Bigg[\sum_{l=0}^{n-1} e^{-\beta t_l}\Big((\xb^{\Theta,\Delta}_{t_l}-\E_0[\xb^{\Theta,\Delta}_{t_l}])\trans Q (\xb^{\Theta,\Delta}_{t_l}-\E_0[\xb^{\Theta,\Delta}_{t_l}])+\E_0[\xb^{\Theta,\Delta}_{t_l}]\trans\hat Q\E_0[\xb^{\Theta,\Delta}_{t_l}]\\
        &+(\alpha^{\Theta,\Delta}_{t_l})\trans R \alpha^{\Theta,\Delta}_{t_l}+\lambda\log p^\Theta(\xb^{\Theta,\Delta}_{t_l}-\E_0[\xb^{\Theta,\Delta}_{t_l}],\E_0[\xb^{\Theta,\Delta}_{t_l}],\alpha^{\Theta,\Delta}_{t_l})\Big)\Bigg],
    \end{aligned}
\end{equation}

\noindent where $(\xb^{\Theta,\Delta}_{t_l})_{0\leq l\leq n}$ is the time-discretization scheme of $(X^\Theta_t)_{t\in[0,T]}$, over the same time grid $\Delta$ as the interacting agents dynamics \eqref{XDeltThet}, with dynamics
\begin{equation}\label{xDeltThet}
\begin{cases} 
    \xb^{\Theta,\Delta}_{t_{l+1}}  & = \; \xb^{\Theta,\Delta}_{t_l}+(B\xb^{\Theta,\Delta}_{t_l}+\bar B \E_0[\xb^{\Theta,\Delta}_{t_l}]+D\alpha^{\Theta,\Delta}_{t_l})) {{h}} \\
    & \qquad \qquad + \; \gamma\sqrt{{h}} \bw_l+\gamma_0 \sqrt{{h}}\bw^0_l\; , \\
    \xb^{\Theta,\Delta}_{0} & = \; X_0
\end{cases}     
\end{equation}
with $\alpha^{\Theta,\Delta}_{t_l}=\theta(\xb^{\Theta,\Delta}_{t_l}-\E_0[\xb^{\Theta,\Delta}_{t_l}])+ \zeta \E_0[\xb^{\Theta,\Delta}_{t_l}]+\sqrt{\frac{\lambda}{2}R^{-1}}\xib_{t_l}$, $(\xib_{t_l})_{0\leq l \leq n-1}$ being \emph{i.i.d.} random variables with law $\mathcal{N}(0,\I_m)$.



\bigskip

\noindent \emph{Step 2: }  Denoting by $\tilde\varepsilon=\frac{\varepsilon'}{5\sqrt{2}}=\frac{1}{10\sqrt{2}}\sqrt{\frac{\varepsilon}{\kappa}}$, choosing $r<\min(\check r(b),\frac{1}{h_r(b,1/\tilde\varepsilon)},  \frac{1}{h'_r(b,1/\tilde\varepsilon)})$ and $\tilde N>h_{\tilde N}(b,\frac{1}{\varepsilon})$, Proposition \ref{PropPerb} guarantees that with probability $1-4(\frac{d}{\tilde\varepsilon})^{-d}$ 
    $$
   \lVert \rm{E}^{5}_\theta \rVert_F \leq \tilde{\varepsilon} \quad \mbox{ and } \quad \lVert \rm{E}^{5}_\zeta \rVert_F\leq \tilde{\varepsilon}. 
    $$

Choosing $T=\breve{T}(\varepsilon)$ as in the statement of Theorem \ref{ThmConvSGD}, according to Proposition \ref{PropTrunc}, it holds
$$
 \lVert \rm{E}^{4}_\theta \rVert_F \leq \tilde{\varepsilon} \quad \mbox{ and } \quad \lVert \rm{E}^{4}_\zeta \rVert_F \leq \tilde{\varepsilon}, \quad \mathbb{P}-a.s.
$$

Next choosing $n=\breve{n}(\varepsilon)$ as in the statement of Theorem \ref{ThmConvSGD}, it follows from Proposition \ref{PropDiscre} that 
$$
 \lVert \rm{E}^{3}_\theta \rVert_F\leq \tilde{\varepsilon} \quad \mbox{ and } \quad \lVert \rm{E}^{3}_\zeta \rVert_F\leq \tilde{\varepsilon}, \quad \mathbb{P}-a.s.
$$
    
Choosing $N=\breve{N}(\varepsilon)$ large enough (again as in the statement of Theorem \ref{ThmConvSGD}), according to estimates \eqref{estimD6} in Proposition \ref{PropN}, it holds
$$
 \lVert \rm{E}^{2}_\theta \rVert_F \leq \tilde{\varepsilon} \quad \mbox{ and } \quad \lVert \rm{E}^{2}_\zeta \rVert_F \leq \tilde{\varepsilon}, \quad \mathbb{P}-a.s.
$$

Finally, according to Proposition \ref{PropMCError}, up to a modification of $\tilde N$, namely taking $\tilde N$ large enough such that 
$$
\tilde N > \max\{h_{\tilde N}(b,\frac{1}{\breve{r}(\varepsilon)},\frac{1}{\varepsilon}), {h'_{\tilde N}(b,\frac{1}{\breve{r}(\varepsilon)}, \breve{T}(\varepsilon),\breve{n}(\varepsilon),\breve{N}(\varepsilon),1/\tilde\varepsilon)}\}
$$  
\noindent one has
$$
 \lVert \rm{E}^{1}_\theta \rVert_F\leq \tilde \varepsilon \quad \mbox{ and } \quad \lVert \rm{E}^{1}_\zeta \rVert_F\leq \tilde \varepsilon
$$ 

\noindent with probability  $1-(\frac{d}{\tilde\varepsilon})^{-d}$.
    
 Putting the above estimates together yields
     \begin{align*}
        \lVert (\tilde \nabla^{\Delta,N,pop}_\theta -\nabla_\theta)  J(\Theta)\rVert_F \leq \frac{\varepsilon'}{\sqrt{2}} \quad \mbox{ and } \quad
           \lVert (\tilde \nabla^{\Delta,N,pop}_\zeta -\nabla_\zeta)  J(\Theta)\rVert_F\leq \frac{\varepsilon'}{\sqrt{2}}
    \end{align*}
    so that 
    $$
    \lVert (\tilde \nabla^{\Delta,N,pop} - \nabla)  J(\Theta)\rVert_H\leq\varepsilon'
    $$
\noindent with probability at least $1-4(\frac{d}{\tilde\varepsilon})^{-d}$.

\noindent \emph{Step 3: } As discussed in Subsection \ref{strategy:proof:subsection}, the previous inequality implies that the probability of $\M_1 \cap \M_2$ defined in \eqref{EventsM1M2} with $\nu_1=\frac{1}{2}$ and $\nu_2=\frac{9}{4}$ is not smaller than $1-4(\frac{d}{\tilde\varepsilon})^{-d}$.

\noindent It then follows from Lemma \ref{ConvStoGD} that for any $\rho\in (0,\frac{4}{9 \check L(b)})$, still with the probability not smaller than $1-4(\frac{d}{\tilde\varepsilon})^{-d}$, one has
$$
\Theta_{[1]}= \Theta_{[0]}-\rho  \tilde \nabla^{\Delta,N,pop} J(\Theta_{[0]})\in\Rc(b)
$$
\noindent and
$$
J(\Theta_{[1]})-J(\Theta^*)\leq \Big( 1-\frac{\rho(4-9\rho \check L(b))}{8\bar\kappa(b)}\Big)(J(\Theta_{[0]})-J(\Theta^*)).
$$

\noindent By induction, we eventually deduce that
$$J(\Theta_{[k]})-J(\Theta^*)\leq \big( 1-\frac{\rho(4-9\rho \check L(b))}{8\bar\kappa(b)}\big)^k(J(\Theta_{[0]})-J(\Theta^*))$$

\noindent with probability at least $1-4k(\frac{d}{\tilde\varepsilon})^{-d}$. The conclusion of Theorem \ref{ThmConvSGD} easily follows from the previous inequality.

\subsection{Analysis of the single step GD algorithm}\label{ConvSGD}
In this section, we analyze the convergence of a general single-step GD algorithm. This part serves as an important building block in the proof of Theorem \ref{ThmConvSGD}. Our main result is Proposition \ref{ConvStoGD}. We first prove some technical results.

\begin{Lemma}
    The map $(\theta,\zeta)\mapsto \nabla J(\Theta)$ is $\check{L}(b)-$Lipschitz continuous on $\Rc(b)$, that is, for all $\Theta=(\theta,\zeta),\Theta'=(\theta',\zeta')\in \Rc(b)$,
    $$
    \lVert \nabla J(\Theta)-\nabla J(\Theta')\rVert_H \leq \check{L}(b)\lVert \Theta-\Theta'\rVert_H,
    $$
\noindent recalling that $\check L(b)=\max\{L(b),\hat L(b)\}$,  where $L(b)$ and $\hat L(b)$ are the Lipschitz constants of $\nabla J_1(\theta)$ and $\nabla J_2(\zeta)$ on $\Sc(b)$ and $\hat\Sc(b)$ respectively, defined in \eqref{La} of Proposition \ref{LJ}.
\end{Lemma}
\begin{proof}
    It is clear that $\Theta\in\Rc(b)$ implies $(\theta, \zeta) \in\Sc(b) \times \hat\Sc(b)$ so that
    $$\lVert \nabla J_1(\theta)-\nabla J_1(\theta')\rVert_F\leq L(b)\lVert \theta-\theta'\rVert_F\leq \check L(b)\lVert \theta-\theta'\rVert_F$$
     $$\lVert \nabla J_2(\zeta)-\nabla J_2(\zeta')\rVert_F\leq \hat L(b)\lVert \zeta-\zeta'\rVert_F\leq \check L(b)\lVert \zeta-\zeta'\rVert_F$$
     which in turn clearly yields
     \begin{align*}
         \lVert \nabla J(\Theta)-\nabla J(\Theta')\rVert^2_H&=  \lVert \nabla_\theta J(\theta,\zeta)-\nabla_\theta J(\theta',\zeta')\rVert^2_F+ \lVert \nabla_\zeta J(\theta,\zeta)-\nabla_\zeta J(\theta',\zeta')\rVert^2_F\\
         &=\lVert \nabla J_1(\theta)-\nabla J_1(\theta')\rVert^2_F+\lVert \nabla J_2(\zeta)-\nabla J_2(\zeta')\rVert^2_F\\
         &\leq  \check{L}^2(b) (\lVert \theta-\theta'\rVert^2_F+\lVert \zeta-\zeta'\rVert^2_F) \; = \; 
          \check{L}^2(b)  \lVert \Theta-\Theta'\rVert^2_H.
     \end{align*}
The proof is now complete.
\end{proof}

For $\Theta=(\theta,\zeta)\in\Rc(b)$ and $\Gb=(G_1,G_2)\in H$ satisfying 
\begin{equation}\label{condition:perturbation}
            \langle \Gb,\nabla J(\Theta)\rangle_H \; \geq \;  \nu_1\lVert \nabla   J(\Theta)\rVert_H^2, \qquad 
          \lVert \Gb\rVert_H^2 \; \leq \;  \nu_2\lVert \nabla  J(\Theta)\rVert_H^2, 
\end{equation}
we set 
\begin{equation}\label{def:theta:rho}
\Theta_{\rho}:=\Theta-\rho \Gb=(\theta-\rho G_1,\zeta-\rho G_2).
\end{equation}

\begin{Lemma}
For any $\Theta=(\theta,\zeta)\in\Rc(b)$, any $\Gb\in H$ satisfying \eqref{condition:perturbation} and any $\rho\in(0,\frac{2\nu_1}{\nu_2 \check L(b)})$, $\Theta_\rho$ defined by \eqref{def:theta:rho} satisfies
$$
\Theta_{\rho}\in \Rc(b).
$$

\end{Lemma}
\begin{proof}
    We let 
    $$
    \rho_{\max}=\sup\{\rho'\geq 0|\Theta_{\rho}:=\Theta-\rho \Gb=(\theta-\rho G_1,\zeta-\rho G_2)\in\Rc (b ),\forall \rho \in[0,\rho']\}.
    $$

 \noindent Firstly, from the first condition of $\Gb$,  $\Gb$ is also a descent direction of the function $ J$, we have $\rho_{\max}>0$.  Next, $\Rc(b)$ is compact thus bounded, we have $\rho_{\max}<+\infty$. 

Assume that $\rho_{\max}<\frac{2\nu_1}{\nu_2 \check L(b)}$. By the continuity of $ \rho\mapsto J\big(\theta-\rho\Gb\big)$ and the definition of $\Rc(b)$, we have $ J\big(\Theta-\rho_{\max}\Gb\big)=\check J\big(\Theta-\rho_{\max}\Gb\big)+\upsilon(\lambda)=b+\upsilon(\lambda)$.  Therefore, for any $\rho \in(0,\rho_{\max}]$, we have $\Theta-\rho\Gb\in \Rc (b )$. A second order Taylor's expansion together with the Lipschitz continuity of $\nabla J$ and \eqref{condition:perturbation} gives
    \begin{align*}
       J\big(\Theta-\rho\Gb\big)&\leq J(\Theta)-\rho \langle\nabla J(\Theta),\Gb\rangle_H+\frac{ \rho^2 \check L (b )}{2}\lVert \Gb\rVert_H^2\\
       &\leq  J(\Theta )-\rho \nu_1\lVert \nabla  J(\Theta )\rVert_H^2+\frac{ \rho ^2\nu_2 \check L (b )}{2}\lVert \nabla  J(\Theta )\rVert_H^2 
    \end{align*}
    so that
    $$ 
    J\big(\Theta-\rho\Gb\big)- J(\Theta)\leq \frac{-\rho (2\nu_1-\rho  \nu_2 \check L (b ))}{2}\lVert \nabla  J(\Theta)\rVert_H^2<0
    $$
    \noindent which in turn clearly gives
    $$
    J\big(\Theta-\rho\Gb\big)< J(\Theta)=\check J(\Theta)+\upsilon(\lambda)\leq b+\upsilon(\lambda).
    $$ 
    This last inequality contradicts the fact that $ J\big(\Theta-\rho_{\max}\Gb\big)=b+\upsilon(\lambda)$. We thus conclude that $\rho_{\max}\geq \frac{2\nu_1}{\nu_2 \check L(b)}$ and $ \Theta_\rho=\Theta-\rho\Gb\in\Rc(b)$ for all $\rho\in(0,\frac{2\nu_1}{\nu_2 \check L(b)})$.    
\end{proof}

\begin{Proposition}\label{ConvStoGD}
For any $\Theta=(\theta,\zeta)\in\Rc(b)$, any $G \in H$ satisfying \eqref{condition:perturbation} and any $\rho\in(0,\frac{2\nu_1}{\nu_2 \check L(b)}) $, $\Theta_\rho$ defined by \eqref{def:theta:rho} satisfies 
   $$
   J\big(\Theta_\rho)-J(\Theta^* )\leq\big(1-\frac{\rho(2\nu_1-\rho \check L(b)\nu_2)}{2\bar{\kappa}(b)}\big)(J(\Theta)-J(\Theta^*)),
   $$
   where $\bar{\kappa}(b):=\max\Big(\kappa_1,\kappa_2,\frac{\nu_1^2}{2\nu_2 \check L(b)}\Big)+\frac{1}{2}$.

\end{Proposition}

\begin{proof}

Writing again a second order Taylor expansion and using \eqref{condition:perturbation}, we get 
    \begin{align*}
       J\big(\Theta-\rho\Gb\big)&\leq J(\Theta)-\rho \langle\nabla J(\Theta),\Gb\rangle_H+\frac{ \rho^2 \check L (b )}{2}\lVert \Gb\rVert_H^2\\
       &\leq  J(\Theta )-\rho \nu_1\lVert \nabla  J(\Theta )\rVert_H^2+\frac{ \rho ^2\nu_2 \check L (b )}{2}\lVert \nabla  J(\Theta )\rVert_H^2 \\
       &\leq J(\Theta)+\big(\frac{\rho^2\nu_2\check L(b)}{2}-\rho\nu_1\big)\lVert \nabla  J(\Theta )\rVert_H^2. 
    \end{align*}

  \noindent  Note that the gradient domination inequality \eqref{GD} is satisfied with $\bar\kappa(b)$ instead of $\kappa$ and since
    $$
    \frac{\rho(2\nu_1-\rho \check L(b)\nu_2)}{2\bar\kappa(b)}\in(0,1),
    $$  
    \noindent we get 
    \begin{align*}
        J(\Theta-\rho\Gb)-J(\Theta)&\leq \;  \big( \frac{ \rho ^2\nu_2 \check L (b)}{2}-\rho\nu_1\big) \lVert \nabla J(\Theta)\rVert_F^2
        \; \leq \; \frac{1}{\bar\kappa(b)}\big(\frac{\rho^2 \nu_2\check L(b)}{2}-\rho\nu_1\big)(J(\Theta)-J(\Theta^*))
    \end{align*}
    so that
$$  
J(\Theta-\rho\Gb)-J(\Theta^*)\leq\big(1-\frac{\rho(2\nu_1-\rho \nu_2\check L(b))}{2\bar\kappa(b)}\big)(J(\Theta)-J(\Theta^*)).
$$

\end{proof}

\subsection{Error analysis of $J$ and its gradient}\label{SecErrGE}

We here study the five terms appearing in the decomposition of the error of $(\tilde{\nabla}_\theta^{\Delta, N, pop}  -\nabla_\theta) J$ and $(\tilde{\nabla}_\zeta^{\Delta, N, pop}  -\nabla_\zeta) J$ introduced in Step 1 of Section \ref{main:proof:theorem:conv:sgd}.\\

\noindent $\bullet$ 
\textbf{Bounding $\lVert\hat\nabla-\nabla\rVert_F$ (Error of perturbation with the exact expected functional cost).} 

\vspace{1mm}
 
The proof of the following result is postponed to Section \ref{proof:prop:PropPerb}.

\begin{Proposition}\label{PropPerb}
   For all $\varepsilon>0$, there exist $\h_r(b,\frac{1}{\varepsilon})$  with at most polynomial growth in $b$ and $\frac{1}{\varepsilon}$  and $\h_{\tilde N}(b,\frac{1}{r},\frac{1}{\varepsilon})$ with at most polynomial growth in $b$, $\frac{1}{r}$ and $\frac{1}{\varepsilon}$ such that for all $r<\min(\check{r}(b),\frac{1}{\h_r(b,\frac{1}{\varepsilon})})$, all $\tilde N>\h_{\tilde N}(b,\frac{1}{r},\frac{1}{\varepsilon})$ and all $\Theta=(\theta,\zeta)\in\Rc(b)$,  
   \begin{equation}\label{E1}
       \lVert (\hat\nabla_\theta -\nabla_\theta) J(\Theta)\rVert_F\leq\varepsilon \quad \mbox{ and } \quad \lVert (\hat\nabla_\zeta -\nabla_\zeta) J(\Theta)\rVert_F\leq\varepsilon,
   \end{equation}
\noindent with probability at least $1-(d/\varepsilon)^{-d}$.
\end{Proposition} 

\bigskip

\noindent $\bullet$ 
\textbf{Bounding $\lVert \hat\nabla^{(T)}-\hat\nabla\rVert_F$ (Horizon truncation error).}

\vspace{1mm}

The proof of the following result is postponed to Section \ref{proof:prop:PropTrunc}. 

\begin{Proposition}\label{PropTrunc}
    $\P-$a.s., for all $\Theta\in\Rc(b)$, one has
    $$
    \lVert (\hat\nabla_\theta -\hat\nabla_\theta^{(T)}) J(\Theta)\rVert_F\leq \frac{d}{r}c_1(2b)e^{-c_2(2b) T}, 
    $$
     $$\lVert (\hat\nabla_\zeta -\hat\nabla_\zeta^{(T)}) J(\Theta)\rVert_F\leq \frac{d}{r}c_1(2b)e^{-c_2(2b) T}. 
     $$
\end{Proposition}

\bigskip

\noindent $\bullet$ 
\textbf{Bounding $\lVert \hat\nabla^{(T)}-\hat\nabla^{(T),\Delta}\rVert_F$\, (Time discretization error).}

\vspace{1mm}

The proof of the following result is postponed to Section \ref{proof:prop:PropDiscre}.

\begin{Proposition}\label{PropDiscre}
    There exists a constant $c_3=c_3(2b)>0$ (non-decreasing with respect to $b$) such that $\P-$a.s., for all $\Theta\in\Rc(b)$, it holds 
    $$
    \lVert (\hat\nabla^{(T)}_\theta - \hat\nabla_\theta^{(T),\Delta}) J(\Theta)\rVert_F\leq \frac{d}{r}c_3(2b)\frac{T}{n}, 
    $$
     $$
     \lVert (\hat\nabla^{(T)}_\zeta - \hat\nabla_\zeta^{(T),\Delta}) J(\Theta)\rVert_F\leq \frac{d}{r}c_3(2b)\frac{T}{n}.
     $$
\end{Proposition}

\bigskip

\noindent $\bullet$ 
\textbf{Bounding $\lVert \tilde \nabla^{\Delta, N,pop}-\hat\nabla^{\Delta, N,pop}\rVert_F$\, (Statistical error).}

\vspace{1mm}

The proof of the following result is postponed to Section \ref{secMcerror}.

\begin{Proposition}\label{PropMCError}
    (Statistical error on gradient estimators) For all $\varepsilon>0$, there exist $ \h'_r(b,\frac{1}{\varepsilon})$ and $\h'_{\tilde N}(b,\frac{1}{r}, T , n , N , \frac{1}{\varepsilon})$ with at most polynomial growth in $\frac{1}{r}$, $b$, $\frac{1}{\varepsilon}$, $\log N$, $\log n$, $\log T$ such that for all $r<\min\{\check{ r}(b),1/\h'_r(b,\frac{1}{\varepsilon})\}$ and all $\tilde N\geq \h'_{\tilde N}(b,\frac{1}{r},T,n,N,\frac{1}{\varepsilon})$, with probability at least $1-(d/\varepsilon)^{-d}$, for all $\Theta\in\Rc(b)$, it holds
    \begin{equation}\label{ErrorMCE}
        \begin{aligned}
               \lVert (\tilde\nabla^{\Delta,N,pop}_{\theta} -\hat\nabla^{\Delta,N,pop}_{\theta}) J(\Theta)\rVert_F&\leq\varepsilon,\\
        \lVert (\tilde\nabla^{\Delta,N,pop}_{\zeta} -\hat\nabla^{\Delta,N,pop}_{\zeta}) J(\Theta)\rVert_F&\leq\varepsilon.
        \end{aligned}
    \end{equation}
\end{Proposition}

\bigskip

\noindent $\bullet$  \textbf{Bounding $\lVert  \hat \nabla^{\Delta, N,pop}-\hat\nabla^{(T),\Delta}\rVert_F$ \, 
(Particle discretization error).}

\vspace{1mm}

The proof of the following result is postponed to Section \ref{secparticle}.

\begin{Proposition}\label{PropN}
Let $\bar \Jc^{\Delta,N}_{pop}(\Theta):=\E[\Jc^{\Delta,N}_{pop}(\Theta)]$, $\Theta\in\Rc(b)$, recalling that $\Jc^{\Delta,N}_{pop}$ is defined by \eqref{EstimpopDef}. There exists $c_4=c_4(b)>0$ such that for all $\Theta\in\Rc(b)$, 
\begin{align} \label{estimlemD5} 
    \lvert (\bar \Jc^{\Delta,N}_{pop}-J^{(T),\Delta})(\Theta)\rvert &\leq \;  \frac{c_4}{N}.
\end{align} 
Moreover, $\P-$a.s., 
for all $\Theta\in\Rc(b)$, it holds
\begin{equation} \label{estimD6} 
\begin{aligned} 
    \lVert  (\hat \nabla^{\Delta, N,pop}_\theta -\hat\nabla^{(T),\Delta}_\theta) J(\Theta)\rVert_F &\leq \;  
    \frac{d}{r}\frac{c_4(2b)}{N}, \\
    \lVert  
     (\hat \nabla^{\Delta, N,pop}_\zeta -\hat\nabla^{(T),\Delta}_\zeta )J(\Theta)\rVert_F & \leq \;  
     \frac{d}{r}\frac{c_4(2b)}{N}. 
\end{aligned} 
\end{equation} 
\end{Proposition}

\subsubsection{Proof of Proposition \ref{PropPerb}}\label{proof:prop:PropPerb}

The proof is reminiscent of Lemma 31 \cite{carlautan19a} or Lemma 30 \cite{fazetal18}. 
For any $\Theta=(\theta,\zeta)\in\Sc\times\hat \Sc$ and $r>0$, let us introduce the following smooth approximation of $J$ defined by 
\begin{equation}\label{DefJr}
    J_r(\theta,\zeta)=\E_{(P_r,\hat{P}_r)}[J(\theta+P_r,\zeta+\hat{P}_r)]=\E_{(P_r,\hat{P}_r)}[J_1(\theta+P_r)+J_2(\zeta+\hat{P}_r)]+\upsilon(\lambda)
\end{equation}
where $P_r$ and $\hat{P}_r$ are two independent random variables uniformly distributed in $\S_r$  and $\E_{(P_r,\hat{P}_r)}$ stands for the expectation taken with respect to the random variables $P_r,\hat{P}_r$.

\vspace{1mm}

The following result is directly taken from Lemma 26 in \cite{fazetal18}. 
For any $(\theta,\zeta)\in\Sc\times\hat \Sc$, it holds
    \begin{equation} \label{LemmaGradJr}
    \begin{aligned} 
        \nabla_\theta J_r(\theta,\zeta)=\frac{d}{r^2}\E_{P_r}[J_1(\theta+P_r)P_r]=\frac{d}{r^2}\E_{(P_r,\hat P_r)}[J(\theta+P_r,\zeta+\hat{P}_r)P_r] \\
        \nabla_\zeta J_r(\theta,\zeta)=\frac{d}{r^2}\E_{\hat{P}_r}[J_2(\zeta+\hat{P}_r)\hat{P}_r]=\frac{d}{r^2}\E_{(P_r, \hat{P}_r)}[J(\theta+P_r,\zeta+\hat{P}_r)\hat{P}_r].
        \end{aligned}
    \end{equation}



We then write 
\begin{align} 
\lVert (\hat\nabla_\theta-\nabla_\theta)J(\Theta)\rVert_F & \leq \;  \lVert \hat\nabla_\theta J(\Theta)-\nabla_\theta J_r(\Theta)\rVert_F+\lVert \nabla_\theta J_r(\Theta)-\nabla_\theta J(\Theta)\rVert_F. 
\end{align} 

For the second term $\lVert \nabla_\theta J_r(\Theta)-\nabla_\theta J(\Theta)\rVert_F$, using the Lipschitz property of the gradient $\nabla J(\Theta)$ shown in the Section \ref{ProofLip}, there exists $\tilde{\h}_r(b,\frac{1}{\varepsilon})$ with polynomial growth in $b$ and $\frac{1}{\varepsilon}$ such that for all $r<1/\tilde{\h}_r(b,\frac{1}{\varepsilon})$, for all $\Theta=(\theta,\zeta)\in\Rc(b)$ and $\Theta'=(\theta',\zeta')$ such that $\lVert\theta'-\theta\rVert_F\leq r$ and $\lVert\zeta'-\zeta\rVert_F\leq r$, we have $\lVert \nabla_\theta J(\Theta')-\nabla_\theta J(\Theta)\rVert_F\leq\frac{\varepsilon}{2}$. Then according  to \eqref{DefJr}, noticing that $\nabla_\theta J_r(\Theta)=\E_{(P_r,\hat P_r)}[\nabla_\theta J(\theta+P_r,\zeta+\hat P_r)]$, we have
\begin{align*}
    \lVert \nabla_\theta J_r(\Theta)-\nabla_\theta J(\Theta)\rVert_F   &\leq \E_{(P_r,\hat P_r)}[\lVert \nabla_\theta J(\theta+P_r,\zeta+\hat P_r)- \nabla_\theta J(\Theta)\rVert_F]\leq\frac{\varepsilon}{2}. 
\end{align*}

Then, for the first term $\lVert \hat\nabla_\theta J(\Theta)-\nabla_\theta J_r(\Theta)\rVert_F$, using  \eqref{LemmaGradJr}, we notice that  
$\hat\nabla_\theta J(\Theta)$ is exactly the Monte-Carlo approximation for the expectation $\nabla_\theta J_r(\Theta)$. Assuming that $r<\check{r}(b)$, according to Lemma \ref{ra}, for all $\Theta\in\Rc(b), i=1,\dots,\tilde N,(\theta+U_i,\zeta+V_i)\in \Rc(2b)$ and each  individual sample has the norm bounded by $\frac{d}{r}(2b+\upsilon(\lambda))$ according to \eqref{Rcb}-\eqref{checkJ}.  Thus we can apply vector Bernstein’s Inequality to deduce that there exists $\h_{sample}(b,\frac{1}{r},\frac{1}{\varepsilon})$ with polynomial growth in $b$, $\frac{1}{r}$ and $\frac{1}{\varepsilon}$ such that for all $\tilde N\geq \h_{sample}(b,\frac{1}{r},\frac{1}{\varepsilon})$, for all $\Theta\in\Rc(b)$, we have $\lVert \hat\nabla_\theta J(\Theta)-\nabla_\theta J_r(\Theta)\rVert_F\leq \frac{\varepsilon}{2}$ with probability at least $1-(d/\varepsilon)^{-d}$.

The bound for $\lVert (\hat\nabla_\zeta -\nabla_\zeta) J(\Theta)\rVert_F$ is derived in a similar manner, and thus its proof is omitted.

\subsubsection{Proof of Proposition \ref{PropTrunc}}\label{proof:prop:PropTrunc}

In order to establish our result on $\lVert (\hat\nabla_\theta -\hat\nabla_\theta^{(T)}) J(\Theta)\rVert_F$ and $\lVert (\hat\nabla_\zeta -\hat\nabla_\zeta^{(T)}) J(\Theta)\rVert_F$, we first need to derive an error bound on $\lvert (J^{(T)}-J)(\Theta)\rvert$. 

\begin{Lemma}\label{DiffJJT} 
\noindent There exist two positive constants $c_1=c_1(b)$ and $c_2=c_2(b)$ depending only upon $b\in\R_+$, the model parameters and $\lambda$ such that for all $\Theta\in\Rc(b)$
$$
\lvert  (J^{(T)}-J)(\Theta)\rvert\leq c_1(b)e^{-c_2(b)T}.
$$
Besides, $c_1(.)$ and $c_2(.)$ are non-decreasing functions.
\end{Lemma}

Before proving the above lemma, we will need the following technical result which is directly taken  from Lemma 12 \cite{ConvSampGradMethod}. Its proof is thus omitted.
\begin{Lemma}\label{LemmaCited}
    Let the matrices $F,X\succ 0$ and $\Omega\succ 0$ satisfy $$FX+XF\trans+\Omega=0$$ then for any $t\geq 0$,$$\lVert e^{Ft}\rVert_2^2\leq \frac{\lVert X\rVert_2}{\sigma_{\min}(X)}e^{-\frac{\sigma_{\min}(\Omega)}{\lVert X\rVert_2}t}$$
    where $\lVert\cdot\rVert_2$ denotes the largest singular value of the matrices.
\end{Lemma} 

\vspace{1mm}

\noindent \emph{ Proof of Lemma \ref{DiffJJT}}.  
From the definition of $J$ and $J^{(T)}$, we get
\begin{align*}
  &  (J-J^{(T)})(\Theta)\\
&=\E\Big[ \int_T^\infty e^{-\beta t} \big(  ( Y^\Theta_t)\trans Q Y^\Theta_t+(Z^\Theta_t)\trans \hat QZ^\Theta_t+ \int_{\R^m}(a\trans Ra  + \lambda \log p^\Theta( Y^\Theta_t, Z^\Theta_t,a))p^\Theta(  Y^\Theta_t,  Z^\Theta_t,a)\d a \big)  \d t \Big]\\
&=\E\Big[ \int_T^\infty e^{-\beta t} \big(  ( Y^\Theta_t)\trans (Q+\theta\trans R\theta) Y^\Theta_t+( Z^\Theta_t)\trans (\hat Q+\zeta\trans R\zeta) Z^\Theta_t\big)\d t\Big]+\upsilon^{(T)}(\lambda) \\
\end{align*}
where 
\begin{equation}\label{defupsilonlambdaT}
    \upsilon^{(T)}(\lambda)=\Big(   - \frac{\lambda m}{2} \log (\pi\lambda) + \frac{\lambda}{2} \log\big| \mathrm{det}(R) \big| \Big)\int_T^{\infty} e^{-\beta t}\d t= \frac{e^{-\beta T}}{\beta}\Big(   - \frac{\lambda m}{2} \log (\pi\lambda) + \frac{\lambda}{2} \log\big| \rm det(R) \big| \Big).
\end{equation}

  In particular, note that $\upsilon(\lambda)$ defined in \eqref{defupsilonlambda} is in fact $\upsilon^{(0)}(\lambda)$ for $T=0$.


\noindent Denoting for simplicity $\bC_t=\E[  Y^\theta_t ( Y^\theta_t)\trans]$ and $\hat\bC_t=\E[  Z^\zeta_t( Z^\zeta_t)\trans]$, one has
\begin{equation}\label{J-JT}
    (J-J^{(T)})(\Theta)=\langle Q+\theta\trans R\theta,\int_T^\infty e^{-\beta t}\bC_t\d t\rangle+ \langle \hat Q+\zeta\trans R\zeta,\int_T^\infty e^{-\beta t}\hat \bC_t\d t\rangle+\upsilon^{(T)}(\lambda)
\end{equation}

\noindent As proved in Lemma \ref{LemmaB2} (using It\^o's formula), 
\begin{align*}
    \bC_t&=\exp(t(B+D\theta))\bC_0\exp(t(B+D\theta)\trans)+\int_0^t \exp(-(s-t)(B+D\theta))\gamma\gamma\trans\exp(-(s-t)(B+D\theta)\trans)\d s,\\
   \hat\bC_t&=\exp(t(\hat B+D\zeta))\hat\bC_0\exp(t(\hat B+D\zeta)\trans)+\int_0^t \exp(-(s-t)(\hat B+D\zeta))\gamma_0\gamma_0\trans\exp(-(s-t)(\hat B+D\zeta)\trans)\d s.
\end{align*}

\noindent
The equations \eqref{PropSigmaTheta} satisfied by $\Sigma_\theta$ and $\hat\Sigma_\zeta$ together with Lemma \ref{LemmaCited} guarantee that $$
\lVert e^{(B-\frac{\beta}{2}I_d+D\theta) t}\rVert_2^2\leq\frac{\lVert \Sigma_\theta\rVert_2}{\sigma_{\min}(\Sigma_\theta)}e^{-\frac{\sigma_{\min}(M)}{\lVert \Sigma_\theta\rVert_2}t},
$$
$$
\lVert e^{(\hat B-\frac{\beta}{2}I_d+D\zeta) t}\rVert_2^2\leq\frac{\lVert \hat\Sigma_\zeta\rVert_2}{\sigma_{\min}(\hat\Sigma_\zeta)}e^{-\frac{\sigma_{\min}(\hat M)}{\lVert \hat\Sigma_\zeta\rVert_2}t},
$$

\noindent so that
\begin{align*}
    \lVert e^{(B+D\theta) t}\rVert_2^2&\leq \lVert e^{(B-\frac{\beta}{2}I_d+D\theta) t}\rVert_2^2\lVert e^{\frac{\beta}{2}I_d t}\rVert_2^2 \leq \frac{\lVert \Sigma_\theta\rVert_2}{\sigma_{\min}(\Sigma_\theta)}e^{-\frac{\sigma_{\min}(M)}{\lVert \Sigma_\theta\rVert_2} t} e^{2\lVert \frac{\beta t}{2}I_d\rVert_2}=\frac{\lVert \Sigma_\theta\rVert_2}{\sigma_{\min}(\Sigma_\theta)}e^{(\beta- \frac{\sigma_{\min}(M)}{\lVert \Sigma_\theta\rVert_2}) t},\\
 \lVert e^{(\hat B+D\zeta) t}\rVert_2^2 &\leq \lVert e^{(\hat B-\frac{\beta}{2}I_d+D\zeta) t}\rVert_2^2\lVert e^{\frac{\beta}{2}I_d t}\rVert_2^2  \leq \frac{\lVert \hat\Sigma_\zeta\rVert_2}{\sigma_{\min}(\hat\Sigma_\zeta)}e^{-\frac{\sigma_{\min}(\hat M)}{\lVert \hat\Sigma_\zeta\rVert_2} t} e^{2\lVert \frac{\beta t}{2}I_d\rVert_2}=\frac{\lVert \hat\Sigma_\zeta\rVert_2}{\sigma_{\min}(\hat\Sigma_\zeta)}e^{(\beta- \frac{\sigma_{\min}(\hat M)}{\lVert \hat\Sigma_\zeta\rVert_2}) t}.
\end{align*}

\noindent Hence,
\begin{align*}
    \lVert\bC_t\rVert_F&\leq \lVert \bC_0\rVert_F\lVert e^{(B+D\theta) t}\rVert_2^2+\lVert \gamma\rVert^2\int_0^t \lVert e^{(B+D\theta) v}\rVert_2^2\d v\\
    &\leq\frac{\lVert \Sigma_\theta\rVert_2}{\sigma_{\min}(\Sigma_\theta)}\Bigg( \big(\lVert\bC_0\rVert_F+\frac{\lVert\gamma\rVert_F^2}{\beta-\frac{\sigma_{\min}(M)}{\lVert\Sigma_\theta\rVert_2}}\big)e^{(\beta- \frac{\sigma_{\min}(M)}{\lVert \Sigma_\theta\rVert_2}) t}-\frac{\lVert\gamma\rVert_F^2}{\beta-\frac{\sigma_{\min}(M)}{\lVert\Sigma_\theta\rVert_2}}\Bigg)
\end{align*}
and similarly
$$\lVert\hat \bC_t\rVert_F\leq \frac{\lVert \hat\Sigma_\zeta\rVert_2}{\sigma_{\min}(\hat\Sigma_\zeta)}\Bigg( \big(\lVert\hat\bC_0\rVert_F+\frac{\lVert\gamma_0\rVert_F^2}{\beta-\frac{\sigma_{\min}(\hat M)}{\lVert\hat \Sigma_\zeta\rVert_2}}\big)e^{(\beta- \frac{\sigma_{\min}(\hat M)}{\lVert \hat\Sigma_\zeta\rVert_2}) t}-\frac{\lVert\gamma_0\rVert_F^2}{\beta-\frac{\sigma_{\min}(\hat M)}{\lVert\hat\Sigma_\zeta\rVert_2}}\Bigg).
$$
Since $\Theta=(\theta,\zeta)\in\Rc(b)\subset\Sc(b)\times\hat\Sc(b)$, recalling that $\Sc(b)$ as well as $\hat\Sc(b)$ are compact sets, using the continuity to $\theta\mapsto\lVert\Sigma_\theta\rVert_2$ and $\zeta\mapsto\lVert\hat\Sigma_\zeta\rVert_2$, there exist $\tilde{c}_1(b), \, \tilde{c}_2(b)>0$ such that for all $\Theta=(\theta,\zeta)\in\Rc(b)$
$$
\tilde{c}_2(b)\leq\lVert\Sigma_\theta\rVert_2\leq \tilde{c}_1(b), \qquad  \tilde{c}_2(b)\leq\lVert\hat\Sigma_\zeta\rVert_2\leq \tilde{c}_1(b)
$$

\noindent and using the fact that $\Sigma_\theta\succeq M,\hat\Sigma_\theta\succeq \hat M$, we get
\begin{align*}
    e^{-\beta t} \lVert\bC_t\rVert_F&\leq\frac{\lVert \Sigma_\theta\rVert_2}{\sigma_{\min}(\Sigma_\theta)}\Bigg( \big(\lVert\bC_0\rVert_F+\frac{\lVert\gamma\rVert_F^2}{\beta-\frac{\sigma_{\min}(M)}{\lVert\Sigma_\theta\rVert_2}}\big)e^{- \frac{\sigma_{\min}(M)}{\lVert \Sigma_\theta\rVert_2} t}-\frac{\lVert\gamma\rVert_F^2e^{-\beta t}}{\beta-\frac{\sigma_{\min}(M)}{\lVert\Sigma_\theta\rVert_2}}\Bigg)\\
    &\leq \frac{\tilde c_1(b)}{\sigma_{\min}(M)}\Bigg( \big(\lVert\bC_0\rVert_F+\frac{\lVert\gamma\rVert_F^2}{\beta-\frac{\sigma_{\min}(M)}{\tilde{c}_2(b)}}\big)e^{- \frac{\sigma_{\min}(M)}{\tilde{c}_2(b)} t}-\frac{\lVert\gamma\rVert_F^2e^{-\beta t}}{\beta-\frac{\sigma_{\min}(M)}{\tilde{c}_1(b)}}\Bigg),
\end{align*}
\noindent and
\begin{align*}
    e^{-\beta t}\lVert\hat \bC_t\rVert_F&\leq \frac{\lVert \hat\Sigma_\zeta\rVert_2}{\sigma_{\min}(\hat\Sigma_\zeta)}\Bigg( \big(\lVert\hat\bC_0\rVert_F+\frac{\lVert\gamma_0\rVert_F^2}{\beta-\frac{\sigma_{\min}(\hat M)}{\lVert\hat \Sigma_\zeta\rVert_2}}\big)e^{- \frac{\sigma_{\min}(\hat M)}{\lVert \hat\Sigma_\zeta\rVert_2} t}-\frac{\lVert\gamma_0\rVert_F^2e^{-\beta t}}{\beta-\frac{\sigma_{\min}(\hat M)}{\lVert\hat\Sigma_\zeta\rVert_2}}\Bigg)\\
    &\leq \frac{\tilde c_1(b)}{\sigma_{\min}(\hat M)}\Bigg( \big(\lVert\hat\bC_0\rVert_F+\frac{\lVert\gamma_0\rVert_F^2}{\beta-\frac{\sigma_{\min}(\hat M)}{\tilde{c}_2(b)}}\big)e^{- \frac{\sigma_{\min}(\hat M)}{\tilde{c}_1(b)} t}-\frac{\lVert\gamma_0\rVert_F^2e^{-\beta t}}{\beta-\frac{\sigma_{\min}(\hat M)}{\tilde{c}_1(b)}}\Bigg).
\end{align*}


By using again the fact that $\Rc(b)$ is compact, there exists a constant $\tilde c_3(b)$ such that for all $\Theta=(\theta,\zeta)\in\Rc(b)$:
$$
\max\{\lVert Q+\theta\trans R\theta\rVert_F,\lVert \hat Q+\zeta\trans R\zeta\rVert_F\}\leq \tilde c_3(b).
$$

Combing back to \eqref{J-JT} and plugging the above estimates, we obtain for all $\Theta\in\Rc(b)$
\begin{align*}
    \lvert (J^{(T)}-J)(\Theta)\rvert&\leq \lVert Q+\theta\trans R\theta\rVert_F\int_T^\infty e^{-\beta t} \lVert\bC_t\rVert_F\d t+\lVert \hat Q+\zeta\trans R\zeta\rVert_F\int_T^\infty e^{-\beta t} \lVert\hat \bC_t\rVert_F\d t+\lvert \upsilon^{(T)}(\lambda)\rvert\\
    &\leq \tilde c_3(b) C_1(b)\int_T^\infty e^{-C_2(b)t}\d t+\tilde c_3(b)  C_3(b) \int_T^\infty e^{-C_4(b)t}\d t +O_\lambda(e^{-\beta T})\\
   & \leq \;  c_1(b)e^{-c_2(b)T}, 
\end{align*}
for some $c_1(b)>0$ and $c_2(b)>0$ depending only upon $b$ and $\lambda$.

Finally, the monotonicity of $c_1$ and $c_2$ with respect to $b$ is a consequence of the fact that if $b_1\leq b_2$ then $\Rc(b_1)\subset\Rc(b_2)$.
\ep

\vspace{3mm}

\noindent \textit{Proof of Proposition \ref{PropTrunc}.} 
    Noticing that $\P$-$a.s.$ $(\theta+U_i ,\zeta+V_i)\in \Rc(2b)$, from Lemma \ref{DiffJJT}, 
    $$
    \lvert (J^{(T)}-J)(\Theta_i)\rvert\leq c_1(2b)e^{-c_2(2b) T}, \quad \P\mbox{-}a.s. 
    $$ 
    \noindent and, recalling that $\lVert U_i\rVert_F=\lVert V_i\rVert=r$, $\P$-\emph{a.s.} it holds
    \begin{align*}
        \lVert (\hat\nabla_\theta -\hat\nabla_\theta^{(T)}) J(\Theta_i)\rVert_F&\leq \frac{d}{r^2}\frac{1}{\tilde N}\sum_{i=1}^{\tilde N}\lvert (J^{(T)}-J)(\Theta_i)\rvert \lVert U_i\rVert_F\leq \frac{d}{r}c_1(2b)e^{-c_2(2b) T},\\
         \lVert (\hat\nabla_\zeta -\hat\nabla_\zeta^{(T)}) J(\Theta_i)\rVert_F&\leq \frac{d}{r^2}\frac{1}{\tilde N}\sum_{i=1}^{\tilde N}\lvert (J^{(T)}-J)(\Theta_i)\rvert \lVert V_i\rVert_F \leq \frac{d}{r}c_1(2b)e^{-c_2(2b) T}.
    \end{align*}
\ep

\subsubsection{Proof of Proposition \ref{PropDiscre}}\label{proof:prop:PropDiscre}

We start with the following technical result related to the weak discretization error on the cost value function.
\begin{Lemma}\label{ErrJTDelta-JT}
    There exists a constant $c_3=c_3(b)>0$ (non-decreasing with respect to $b$) such that for all $\Theta\in\Rc(b)$, it holds 
    $$
    \lvert (J^{(T),\Delta}-J^{(T)})(\Theta)\rvert\leq c_3(b){h}=c_3(b)\frac{T}{n}.
    $$
\end{Lemma}
\begin{proof}

\emph{Step 1:} 
We introduce the two processes $(\yb^{\Theta,\Delta}_{t_l})_{0 \leq l \leq  n}$ and $(\zb^{\Theta,\Delta}_{t_l})_{0\leq l \leq n}$ defined by 
$$
\yb^{\Theta,\Delta}_{t_l}=\xb^{\Theta,\Delta}_{t_l}-\E_0[\xb^{\Theta,\Delta}_{t_l}] \mbox{ and } \zb^{\Theta,\Delta}_{t_l}=\E_0[\xb^{\Theta,\Delta}_{t_l}],
$$ 

\noindent with dynamics
\begin{equation}\label{yzDeltTheta}
    \begin{aligned}
            \yb^{\Theta,\Delta}_{t_{l+1}}&= \yb^{\Theta,\Delta}_{t_l}+((B+D\theta) \yb^{\Theta,\Delta}_{t_l}+D\sqrt{\frac{\lambda}{2}R^{-1}}\xib_{t_l}){h}+\sqrt{{h}}\gamma\bw_l, \quad \yb^{\Theta,\Delta}_0=X_0-\E_0[X_0],\\
 \zb^{\Theta,\Delta}_{t_{l+1}}&=\zb^{\Theta,\Delta}_{t_l}+(\hat B+D\zeta) \zb^{\Theta,\Delta}_{t_l}{h}+\sqrt{{h}}\gamma_0\bw^0_l, \quad  \zb^{\Theta,\Delta}_0=\E_0[X_0].
    \end{aligned}
\end{equation}

Recalling \eqref{DiscApproFC} together with the dynamics defined in \eqref{xDeltThet}, one has 
\begin{equation}\label{DiscApproFC2}
    \begin{aligned}
        J^{(T),\Delta}(\Theta)&={h}\E\Bigg[\sum_{l=0}^{n-1} e^{-\beta t_l}\Big((\yb^{\Theta,\Delta}_{t_l})\trans (Q+\theta\trans R\theta) \yb^{\Theta,\Delta}_{t_l}+(\zb^{\Theta,\Delta}_{t_l})\trans(\hat Q+\zeta\trans R\zeta)\zb^{\Theta,\Delta}_{t_l}+\beta\upsilon(\lambda)\Big)\Bigg].
    \end{aligned}
\end{equation}

We also define another auxiliary functional cost:
\begin{equation}\label{tildDiscApproFC}
    \begin{aligned}
        \tilde J^{(T),\Delta}(\Theta)&={h}\E\Bigg[\sum_{l=0}^{n-1} e^{-\beta t_l}\Big(( \tilde\xb^{\Theta,\Delta}_{t_l}-\E_0[ \tilde\xb^{\Theta,\Delta}_{t_l}])\trans Q ( \tilde\xb^{\Theta,\Delta}_{t_l}-\E_0[ \tilde\xb^{\Theta,\Delta}_{t_l}])+\E_0[ \tilde\xb^{\Theta,\Delta}_{t_l}]\trans\hat Q\E_0[ \tilde\xb^{\Theta,\Delta}_{t_l}]\\
        &+\int \big(a\trans R a+\lambda\log p^\Theta( \tilde\xb^{\Theta,\Delta}_{t_l}-\E_0[ \tilde\xb^{\Theta,\Delta}_{t_l}],\E_0[ \tilde\xb^{\Theta,\Delta}_{t_l}],a)\big)p^{\Theta}( \tilde\xb^{\Theta,\Delta}_{t_l}-\E_0[ \tilde\xb^{\Theta,\Delta}_{t_l}],\E_0[ \tilde\xb^{\Theta,\Delta}_{t_l}],a)\d a\Big)\Bigg],
    \end{aligned}
\end{equation}

\noindent where $( \tilde\xb^{\Theta,\Delta}_{t_l})_{0\leq l \leq n}$ is the time-discretization scheme of $(X^\Theta_t)_{t\in[0,T]}$, over the same time grid $\Delta$ as the interacting agents dynamics \eqref{XDeltThet}, with dynamics
\begin{equation}\label{tildxDeltThet}
\begin{cases} 
     \tilde\xb^{\Theta,\Delta}_{t_{l+1}}  & = \;  \tilde\xb^{\Theta,\Delta}_{t_l}+(B \tilde\xb^{\Theta,\Delta}_{t_l}+\bar B \E_0[ \tilde\xb^{\Theta,\Delta}_{t_l}]+D\int a\pi^\Theta(\d a| \tilde\xb^{\Theta,\Delta}_{t_l}-\E_0[ \tilde\xb^{\Theta,\Delta}_{t_l}],\E_0[ \tilde\xb^{\Theta,\Delta}_{t_l}])) {{h}} \\
    & \qquad \qquad + \; \sqrt{{h}} \gamma\bw_l+ \sqrt{{h}}\gamma_0\bw^0_l\; ; \\
      &=\tilde\xb^{\Theta,\Delta}_{t_l}+((B+D\theta)\tilde\xb^{\Theta,\Delta}_{t_l}+(\bar B-D\theta+D\zeta)\E_0[\tilde\xb^{\Theta,\Delta}_{t_l}]){h}+\sqrt{{h}}\gamma \bw_l+ \sqrt{{h}}\gamma_0\bw^0_l\\
   \tilde{\xb}^{\Theta,\Delta}_{0} & = \; X_0
\end{cases}     
\end{equation}
\noindent with $\Theta=(\theta,\zeta)\in \mathcal{R}(b)$. 



 Let us  also introduce the two processes 
$( \tilde\yb^{\Theta,\Delta}_{t_l})_{l=0,\dots,n}$ and $( \tilde\zb^{\Theta,\Delta}_{t_l})_{l=0,\dots,n}$ defined by
\begin{equation}\label{tildeyzDeltThet}
    \begin{aligned}
        \tilde \yb^{\Theta,\Delta}_{t_l}= \tilde\xb^{\Theta,\Delta}_{t_l}-\E_0[\tilde\tilde \tilde\xb^{\Theta,\Delta}_{t_l}], \quad \tilde\zb^{\Theta,\Delta}_{t_l}=\E_0[\tilde\zb^{\Theta,\Delta}_{t_l}],
    \end{aligned}
\end{equation}
with dynamics
\begin{align*}
     \tilde\yb^{\Theta,\Delta}_{t_{l+1}}&= \tilde\yb^{\Theta,\Delta}_{t_l}+(B+D\theta) \tilde\yb^{\Theta,\Delta}_{t_l}{h}+\sqrt{{h}}\gamma\bw_l, \quad \tilde\yb^{\Theta,\Delta}_0=X_0-\E_0[X_0],\\
\tilde\zb^{\Theta,\Delta}_{t_{l+1}}&=\tilde\zb^{\Theta,\Delta}_{t_l}+(\hat B+D\zeta)\tilde\zb^{\Theta,\Delta}_{t_l}{h}+\sqrt{{h}}\gamma_0\bw^0_l, \quad  \tilde\zb^{\Theta,\Delta}_0=\E_0[X_0].
\end{align*}
Note that $\tilde J^{(T),\Delta}$ in  \eqref{tildDiscApproFC}, can be written using the dynamics of $(\tilde\yb^{\Theta,\Delta}_{t_l})_{l=0,\cdots, n}$ and $( \tilde\zb^{\Theta,\Delta}_{t_l})_{l=0, \cdots, n}$. Namely, one has 
\begin{equation}\label{tildeDiscApproFC2}
    \begin{aligned}
     \tilde   J^{(T),\Delta}(\Theta)&= \; {h}\E\Bigg[\sum_{l=0}^{n-1} e^{-\beta t_l}\Big((\tilde\yb^{\Theta,\Delta}_{t_l})\trans Q \tilde\yb^{\Theta,\Delta}_{t_l}+(\tilde\zb^{\Theta,\Delta}_{t_l})\trans\hat Q\tilde\zb^{\Theta,\Delta}_{t_l}\\
        &\qquad + \; \int_\Ac (a\trans R a+\lambda\log p^\Theta(\tilde\yb^{\Theta,\Delta}_{t_l},\tilde\zb^{\Theta,\Delta}_{t_l},z))p^\Theta(\tilde\yb^{\Theta,\Delta}_{t_l},\tilde\zb^{\Theta,\Delta}_{t_l},z)\d z\Big)\Bigg]\\
        &=\; {h}\E\Bigg[\sum_{l=0}^{n-1} e^{-\beta t_l}\Big((\tilde\yb^{\Theta,\Delta}_{t_l})\trans (Q+\theta\trans R\theta) \tilde\yb^{\Theta,\Delta}_{t_l}+(\tilde\zb^{\Theta,\Delta}_{t_l})\trans(\hat Q+\zeta\trans R\zeta)\tilde\zb^{\Theta,\Delta}_{t_l}\\
        &\qquad + \; (   - \frac{\lambda m}{2} \log (\pi\lambda) + \frac{\lambda}{2} \log\big| \mathrm{det}(R) \big| )\Big)\Bigg]\\
        &=\; {h}\E\Bigg[\sum_{l=0}^{n-1} e^{-\beta t_l}\big((\tilde\yb^{\Theta,\Delta}_{t_l})\trans (Q+\theta\trans R\theta) \tilde\yb^{\Theta,\Delta}_{t_l}+(\zb^{\Theta,\Delta}_{t_l})\trans(\hat Q+\zeta\trans R\zeta)\tilde\zb^{\Theta,\Delta}_{t_l}\big)\Bigg]\\
        &\qquad + \; \beta\upsilon(\lambda){h}\sum_{l=0}^{n-1} e^{-\beta t_l}.\\
    \end{aligned}
\end{equation}

\noindent \emph{Step 2: } We first prove an upper-bound on $\lvert   (\tilde J^{(T),\Delta}-J^{(T),\Delta})(\Theta)\rvert$. From \eqref{yzDeltTheta}-\eqref{tildeyzDeltThet}, we get that $\tilde \zb^{\Theta,\Delta}_{t_l}=\zb^{\Theta,\Delta}_{t_l}$, for all $l=0,\dots n$ and 
$$\big( \yb^{\Theta,\Delta}_{t_{l+1}}-\tilde \yb^{\Theta,\Delta}_{t_{l+1}}\big)=( \yb^{\Theta,\Delta}_{t_{l}}-\tilde \yb^{\Theta,\Delta}_{t_{l}})+\big((B+D\theta)( \yb^{\Theta,\Delta}_{t_{l}}-\tilde \yb^{\Theta,\Delta}_{t_{l}})+D\sqrt{\frac{\lambda}{2}R^{-1}}\xi_{t_l}\big)h,\quad  \yb^{\Theta,\Delta}_{0}-\tilde \yb^{\Theta,\Delta}_{0}=0, 
$$
so that
$$\yb^{\Theta,\Delta}_{t_{l}}-\tilde \yb^{\Theta,\Delta}_{t_{l}}=\sum_{l'=0}^{l-1}(I_d+h(B+D\theta))^{l-1-l'}hD\sqrt{\frac{\lambda}{2}R^{-1}}\xib_{t_{l'}}.
$$
Therefore,  
\begin{align*}
   & \E\big[(\yb^{\Theta,\Delta}_{t_l})\trans (Q+\theta\trans R\theta) \yb^{\Theta,\Delta}_{t_l}-(\tilde\yb^{\Theta,\Delta}_{t_l})\trans (Q+\theta\trans R\theta) \tilde\yb^{\Theta,\Delta}_{t_l}\big] \\
   &=2\E\big[(\yb^{\Theta,\Delta}_{t_l}-\tilde \yb^{\Theta,\Delta}_{t_{l}})\trans (Q+\theta\trans R\theta) \tilde\yb^{\Theta,\Delta}_{t_l}\big]+ \E\big[(\yb^{\Theta,\Delta}_{t_l}-\tilde \yb^{\Theta,\Delta}_{t_{l}})\trans (Q+\theta\trans R\theta) (\yb^{\Theta,\Delta}_{t_l}-\tilde \yb^{\Theta,\Delta}_{t_{l}})\big]\\
   &=2\E\big[\big(\sum_{l'=0}^{l-1}(I_d+h(B+D\theta))^{l-1-l'}hD\sqrt{\frac{\lambda}{2}R^{-1}}\xib_{t_{l'}}\big)\trans (Q+\theta\trans R\theta) \tilde\yb^{\Theta,\Delta}_{t_l}\big]\\
   &+\E\big[\big(\sum_{l'=0}^{l-1}(I_d+h(B+D\theta))^{l-1-l'}hD\sqrt{\frac{\lambda}{2}R^{-1}}\xib_{t_{l'}}\big)\trans (Q+\theta\trans R\theta) \big(\sum_{l'=0}^{l-1}(I_d+h(B+D\theta))^{l-1-l'}hD\sqrt{\frac{\lambda}{2}R^{-1}}\xib_{t_{l'}}\big)\big] \\
   &=:  \; \rm{A}_1+ \rm{A}_2. 
\end{align*}
Since $(\xib_{t_l})_{0 \leq l \leq n}$ is independent of $(\tilde\yb^{\Theta,\Delta}_{t_l})_{0\leq l \leq n}$ which has zero-mean, we have $\rm{A}_1=0$. As for $\rm{A}_2$, also using the independence of $(\xib_{t_l})_{l=0,\dots,n}$, we have for any $S\in\S^m_{\geq 0}$,  
$\E[\xib_{t_i}\trans S\xib_{t_j}]=\delta_{i,j}\tr(S),i,j=0,\dots,n$, $\delta_{i,j}$ $=$ $1_{i\neq j}$, so that 
\begin{align*}
    \rm{A}_2&=\sum_{i=0}^{l-1}\sum_{j=0}^{l-1}\E \big[\big((I_d+h(B+D\theta))^{l-1-i}hD\sqrt{\frac{\lambda}{2}R^{-1}}\xib_{t_{i}}\big)\trans\\
    &(Q+\theta\trans R\theta) \big((I_d+h(B+D\theta))^{l-1-j}hD\sqrt{\frac{\lambda}{2}R^{-1}}\xib_{t_{j}}\big)\big]\\
    &=h^2\sum_{l'=0}^{l-1}\tr\big(((I_d+h(B+D\theta))^{l-1-l'}D\sqrt{\frac{\lambda}{2}R^{-1}})\trans (Q+\theta\trans R\theta) ((I_d+h(B+D\theta))^{l-1-l'}D\sqrt{\frac{\lambda}{2}R^{-1}})\big).
\end{align*}

 Consequently, there exists a constant $C_1>0$ depending only upon $T$, $Q$, $D$, $\lambda$, $R$ and $\Theta$ (with at most of polynomial growth in $\|\Theta\|$) such that $\lvert \rm{A}_2\rvert\leq\frac{C_1}{n}$ for all $l=0, \cdots, n$. Hence, up to a modification of $C_1$, for all $\Theta \in \mathcal{R}(b)$, it holds
\begin{equation}\label{ErrDiscreBound1}
    \begin{aligned}
           \lvert   (\tilde   J^{(T),\Delta}-J^{(T),\Delta})(\Theta)\rvert&=\Bigg| h\E\Bigg[\sum_{l=0}^{n-1} e^{-\beta t_l}\big((\yb^{\Theta,\Delta}_{t_l})\trans (Q+\theta\trans R\theta) \yb^{\Theta,\Delta}_{t_l}-(\tilde\yb^{\Theta,\Delta}_{t_l})\trans (Q+\theta\trans R\theta) \tilde\yb^{\Theta,\Delta}_{t_l}\big)\Bigg] \Bigg|\\
   &=\Bigg| h\sum_{l=0}^{n-1} e^{-\beta t_l}\E\big[(\yb^{\Theta,\Delta}_{t_l})\trans (Q+\theta\trans R\theta) \yb^{\Theta,\Delta}_{t_l}-(\tilde\yb^{\Theta,\Delta}_{t_l})\trans (Q+\theta\trans R\theta) \tilde\yb^{\Theta,\Delta}_{t_l}\big] \Bigg|\\
   &\leq h\sum_{l=0}^{n-1} \frac{C_1(\theta)}{n}\\
   & \leq \frac{C_1(\theta)}{n}.
    \end{aligned}
\end{equation}

\noindent \emph{Step 3: }
We now establish an upper-bound for $\lvert   (J^{(T)}-\tilde J^{(T),\Delta})(\Theta)\rvert$. Let us recall that
\begin{align*}
    J^{(T)}(\Theta)&=\E\Big[ \int_0^T e^{-\beta t} \big(  ( Y^\Theta_t)\trans (Q+\theta\trans R\theta)Y^\Theta_t+( Z^\Theta_t)\trans (\hat Q+\zeta\trans R\zeta) Z^\Theta_t\big)\d t\Big]+\beta\upsilon(\lambda)\int_0^T e^{-\beta t}\d t.
\end{align*}

Note carefully that $(\tilde\yb^{\Theta,\Delta}_{t_l})_{0 \leq l \leq n}$ (resp. 
$( \tilde\zb^{\Theta,\Delta}_{t_l})_{0 \leq l \leq n}$) is the time discretized version of the process $( Y^\Theta_t)_{t\in[0,T]}$ (resp. $( Z^\Theta_t)_{t\in[0,T]}$) with linear drift and constant diffusion coefficients. Hence, it follows from standard results on the weak approximation error (see e.g. \cite{JMLR:v7:munos06b},\cite{NumericalSolu}) that
\begin{equation}\label{ErrDiscreBound2}
\begin{aligned}
& \Bigg| \E\Big[ \int_0^T e^{-\beta t} \big(  ( Y^\Theta_t)\trans (Q+\theta\trans R\theta) Y^\Theta_t+( Z^\Theta_t)\trans (\hat Q+\zeta\trans R\zeta) Z^\Theta_t\big)\d t\Big] \\ 
& - \E\Big[\sum_{l=0}^{n-1} e^{-\beta t_l} h\big((\tilde\yb^{\Theta,\Delta}_{t_l})\trans (Q+\theta\trans R\theta) \tilde\yb^{\Theta,\Delta}_{t_l}+(\tilde\zb^{\Theta,\Delta}_{t_l})\trans(\hat Q+\zeta\trans R\zeta)\tilde\zb^{\Theta,\Delta}_{t_l}\big)\Big] \Bigg|\leq c_3(\Theta) h, 
\end{aligned}
\end{equation}
for some $\Theta \mapsto c_3(\Theta)$ with at most polynomial growth.
Moreover, standard results on Riemann integrals give $\lvert {h}\sum_{l=0}^{n-1} e^{-\beta t_l}-\int_0^T e^{-\beta t}\d t\rvert\leq C{h}$. 

Combining the two previous bounds with \eqref{ErrDiscreBound1}, we conclude that there exists $c_3=c_3(\Theta)>0$ (with at most polynomial growth in $\Theta$) such that 
$$
\lvert  (J^{(T),\Delta}- J^{(T)})(\Theta)\rvert\leq \lvert  (J^{(T),\Delta}-\tilde J^{(T),\Delta})(\Theta)\rvert+\lvert  (\tilde J^{(T),\Delta}- J^{(T)})(\Theta)\rvert\leq c_3{h}.
$$
Since $\Rc(b)$ is compact, $c_3$ may be considered to depend on $\Theta$ only through $b$. 
\end{proof}
\vspace{3mm}

\noindent \textit{Proof of Proposition \ref{PropDiscre}}
    Recalling that $\P$-a.s. $(\theta+U_i ,\zeta+V_i)\in\Rc(2b)$, it follows from Lemma \ref{ErrJTDelta-JT} and the fact that $\lVert U_i\rVert_F=\lVert V_i\rVert=r$ that $\mathbb{P}$-a.s.
    \begin{align*}
       \lVert (\hat\nabla^{(T)}_\theta -\hat\nabla_\theta^{(T),\Delta}) J(\Theta)\rVert_F&\leq \frac{d}{r^2}\frac{1}{\tilde N}\sum_{i=1}^{\tilde N}\lvert (J^{(T)}-J^{(T),\Delta})(\Theta_i)\rvert\lVert U_i\rVert_F \leq \frac{d}{r}c_3(2b)\frac{T}{n}, \\
          \lVert (\hat\nabla^{(T)}_\zeta -\hat\nabla_\zeta^{(T),\Delta}) J(\Theta)\rVert_F&\leq \frac{d}{r^2}\frac{1}{\tilde N}\sum_{i=1}^{\tilde N}\lvert (J^{(T)}-J^{(T),\Delta})(\Theta_i)\rvert\lVert V_i\rVert_F \leq \frac{d}{r}c_3(2b)\frac{T}{n}.
    \end{align*}
\ep

\subsubsection{Proof of Proposition \ref{PropMCError}} \label{secMcerror} 


Taking the empirical average over the particles in both sides of \eqref{XDeltThet2}, we deduce that the dynamics of the process $(\hat\mu^{\Theta,\Delta,N}_{t_l}:=\frac{1}{N}\sum_{j=1}^N X^{\Theta,\Delta,(j)}_{t_l})_{0\leq l \leq N}$ is given by
\begin{equation}\label{Dymhatmu}
\begin{aligned}    \hat\mu^{\Theta,\Delta,N}_{t_{l+1}}&=\hat\mu^{\Theta,\Delta,N}_{t_l}+((\hat B+D\zeta)\hat\mu^{\Theta,\Delta,N}_{t_l}+D\sqrt{\frac{\lambda}{2}R^{-1}}\bar{\xib}^N_{t_l}){h}+ \sqrt{{h}} \gamma\bar\bw^{N}_l+\sqrt{{h}}\gamma_0\bw^0_l, \\
    \hat\mu^{\Theta,\Delta,N}_{0}&=\hat\mu^N_0:=\frac{1}{N}\sum_{j=1}^N X^{(j)}_0, 
\end{aligned}
\end{equation}

\noindent where $\bar{\xib}^N_{t_l}=\frac{1}{N}\sum_{j=1}^N \xib^{(j)}_{t_l}$ and $\bar\bw^N_l=\frac{1}{N}\sum_{j=1}^N \bw^{(j)}_l$. Then, let us introduce the auxiliary process $(Y^{\Theta,\Delta,(j)}_{t_l})$ defined by 
$$
Y^{\Theta,\Delta,(j)}_{t_l}=X^{\Theta,\Delta,(j)}_{t_l}-\hat{\mu}^{\Theta,\Delta,N}_{t_l}.
$$ 
From \eqref{XDeltThet2} and \eqref{Dymhatmu}, we get
\begin{equation}\label{DymYthetaDelta}
\begin{aligned}
    Y^{\Theta,\Delta,(j)}_{t_{l+1}}&=Y^{\Theta,\Delta,(j)}_{t_l}+\big((B+D\theta)Y^{\Theta,\Delta,(j)}_{t_l}+D\sqrt{\frac{\lambda}{2}R^{-1}}(\xib^{(j)}_{t_l}-\bar{\xib}^N_{t_l})\big){h}+\sqrt{{h}}\gamma(\bw^{(j)}_l- \bar\bw^N_l),\\
     Y^{\Theta,\Delta,(j)}_{0}&=Y^{(j)}_0:=X^{(j)}_0-\hat\mu^N_0.
\end{aligned}
\end{equation}

Note that $(Y^{\Theta, \Delta, (j)})_{1\leq j \leq N}$ (resp.  $\hat \mu^{\Theta,\Delta,N}$) depends on $\Theta$ only through $\theta$ (resp. $\zeta$). In order to simplify the notation, from now on, we will write $(Y^{\theta,\Delta,(j)})_{1\leq j\leq N}$ and $\hat \mu^{\zeta,\Delta,N}$. 

The average cost \eqref{EstimpopDef} can be decomposed as follows
\begin{equation}\label{Estimpop2}
\begin{aligned}
    \Jc^{\Delta,N}_{pop}(\Theta):&=\frac{{h}}{N}\sum_{j=1}^N \Bigg(\sum_{l=0}^{n-1} e^{-\beta t_l} \big((X^{\Theta,\Delta,(j)}_{t_l}-\hat\mu^{\Theta,\Delta,N}_{t_l})\trans (Q+\theta\trans R\theta)(X^{\Theta,\Delta,(j)}_{t_l}-\hat\mu^{\Theta,\Delta,N}_{t_l})\\
    &+(\hat\mu^{\Theta,\Delta,N}_{t_l})\trans (\hat Q+\zeta\trans R\zeta)\hat\mu^{\Theta,\Delta,N}_{t_l}+2(X^{\Theta,\Delta,(j)}_{t_l}-\hat\mu^{\Theta,\Delta,N}_{t_l})\trans\theta\trans R\zeta \hat\mu^{\Theta,\Delta,N}_{t_l}\\
    &+2(\theta(X^{\Theta,\Delta,(j)}_{t_{l}}-\hat\mu^{\Theta,\Delta,N}_{t_l})+\zeta\hat\mu^{\Theta,\Delta,N}_{t_l})\trans R\sqrt{\frac{\lambda}{2}R^{-1}}\xib^{(j)}_{t_l}+\big( \sqrt{\frac{\lambda}{2}R^{-1}}\xib^{(j)}_{t_l}\big)\trans R\sqrt{\frac{\lambda}{2}R^{-1}}\xib^{(j)}_{t_l}\\
        &- \big( \sqrt{\frac{\lambda}{2}R^{-1}}\xib^{(j)}_{t_l}\big)\trans R\sqrt{\frac{\lambda}{2}R^{-1}}\xib^{(j)}_{t_l}-\frac{\lambda}{2}\log(\frac{(\pi\lambda)^m}{\det (R)})\big)\Bigg)\\
        &=\frac{{h}}{N}\sum_{j=1}^N \Bigg(\sum_{l=0}^{n-1} e^{-\beta t_l} \big((Y^{\theta,\Delta,(j)}_{t_l})\trans (Q+\theta\trans R\theta)Y^{\theta,\Delta,(j)}_{t_l}+(\hat\mu^{\zeta,\Delta,N}_{t_l})\trans (\hat Q+\zeta\trans R\zeta)\hat\mu^{\zeta,\Delta,N}_{t_l}\\
        &+2(\theta Y^{\theta,\Delta,(j)}_{t_{l}}+\zeta\hat\mu^{\zeta,\Delta,N}_{t_l})\trans R\sqrt{\frac{\lambda}{2}R^{-1}}\xib^{(j)}_{t_l}+\beta\upsilon(\lambda)\big)\Bigg)\\
        &= \Jc_1(\theta)+\Jc_2(\zeta)+\beta\upsilon(\lambda){h}\sum_{l=0}^{n-1} e^{-\beta t_l}, 
\end{aligned}
\end{equation}
where
\begin{equation}\label{Estimpop2Auxi}
\begin{aligned}
    \Jc_1(\theta) &:=\frac{{h}}{N}\sum_{j=1}^N \sum_{l=0}^{n-1} e^{-\beta t_l} \big((Y^{\theta,\Delta,(j)}_{t_l})\trans (Q+\theta\trans R\theta)Y^{\theta,\Delta,(j)}_{t_l}+2(Y^{\theta,\Delta,(j)}_{t_{l}})\trans \theta\trans R \sqrt{\frac{\lambda}{2}R^{-1}} \xib^{(j)}_{t_l}\big),\\
    \Jc_2(\zeta) &:={h} \sum_{l=0}^{n-1}  e^{-\beta t_l}\big((\hat\mu^{\zeta,\Delta,N}_{t_l})\trans (\hat Q+\zeta\trans R\zeta)\hat\mu^{\zeta,\Delta,N}_{t_l}+2(\hat\mu^{\zeta,\Delta,N}_{t_l})\trans\zeta\trans R \sqrt{\frac{\lambda}{2}R^{-1}} \bar\xib^{N}_{t_l}\big).  
\end{aligned}
\end{equation}

\vspace{1mm}

\noindent \emph{Step 1.} Recall that, for $\Theta=(\theta, \zeta)$, Algorithm \ref{Gradientestim} gives
\begin{align*}
    \tilde\nabla^{\Delta,N,pop}_{\theta} J(\Theta)&=\frac{d}{r^2}\frac{1}{\tilde N}\sum_{i=1}^{\tilde N}\Jc^{\Delta,N,i}_{pop} U_i=\frac{d}{r^2}\frac{1}{\tilde N}\sum_{i=1}^{\tilde N}\Jc^{\Delta,N,i}_{pop} U_i,\\
     \tilde\nabla^{\Delta,N,pop}_{\zeta} J(\Theta)&=\frac{d}{r^2}\frac{1}{\tilde N}\sum_{i=1}^{\tilde N}\Jc^{\Delta, N, i}_{pop} V_i=\frac{d}{r^2}\frac{1}{\tilde N}\sum_{i=1}^{\tilde N}\Jc^{\Delta,N, i}_{pop} V_i,
\end{align*}

\noindent and observe that $\Jc^{\Delta,N,i}_{pop}$ can be decomposed as follows  
$$
\Jc^{\Delta,N,i}_{pop}=\Jc^{\Delta,N,i}_{pop,1}+\Jc^{\Delta,N,i}_{pop,2} +\beta\upsilon(\lambda){h}\sum_{l=1}^n e^{-\beta t_l}
$$ 

\noindent with
\begin{equation}\label{Estimpop12Aux}
    \begin{aligned}
     \Jc^{\Delta,N,i}_{pop,1} & := \frac{{h}}{N}\sum_{j=1}^N \sum_{l=0}^{n} e^{-\beta t_l} \big((Y^{\theta_i,\Delta,(j)}_{t_l})\trans (Q+\theta_i\trans R\theta_i)Y^{\theta_i,\Delta,(j)}_{t_l}+2(Y^{\theta_i,\Delta,(j)}_{t_{l}})\trans \theta_i\trans R \sqrt{\frac{\lambda}{2}R^{-1}} \xib^{i,(j)}_{t_l}\big),\\
     \Jc^{\Delta,N,i}_{pop,2} & := {h} \sum_{l=0}^{n-1}  e^{-\beta t_l}\big((\hat\mu^{\zeta_i,\Delta,N}_{t_l})\trans (\hat Q+\zeta_i\trans R\zeta_i)\hat\mu^{\zeta_i,\Delta,N}_{t_l}+2(\hat\mu^{\zeta_i,\Delta,N}_{t_l})\trans\zeta\trans R \sqrt{\frac{\lambda}{2}R^{-1}} \bar\xib^{i,N}_{t_l}\big),
    \end{aligned}
\end{equation}

 \noindent where $(\theta_i,\zeta_i)=(\theta+U_i,\zeta+V_i)$ and the two processes $(Y^{\theta_i,\Delta,(j)}_{t_l})_{0\leq l \leq n}$, $(\hat{\mu}^{\zeta_i,\Delta,N}_{t_l})_{0\leq l \leq n}$ satisfy
\begin{equation}\label{dynamics:y:thetai:delta:j}
\begin{aligned}
    Y^{\theta_i,\Delta,(j)}_{t_{l+1}}&=Y^{\theta_i,\Delta,(j)}_{t_l}+\big((B+D\theta_i)Y^{\theta_i,\Delta,(j)}_{t_l}+D\sqrt{\frac{\lambda}{2}R^{-1}}(\xib^{i,(j)}_{t_l}-\bar{\xib}^{i,N}_{t_l})\big){h}+\sqrt{{h}}\gamma(\bw^{i,(j)}_l- \bar\bw^{i,N}_l),\\
     Y^{\theta_i,\Delta,(j)}_{0}&=Y^{i,(j)}_0:=X^{i,(j)}_0-\hat\mu^{i,N}_0,\\
\hat\mu^{\zeta_i,\Delta,N}_{t_{l+1}}&=\hat\mu^{\zeta_i,\Delta,N}_{t_l}+((\hat B+D\zeta_i)\hat\mu^{\zeta_i,\Delta,N}_{t_l}+D\sqrt{\frac{\lambda}{2}R^{-1}}\bar{\xib}^{i,N}_{t_l}){h}+ \sqrt{{h}}\gamma \bar\bw^{i,N}_l+\sqrt{{h}}\gamma_0\bw^{i,0}_l, \\
   \hat\mu^{\zeta_i,\Delta,N}_{0}&=\hat\mu^{i,N}_0:=\frac{1}{N}\sum_{j=1}^N X^{i,(j)}_0.
\end{aligned}
\end{equation}
Here, $\left\{(\xib^{i,(j)},\bw^{i,(j)}, X_0^{i,(j)})_{1\leq j \leq N}, \bw^{i,0} ; i= 1, \dots, \tilde N \right\}$ are \emph{i.i.d.} copies of $(\xib^{(j)}, \bw^{(j)}, X_0^{(j)})_{1\leq j \leq N}$,  $\bw^0$ and $\bar\xib^{i,N}=\frac{1}{N}\sum_{j=1}^{N} \xib^{j} $, $\bar{\bw}^{i, N} = \frac{1}{N}\sum_{j=1}^{N} \bw^{j}$.

\vspace{5mm}

\noindent \emph{Step 2:}
We adapt the arguments of  Lemma 35 and Lemma 44 in \cite{carlautan19a}.  
We introduce the sub-exponential norm $\lVert\cdot \rVert_{\psi_1}$ and sub-Gaussian norm $\lVert\cdot \rVert_{\psi_2}$ of the random vector $X$:
$$
\lVert X\rVert_{\psi_1}=\inf\Bigg\{t>0,\E\Bigg[\exp\big(\frac{\lvert X\rvert}{t}\big)\Bigg]\leq 2\Bigg\}\;;\;\lVert X\rVert_{\psi_2}=\inf\Bigg\{t>0,\E\Bigg[\exp\big(\frac{\lvert X\rvert^2}{t^2}\big)\Bigg]\leq 2\Bigg\}.
$$
Recalling \eqref{grad:hat:pop:expression}, one has 
\begin{align*}
    &\lVert  (\tilde\nabla^{\Delta,N,pop}_{\theta} -\hat\nabla^{\Delta,N,pop}_{\theta}) J(\Theta)\rVert_F\\&=\lVert\frac{d}{r^2}\frac{1}{\tilde N}\sum_{i=1}^{\tilde N}\big( \Jc^{\Delta, N, i}_{pop} -\E[\Jc^{\Delta, N, i}_{pop}|U_i,V_i]\big) U_i\rVert_F\\
    &\leq \lVert \frac{d}{r^2}\frac{1}{\tilde N}\sum_{i=1}^{\tilde N}\big( \Jc^{\Delta, N, i}_{pop,1} -\E[\Jc^{\Delta, N, i}_{pop,1}|U_i]\big) U_i\rVert_F+\lVert\frac{d}{r^2}\frac{1}{\tilde N}\sum_{i=1}^{\tilde N}\big( \Jc^{\Delta, N, i}_{pop,2} -\E[\Jc^{\Delta, N, i}_{pop,2}|U_i]\big) U_i\rVert_F, 
\end{align*}
and 
\begin{align*}
    &\lVert  (\tilde\nabla^{\Delta,N,pop}_{\zeta} -\hat\nabla^{\Delta,N,pop}_{\zeta}) J(\Theta)\rVert_F\\&=\lVert\frac{d}{r^2}\frac{1}{\tilde N}\sum_{i=1}^{\tilde N}\big( \Jc^{\Delta, N, i}_{pop} -\E[\Jc^{\Delta, N, i}_{pop}|U_i,V_i]\big) V_i\rVert_F\\
    &\leq \lVert\frac{d}{r^2}\frac{1}{\tilde N}\sum_{i=1}^{\tilde N}\big( \Jc^{\Delta, N, i}_{pop,1} -\E[\Jc^{\Delta, N, i}_{pop,1}|V_i]\big) V_i\rVert_F+\lVert\frac{d}{r^2}\frac{1}{\tilde N}\sum_{i=1}^{\tilde N}\big( \Jc^{\Delta, N, i}_{pop,2} -\E[\Jc^{\Delta, N, i}_{pop,2}|V_i]\big) V_i\rVert_F.
\end{align*}
We now provide some upper-estimate on the probabilities $\mathbb{P}(\lVert \frac{d}{r^2}\frac{1}{\tilde N}\sum_{i=1}^{\tilde N}\big( \Jc^{\Delta, N, i}_{pop,\ell} -\E[\Jc^{\Delta, N, i}_{pop,\ell}|U_i]\big) U_i\rVert_F\geq \varepsilon)$ and $\mathbb{P}(\lVert \frac{d}{r^2}\frac{1}{\tilde N}\sum_{i=1}^{\tilde N}\big( \Jc^{\Delta, N, i}_{pop,\ell} -\E[\Jc^{\Delta, N, i}_{pop, \ell}|V_i]\big) V_i\rVert_F\geq \varepsilon)$ for $\ell=1,2$.

We only deal with $\mathbb{P}(\lVert \frac{d}{r^2}\frac{1}{\tilde N}\sum_{i=1}^{\tilde N}\big( \Jc^{\Delta, N, i}_{pop, 1} -\E[\Jc^{\Delta, N, i}_{pop, 1}|U_i]\big) U_i\rVert_F\geq \varepsilon)$ inasmuch the proof for the other terms are similar.

For $i=1,\dots,\tilde N$, $j=1,\dots,N$ and $l=0,\dots,n-1$, we let 
\begin{align*}
    \eta^{i,j}_l&=(Y^{\theta_i,\Delta,(j)}_{t_l})\trans (Q+\theta_i\trans R\theta_i)Y^{\theta_i,\Delta,(j)}_{t_l}+2(Y^{\theta_i,\Delta,(j)}_{t_{l}})\trans \theta_i\trans R \sqrt{\frac{\lambda}{2}R^{-1}} \xib^{i,(j)}_{t_l}\\
    &-\E[(Y^{\theta_i,\Delta,(j)}_{t_l})\trans (Q+\theta_i\trans R\theta_i)Y^{\theta_i,\Delta,(j)}_{t_l}+2(Y^{\theta_i,\Delta,(j)}_{t_{l}})\trans \theta_i\trans R \sqrt{\frac{\lambda}{2}R^{-1}} \xib^{i,(j)}_{t_l}|U_i] \\
    &=(Y^{\theta_i,\Delta,(j)}_{t_l})\trans (Q+\theta_i\trans R\theta_i)Y^{\theta_i,\Delta,(j)}_{t_l}+2(Y^{\theta_i,\Delta,(j)}_{t_{l}})\trans \theta_i\trans R \sqrt{\frac{\lambda}{2}R^{-1}} \xib^{i,(j)}_{t_l}\\
    &-\E[(Y^{\theta_i,\Delta,(j)}_{t_l})\trans (Q+\theta_i\trans R\theta_i)Y^{\theta_i,\Delta,(j)}_{t_l}|U_i]. 
\end{align*}
\noindent where the second equality come from the fact that, conditionally on $U_i$, the two random variables $Y_{t_l}^{\theta_i, \Delta, (j)}$ and $\xib^{i,(j)}_{t_l}$ are independent and that $\mathbb{E}[Y_{t_l}^{\theta_i, \Delta, (j)}|U_i] = \mathbb{E}[Y_{t_l}^{\theta, \Delta, (j)}]_{|\theta=\theta_i}=0$.
From \eqref{Estimpop12Aux} and the definition of $\eta^{i,j}_l$, we get 
\begin{align} 
\Jc^{\Delta, N, i}_{pop,1}-\E[\Jc^{\Delta, N, i}_{pop,1}|U_i] & = \; \frac{h}{N}\sum_{j=1}^N\sum_{l=0}^{n-1} e^{-\beta t_l}\eta^{i,j}_{l}
\end{align} 
so that
\begin{align*}
     & \lVert\frac{d}{r^2}\frac{1}{\tilde N} \sum_{i=1}^{\tilde N}\big(\Jc^{\Delta, N, i}_{pop,1}-\E[\Jc^{\Delta, N, i}_{pop,1}|U_i]\big) U_i\rVert_F \\
    & =  \; \lVert\frac{d}{r^2}\frac{1}{\tilde N}\sum_{i=1}^{\tilde N}\big( \frac{h}{N}\sum_{j=1}^N\sum_{l=0}^{n-1} e^{-\beta t_l}\eta^{i,j}_{l}\big) U_i\rVert_F\\
   & =  \; \frac{hd}{r^2}\lVert\big( \frac{1}{N}\sum_{j=1}^N\sum_{l=0}^{n-1}  \frac{e^{-\beta t_l}}{\tilde N}\sum_{i=1}^{\tilde N} \eta^{i,j}_{l}\big) U_i\rVert_F
   \; \\
   & \leq  \;  \frac{Td}{r^2}\sup_{1 \leq j\leq N, 0\leq l \leq n-1} \lVert \frac{e^{-\beta t_l}}{\tilde N}\sum_{i=1}^{\tilde N} \eta^{i,j}_l U_i\rVert_F.
\end{align*}

 \noindent Hence, noting that the random variables $(\eta^{i,j})_{1\leq j \leq N}$ have the same law and letting $\delta=\frac{r^2\varepsilon}{2Td}$, the previous estimate yields 
 \begin{equation}\label{BerInequalityAppliedForProbaeta}
     \begin{aligned}
              &\P\Bigg(  \lVert\frac{d}{r^2}\frac{1}{\tilde N}\sum_{i=1}^{\tilde N}\big( \Jc^{\Delta, N, i}_{pop,1}-\E[\Jc^{\Delta, N, i}_{pop,1}|U_i]\big) U_i\rVert_F\geq\frac{\varepsilon}{2}\Bigg)\\
     & \leq  \;  \P\Bigg( \frac{Td}{r^2}\sup_{j=1,\dots,N}\sup_{l=0,\dots,n-1} \lVert \frac{e^{-\beta t_l}}{\tilde N}\sum_{i=1}^{\tilde N} \eta^{i,j}_l U_i\rVert_F\geq \frac{\varepsilon}{2}\Bigg)\\
     & =  \; \P\Bigg( \sup_{j=1,\dots,N}\sup_{l=0,\dots,n-1} \lVert \frac{e^{-\beta t_l}}{\tilde N}\sum_{i=1}^{\tilde N} \eta^{i,j}_l U_i\rVert_F\geq \delta\Bigg)\\
     & \leq  \; \sum_{j=1}^N\sum_{l=0}^{n-1} \P\Bigg( \lVert \frac{e^{-\beta t_l}}{\tilde N}\sum_{i=1}^{\tilde N} \eta^{i,j}_l U_i\rVert_F\geq \delta\Bigg) \\
     & =  \; N \sum_{l=0}^{n-1}\P\Bigg( \lVert \frac{1}{\tilde N}\sum_{i=1}^{\tilde N} (e^{-\beta t_l}\eta^{i,1}_l) U_i\rVert_F\geq \delta\Bigg).
     \end{aligned}
 \end{equation}

 Now, it easily follows from standard computations based on the dynamics \eqref{dynamics:y:thetai:delta:j} of $Y^{\theta_i, \Delta, (1)}$ and since the  $\theta_i$ are bounded that the random variables $Y^{\theta_i, \Delta, (1)}_{t_l}$ have finite sub-Gaussian norm $\|Y^{\theta_i, \Delta, (1)}_{t_l}\|_{\Psi_2}$ and since $\|\xib^{i,(j)}_{t_l}\|_{\Psi_2}<\infty$ we deduce that $\|\eta_l^{i, 1}\|_{\Psi_1}<\infty$ uniformly in $(l,i)$. By Lemma 37 in \cite{carlautan19a}, we get
\begin{align} \label{BerInequalityApplied}
   \P\Bigg(  \lVert \frac{1}{\tilde N}\sum_{i=1}^{\tilde N} \eta^{i,1}_l U_i\rVert_F\geq \delta\Bigg)\leq 2d\exp\big(-\frac{\tilde N\delta^2}{2(\chi_l^2+\chi_l\delta)}\big), 
\end{align}
 where $\chi_l:=C_1r\sup_{\theta'\in\B(\theta,r)}\lVert\eta^{\theta',1}_l\rVert_{\psi_1}$ and $C_1$ is a universal constant\footnote{Referring to \cite{carlautan19a}, one would use the standard operator norm $\lVert\cdot\rVert$ in \eqref{BerInequalityApplied}. However, due to the equivalence of the matrix norms, we can replace $\lVert\cdot\rVert$ by the Frobenius norm $\lVert\cdot\rVert_F$, up to a modification of the constant $C_1>0$ which may depend on the dimensions $m$ and $d$.}. Hence, combining \eqref{BerInequalityAppliedForProbaeta} and \eqref{BerInequalityApplied}, we obtain   
\begin{align}\label{BerInequalityCombined}
\P\Bigg(  \lVert\frac{d}{r^2}\frac{1}{\tilde N}\sum_{i=1}^{\tilde N}\big( \Jc^{\Delta, N, i}_{pop,1}-\E[\Jc^{\Delta, N, i}_{pop,1}|U_i]\big) U_i\rVert_F\geq\frac{\varepsilon}{2}\Bigg) \leq 2dNn\exp\Big(-\frac{\tilde N\delta^2}{2(\check\chi^2+\check\chi\delta)}\Big) 
\end{align}
 
\noindent with $\check\chi=\sup_{0\leq l\leq n-1}\big(e^{-\beta t_l}\chi_l\big)$. 

Then, from Lemma \ref{AuxLemmaPropSampleError} (see below), taking $r\leq \min\{\check{r}(b),\inf_{\Theta\in\Rc(b)}\frac{1}{\h_1(\Theta)}\}$ (recalling that $\Rc(b)\subseteq\Sc(b)\times\hat\Sc(b)$ is compact and that according to Proposition \ref{Bds}, $\lVert \theta\rVert_F\leq\Bd_\theta(b)$) and letting
$$
\hat\chi(b)= C\check r(b) (1 + \Bd_\theta(b)^2 + \check{r}(b)^2)),
$$

\noindent one has $\check\chi\leq \hat\chi(b)$ uniformly for all $\Theta\in\Rc(b)$.

The previous bound on $\check\chi$ together with \eqref{BerInequalityCombined} implies that if 
$$
\tilde N\geq \h^{(1,1)}_{\tilde N}(r,b):=2\frac{(\hat\chi(b)^2+\hat\chi(b)\delta)}{\delta^2}\Big((d+1)\log(\frac{d}{\varepsilon})+\log N+\log n+\log(4\varepsilon)\Big),
$$

\noindent then, for all $\Theta\in\Rc(b)$, one has
$$
\P\Bigg(  \lVert\frac{d}{r^2}\frac{1}{\tilde N}\sum_{i=1}^{\tilde N}\big( \Jc^{\Delta, N, i}_{pop,1}-\E[\Jc^{\Delta, N, i}_{pop,1}|U_i]\big) U_i\rVert_F\geq\frac{\varepsilon}{2}\Bigg)\leq\frac{1}{2}\Bigg(\frac{d}{\varepsilon}\Bigg)^{-d}.
$$
Similarly, there exists $\h^{(1,2)}_{\tilde N}(r,b)$ such that for any $\tilde N\geq \h^{(1,2)}_{\tilde N}(r,b)$ and any $\Theta\in\Rc(b)$,
$$
\P\Bigg(  \lVert\frac{d}{r^2}\frac{1}{\tilde N}\sum_{i=1}^{\tilde N}\big( \Jc^{\Delta, N, i}_{pop,2}-\E[\Jc^{\Delta, N, 2}_{pop,1}|U_i]\big) U_i\rVert_F\geq\frac{\varepsilon}{2}\Bigg)\leq\frac{1}{2}\Bigg(\frac{d}{\varepsilon}\Bigg)^{-d}.
$$

 Hence, choosing $\tilde N\geq \max\{ \h^{(1,1)}_{\tilde N}(r,b),\h^{(1,2)}_{\tilde N}(r,b)\}$, we get 
\begin{align*}
    \P\Bigg(\lVert  \tilde\nabla^{\Delta,N,pop}_{\theta} J(\Theta)-\hat\nabla^{\Delta,N,pop}_{\theta} J(\Theta)\rVert_F\geq \varepsilon\Bigg)  &\leq \P\Bigg(  \lVert\frac{d}{r^2}\frac{1}{\tilde N}\sum_{i=1}^{\tilde N}\big( \Jc^{\Delta, N, i}_{pop,1}-\E[\Jc^{\Delta, N, i}_{pop,1}|U_i]\big) U_i\rVert_F\geq\frac{\varepsilon}{2}\Bigg)\\
    &+\P\Bigg(  \lVert\frac{d}{r^2}\frac{1}{\tilde N}\sum_{i=1}^{\tilde N}\big(\Jc^{\Delta, N, i}_{pop,2}-\E[\Jc^{\Delta, N, i}_{pop,2}|U_i]\big) U_i\rVert_F\geq\frac{\varepsilon}{2}\Bigg)\\
    &\leq \Bigg(\frac{d}{\varepsilon}\Bigg)^{-d}.
\end{align*}

The proof of the upper-bound on $\P(\lVert  \tilde\nabla^{\Delta,N,pop}_{\zeta} J(\Theta)-\hat\nabla^{\Delta,N,pop}_{\zeta} J(\Theta)\rVert_F\geq \varepsilon )$ being similar is omitted. The proof of Proposition \ref{PropMCError} is now complete.
\ep

 \begin{Lemma}\label{AuxLemmaPropSampleError}
     There exists $\h_1(\|\theta\|)$ with at most of polynomial growth in $\|\theta\|$ and a constant $C<\infty$ such that for all $r\leq \frac{1}{\h_1(\Theta)}$, it holds
     $$
     \check{\chi}\leq C r( 1 + \|\theta\|_F^2 + r^2).
     $$
 \end{Lemma}
\begin{proof}
Let
 $$
 \Psi^{\theta'}_l=(Y^{\theta',\Delta,(1)}_{t_l})\trans (Q+(\theta')\trans R\theta')Y^{\theta',\Delta,(1)}_{t_l}+2(Y^{\theta',\Delta,(1)}_{t_{l}})\trans (\theta')\trans R \sqrt{\frac{\lambda}{2}R^{-1}} \xib^{(1)}_{t_l}, \quad \theta'\in\B(\theta,r),
 $$
 
 \noindent recalling that $Y^{\theta',\Delta,(1)}$ is given by \eqref{DymYthetaDelta} with $\theta=\theta'$.

 Since 
 $$
 \lVert\Psi^{\theta'}_l-\E[\Psi^{\theta'}_l]\rVert_{\psi_1}\leq 2 \lVert \Psi^{\theta'}_l\rVert_{\psi_1},
 $$
 \noindent and $\eta^{\theta', 1}_l = \Psi^{\theta'}_l-\E[\Psi^{\theta'}_l]$, we get
 $$
 \check{\chi} \leq 2C_1 r \sup_{0\leq l\leq n-1, \theta'\in\B(\theta,r)}\big( e^{-\beta t_l} \lVert\Psi^{\theta'}_l\rVert_{\psi_1}\big).
 $$

\noindent The triangle inequality for the sub-exponential norm $\|.\|_{\Psi_1}$ yields
\begin{equation}\label{triangle:ineq:psi1:norm}
\lVert \Psi^{\theta'}_l\rVert_{\psi_1}\leq \lVert (Y^{\theta',\Delta,(1)}_{t_l})\trans (Q+(\theta')\trans R\theta')Y^{\theta',\Delta,(1)}_{t_l}\rVert_{\psi_1}+2\lVert(Y^{\theta',\Delta,(1)}_{t_{l}})\trans (\theta')\trans R \sqrt{\frac{\lambda}{2}R^{-1}} \xib^{(1)}_{t_l} \rVert_{\psi_1}.
\end{equation}

\noindent From Proposition 2.4 \cite{BoundTail},  
\begin{align*}
    \lVert (Y^{\theta',\Delta,(1)}_{t_l})\trans (Q+(\theta')\trans R\theta')Y^{\theta',\Delta,(1)}_{t_l}\rVert_{\psi_1}&\leq \text{Tr}(Q+(\theta')\trans R\theta')\lVert Y^{\theta',\Delta,(1)}_{t_l}\rVert_{\psi_2}^2\\
    &\leq Cd(\lVert Q\rVert_F+\lVert R\rVert_F(\lVert \theta\rVert_F+r)^2) \lVert Y^{\theta',\Delta,(1)}_{t_l}\rVert_{\psi_2}^2,
\end{align*}

\noindent and, from the inequality $\|XY\|_{\Psi_1}\leq \|X\|_{\Psi_2} \|Y\|_{\Psi_2}$ stemming from standard Young's inequality, 
\begin{align*}
    \lVert(Y^{\theta',\Delta,(1)}_{t_{l}})\trans (\theta')\trans R \sqrt{\frac{\lambda}{2}R^{-1}} \xib^{(1)}_{t_l} \rVert_{\psi_1}&=  \lVert\big( \theta' Y^{\theta',\Delta,(1)}_{t_{l}}\big)\cdot\big(R \sqrt{\frac{\lambda}{2}R^{-1}} \xib^{(1)}_{t_l} \big) \rVert_{\psi_1}\\
    &\leq \lVert \theta' Y^{\theta',\Delta,(1)}_{t_{l}}\rVert_{\psi_2}\lVert R \sqrt{\frac{\lambda}{2}R^{-1}} \xib^{(1)}_{t_l} \rVert_{\psi_2}\\
     &\leq C\lVert\theta'\rVert_F\lVert   Y^{\theta',\Delta,(1)}_{t_{l}}\rVert_{\psi_2}\lVert R \sqrt{\frac{\lambda}{2}R^{-1}} \xib^{(1)}_{t_l} \rVert_{\psi_2}\\
    &\leq C (\lVert\theta\rVert_F+r)\lVert Y^{\theta',\Delta,(1)}_{t_l}\rVert_{\psi_2} 
\end{align*}

\noindent for some constant $C=C(\lambda,\lVert R\rVert_F)$.


 Similarly to Lemma 36 \cite{carlautan19a}, we can establish an universal upper-bound for the sub-Gaussian norm of $Y^{\theta',\Delta,(j)}_{t_l}$ for all $\theta'\in\B(\theta,r)$, where $r\leq 1/\h_1(\|\theta\|)$, $\h_1(\|\theta\|)$ being at most of polynomial growth in $\|\theta\|$. Namely, it holds 
 \begin{align*}
    \sup_{0\leq l \leq n-1} e^{-\beta t_l}\lVert Y^{\theta',\Delta,(j)}_{t_l}\rVert_{\psi_2}&\leq C \sup\{ \lVert Y^{(j)}_0\rVert_{\psi_2}, \sup_{0\leq l \leq n-1} \lVert \bw^{(j)}_l- \bar\bw^N_l\rVert_{\psi_2}, \sup_{0\leq l \leq n-1}\lVert \xib^{(j)}_l- \bar\xib^N_l\rVert_{\psi_2}\}.
\end{align*}

Plugging the three previous inequalities into \eqref{triangle:ineq:psi1:norm}, we get 
$$
\lVert \Psi^{\theta'}_l\rVert_{\psi_1} \leq C ( 1 + \|\theta\|_F^2 + r^2)
$$

 \noindent which eventually yields
 \begin{align*}
    \check\chi&\leq 2C_1 r \sup_{l=1,\dots,n}\sup_{\theta'\in\B(\theta,r)}\big( e^{-\beta t_l} \lVert\Psi^{\theta'}_l\rVert_{\psi_1}\big)\leq C r( 1 + \|\theta\|_F^2 + r^2),
     \end{align*}
\noindent for all $r\leq 1/h_1(\|\theta\|)$. 
\end{proof}



\subsubsection{Proof of  Proposition \ref{PropN}.} \label{secparticle}

Recalling that $\bar \Jc^{\Delta,N}_{pop}(\Theta)=\E[\Jc^{\Delta,N}_{pop}(\Theta)]$ and using the fact that $(U_i, V_i)_{1\leq i\leq N}$ is independent of $((X_0^{(j)},(\xib^{(j)}_{t_l})_{0\leq l \leq n-1}, W^{(j)})_{j=1,\dots,N}, W^0)$, we get $\E[\Jc^{\Delta,N, i}_{pop}|U_i,V_i]= \bar \Jc^{\Delta,N}_{pop}(\Theta_i)$ and
\begin{align*}
    \hat\nabla^{\Delta,N,pop}_{\theta} J(\Theta)&=\frac{d}{r^2}\frac{1}{\tilde N}\sum_{i=1}^{\tilde N}\bar\Jc^{\Delta,N}_{pop}(\theta+U_i,\zeta +V_i) U_i,\\
     \hat\nabla^{\Delta,N,pop}_{\zeta} J(\Theta)&=\frac{d}{r^2}\frac{1}{\tilde N}\sum_{i=1}^{\tilde N}\bar\Jc^{\Delta,N}_{pop}(\theta+U_i,\zeta+ V_i)V_i. 
\end{align*}

\noindent \emph{Step 1:} We introduce the process $(\xb^{\Theta,\Delta,(j)})_{j=1,\dots,N}$ with dynamics
    \begin{align*}
            \xb^{\Theta,\Delta,(j)}_{t_{l+1}}&=\xb^{\Theta,\Delta,(j)}_{t_l}+(B\xb^{\Theta,\Delta, (j)}_{t_l}+\bar B\E_0[\xb^{\Theta,\Delta, (j)}_{t_l}]+D\boldsymbol{\alpha}^{\Theta,\Delta,(j)}_{t_l}){h}+\sqrt{{h}} \gamma\bw^{(j)}_l+ \sqrt{{h}}\gamma_0\bw^0_l,\\
            \xb^{\Theta,\Delta,(j)}_{0}&=X_0^{(j)},
    \end{align*} 
    
    \noindent with
    $$
    \boldsymbol{\alpha}^{\Theta,\Delta,(j)}_{t_l}=\theta(\xb^{\Theta,\Delta,(j)}_{t_l}-\E_0[\xb^{\Theta,\Delta, (j)}_{t_l}])+\zeta\E_0[\xb^{\Theta,\Delta, (j)}_{t_l}]+\sqrt{\frac{\lambda}{2}R^{-1}}{\xib}^{(j)}_{t_l},
    $$ 
    and where the random variables $(X_0^{(j)})_{1\leq j \leq N}$, $({\xib}^{(j)}:=({\xib}^{(j)}_{t_l})_{0\leq l\leq n-1})_{1\leq j\leq N}$, $(\bw^{(j)})_{1\leq j \leq N}$ and $\bw^0$ are the same as those employed in \eqref{ExecActRandom}-\eqref{XDeltThet2}.

    
    
We then define the two processes $\yb^{\theta,\Delta,(j)}$ and $\zb^{\zeta,\Delta}$ by
    \begin{equation*}
        \yb^{\Theta,\Delta,(j)}_{t_l}=\xb^{\Theta,\Delta,(j)}_{t_l}-\E_0[\xb^{\Theta,\Delta,(j)}_{t_l}], \quad \zb^{\Theta,\Delta}_{t_l}=\E_0[\xb^{\Theta,\Delta}_{t_l}], 
    \end{equation*}
    \noindent with $\yb^{\Theta,\Delta,(j)}_{0}= X_0^{(j)}-\E[X_0]$ and $\zb^{\Theta,\Delta}_{0}=\E[X_0]$.
    Their dynamics are given by
\begin{equation}\label{DynyzbThetaDelta}
\begin{aligned}
        \yb^{\Theta,\Delta,(j)}_{t_{l+1}}&=\yb^{\Theta,\Delta,(j)}_{t_l}+\big((B+D\theta)\yb^{\Theta,\Delta,(j)}_{t_l}+ D\sqrt{\frac{\lambda}{2}R^{-1}}{\xib}^{(j)}_{t_l} \big){h}+\gamma\sqrt{{h}}\bw^{(j)}_l,\\
\zb^{\Theta,\Delta}_{t_{l+1}}&=\zb^{\Theta,\Delta}_{t_l}+(\hat B+D\zeta)\zb^{\Theta,\Delta}_{t_l}{h}+\gamma_0\sqrt{{h}}\bw^0_{l}.
\end{aligned}
\end{equation}

    \noindent We again notice that $\yb^{\Theta,\Delta,(j)}$ (resp. $\zb^{\Theta,\Delta}$) depends only on $\theta$ (resp. $\zeta$). Hence, in order to simplify the notation, from now on, we will write $\yb^{\theta,\Delta,(j)}$ and $\zb^{\zeta,\Delta}$. Note that $(\yb^{\Theta,\Delta,(j)}_{t_l})_{j=1,\dots N}$ are \emph{i.i.d.} copies of $\yb^{\Theta,\Delta}$ and $\zb^{\zeta,\Delta} = \zb^{\Theta,\Delta}$ where $(\yb^{\Theta,\Delta},\zb^{\Theta,\Delta})$ are defined by \eqref{yzDeltTheta}. Hence,
 \begin{equation}\label{DiscApproFC3}
    \begin{aligned}
       J^{(T),\Delta,N}(\Theta)  & :=\frac{{h}}{N}\sum_{j=1}^N\E\Bigg[\sum_{l=0}^{n-1} e^{-\beta t_l}((\xb^{\Theta,\Delta,(j)}_{t_l}-\E_0[\xb^{\Theta,\Delta}_{t_l}])\trans Q (\xb^{\Theta,\Delta,(j)}_{t_l}-\E_0[\xb^{\Theta,\Delta}_{t_l}]) \\
       & \;  + \;  \E_0[\xb^{\Theta,\Delta,(j)}_{t_l}]\trans\hat Q\E_0[\xb^{\Theta,\Delta}_{t_l}]
        +(\boldsymbol{\alpha}^{\Theta,\Delta,(j)}_{t_l})\trans R \boldsymbol{\alpha}^{\Theta,\Delta,(j)}_{t_l}+\lambda\log\pi^\Theta(\boldsymbol{\alpha}^{\Theta,\Delta,(j)}_{t_l}|\xb^{\Theta,\Delta,(j)}_{t_l},\E_0[\xb^{\Theta,\Delta}_{t_l}])\Bigg]\\
        &=\frac{{h}}{N}\sum_{j=1}^N\E\Bigg[\sum_{l=0}^{n-1} e^{-\beta t_l}\big((\yb^{\theta,\Delta,(j)}_{t_l})\trans (Q+\theta\trans R\theta) \yb^{\theta,\Delta,(j)}_{t_l}+(\zb^{\zeta,\Delta}_{t_l})\trans(\hat Q+\zeta\trans R\zeta)\zb^{\zeta,\Delta}_{t_l}+\beta\upsilon(\lambda)\big)\Bigg]\\
        & = J^{(T),\Delta}(\Theta),
    \end{aligned}
\end{equation}
recalling that $J^{(T),\Delta}(\Theta)$ is defined by \eqref{DiscApproFC}. We will also use the decomposition $J^{(T),\Delta,N}= J^{(T),\Delta}_1 + J^{(T),\Delta}_2 + +\beta\upsilon(\lambda){h}\sum_{l=1}^n e^{-\beta t_l}$ with
\begin{align*}
    J^{(T),\Delta}_1(\theta)& := \frac{{h}}{N}\sum_{j=1}^N\sum_{l=0}^{n-1} e^{-\beta t_l}\E\Bigg[(\yb^{\theta,\Delta,(j)}_{t_l})\trans (Q+\theta\trans R\theta) \yb^{\theta,\Delta,(j)}_{t_l}\Bigg],\\
    J^{(T),\Delta}_2(\zeta)& : =  {h}\sum_{l=0}^{n-1} e^{-\beta t_l}\E\Bigg[ (\zb^{\zeta,\Delta}_{t_l})\trans(\hat Q+\zeta\trans R\zeta)\zb^{\zeta,\Delta}_{t_l}\Bigg]. 
\end{align*}
Then, according to \eqref{Estimpop2} and \eqref{Estimpop2Auxi}, 
$$
\bar\Jc^{\Delta,N}_{pop}(\Theta)=\E[\Jc^{\Delta,N}_{pop}(\Theta)]=\bar\Jc_1(\theta)+\bar\Jc_2(\zeta)+\beta\upsilon(\lambda){h}\sum_{l=1}^n e^{-\beta t_l}
$$
where
\begin{equation*}
\begin{aligned}
    \bar\Jc_1(\theta)&=\E[\Jc_1(\theta)]=\frac{{h}}{N}\sum_{j=1}^N \sum_{l=0}^{n-1} e^{-\beta t_l} \E[(Y^{\theta,\Delta,(j)}_{t_l})\trans (Q+\theta\trans R\theta)Y^{\theta,\Delta,(j)}_{t_l}]\\
    \bar\Jc_2(\zeta)&=\E[\Jc_2(\zeta)]={h} \sum_{l=0}^{n-1}  e^{-\beta t_l}\E\big[(\hat\mu^{\zeta,\Delta,N}_{t_l})\trans (\hat Q+\zeta\trans R\zeta)\hat\mu^{\zeta,\Delta,N}_{t_l}\big]. 
\end{aligned}
\end{equation*}

Hence, in order to establish an upper-bound for $(J^{(T),\Delta,N}-\bar\Jc^{\Delta,N}_{pop})(\Theta) = (J^{(T),\Delta}_1 -\bar\Jc_1)(\theta) + (  J^{(T),\Delta}_2 -\bar\Jc_2)(\zeta)$, we need to quantity the $L^{p}(\mathbb{P})$ error for the difference $\yb^{\theta,\Delta,(j)}_{t_l}-Y^{\theta,\Delta,(j)}_{t_l}$ and $\zb^{\zeta,\Delta}_{t_l}-\mu^{\zeta,\Delta,N}_{t_l}$, $j=1,\dots,N$ and $l=0,\cdots, n-1$. 

\vspace{1mm}

\noindent \emph{Step 2:} Let
\begin{align*}
    \Qc^{\theta,\Delta,(j)}_{t_l}&=\yb^{\theta,\Delta,(j)}_{t_l}-Y^{\theta,\Delta,(j)}_{t_l}, \quad j=1,\dots,N,\\    \hat\Qc^{\zeta,\Delta,N}_{t_l}&=\zb^{\zeta,\Delta}_{t_l}-\hat \mu^{\zeta,\Delta,N}_{t_l}.
\end{align*}
From \eqref{Dymhatmu}-\eqref{DymYthetaDelta} and \eqref{DynyzbThetaDelta}, we easily get
\begin{equation}\label{DYDZDim}
    \begin{aligned}
\Qc^{\theta,\Delta, (j)}_{t_{l+1}}&=\Qc^{\theta,\Delta, (j)}_{t_l}+{h}\big((B+D\theta)\Qc^{\theta,\Delta, (j)}_{t_l}+D\sqrt{\frac{\lambda}{2}R^{-1}}\bar\xib^N_{t_l}\big)+\sqrt{{h}}\gamma\bar\bw^N_l,\\
\hat\Qc^{\zeta,\Delta,N}_{t_{l+1}}&=\hat\Qc^{\zeta,\Delta,N}_{t_l}+{h}\big((\hat B+D\zeta)\hat\Qc^{\zeta,\Delta,N}_{t_l}-D\sqrt{\frac{\lambda}{2}R^{-1}}\bar\xib^N_{t_l}\big)-\sqrt{{h}}\gamma\bar\bw^N_l, 
    \end{aligned}
\end{equation}

\noindent with $\Qc^{\theta,\Delta, (j)}_{0}=\Qc_0:=\hat\mu^N_0-\E[X_0]$, $\hat\Qc^{\zeta,\Delta}_{0}=\hat\Qc_0:=\E[X_0]-\hat\mu^N_0$. Note carefully that for any $l=0,\dots, n$, the random variables $(\Qc^{\theta,\Delta, (j)}_{t_l})_{1\leq j \leq N}$ are equal. We thus omit the superscript $j$ and simply write $\Qc^{\theta,\Delta}_{t_l}$ in what follows. 

From \eqref{DYDZDim}, we get
\begin{align*}
|\Qc^{\theta,\Delta, (j)}_{t_{l+1}}|^2 & = |\Qc^{\theta,\Delta, (j)}_{t_{l}}|^2 + 2 \langle \Qc^{\theta,\Delta, (j)}_{t_{l}}, h(B+D\theta)) \Qc^{\theta,\Delta, (j)}_{t_{l}} + h D\sqrt{\frac{\lambda}{2}R^{-1}}\bar\xib^N_{t_l} + \sqrt{{h}}\gamma\bar\bw^N_l\rangle \\
& \quad + \Big|h(B+D\theta)) \Qc^{\theta,\Delta, (j)}_{t_{l}}+ h D\sqrt{\frac{\lambda}{2}R^{-1}}\bar\xib^N_{t_l} + \sqrt{{h}}\gamma\bar\bw^N_l\Big|^2\\
& \leq (1+h C_1(\theta)) |\Qc^{\theta,\Delta, (j)}_{t_{l}}|^2 + 2 \langle \Qc^{\theta,\Delta, (j)}_{t_{l}}, h D\sqrt{\frac{\lambda}{2}R^{-1}}\bar\xib^N_{t_l} + \sqrt{{h}}\gamma\bar\bw^N_l\rangle \\
& \quad + 2\Big|h D\sqrt{\frac{\lambda}{2}R^{-1}}\bar\xib^N_{t_l} + \sqrt{{h}}\gamma\bar\bw^N_l\Big|^2,
\end{align*}
\noindent for some $\theta\mapsto C_1(\theta)$ with at most of quadratic growth in $\|\theta\|$. Taking expectation in both sides of the previous inequality and recalling that $\Qc^{\theta,\Delta, (j)}_{t_{l}}$ is independent of $\bar\xib^N$ and $\bar\bw^N$, we get
\begin{align*}
\mathbb{E}[|\Qc^{\theta,\Delta, (j)}_{t_{l+1}}|^2]
& \leq (1+h C_1(\theta)) \mathbb{E}[|\Qc^{\theta,\Delta, (j)}_{t_{l}}|^2] + 2 \E\Big|h D\sqrt{\frac{\lambda}{2}R^{-1}}\bar\xib^N_{t_l} + \sqrt{{h}}\gamma\bar\bw^N_l\Big|^2\\
&\leq (1+h C_1(\theta)) \mathbb{E}[|\Qc^{\theta,\Delta, (j)}_{t_{l}}|^2] + C_2 \frac{h}{N},
\end{align*}

\noindent for some constant $C_2<\infty$ and where, for the last inequality, we used the fact that $\E|\bar\xib^N_{t_l}|^2 = m/N$ and $\E|\bar\bw^N_l|^2 = d/N$, $l=0,\dots, n-1$. From the discrete Gr\"onwall lemma, up to a modification of $C_2$, we get
$$
\mathbb{E}[|\Qc^{\theta,\Delta, (j)}_{t_{l}}|^2] \leq (1+hC_1(\theta))^{l} \Big(\mathbb{E}[|\Qc^{\theta,\Delta, (j)}_{0}|^2] + \frac{C_2}{C_1(\theta) N} \Big).
$$
Now, since $\mathbb{E}[|\Qc^{\theta,\Delta, (j)}_{0}|^2] = \mbox{Var}(X_0)/N$, using the standard inequality $1+x\leq \exp(x)$, $x\in \mathbb{R}$, the previous inequality eventually implies
$$
\sup_{0\leq l \leq n} \mathbb{E}[|\Qc^{\theta,\Delta, (j)}_{t_{l}}|^2]\leq \frac{\exp(C_1(\theta))}{N},
$$
up to a modification of $C_1(\theta)$ (with at most of quadratic growth).
Similar arguments yield
$$
\sup_{0\leq l \leq n} \mathbb{E}[|\hat\Qc^{\zeta,\Delta,N}_{t_{l}}|^2]\leq \frac{\exp(C_2(\zeta))}{N}.
$$ 
The two previous upper-bounds in turn imply
\begin{align*}
    \lvert (J_1^{(T),\Delta}-\bar\Jc_1)(\theta)\rvert&\leq C (\lVert Q\rVert_F+\lVert R\rVert_F\lVert\theta\rVert_F^2)\frac{{h}}{N}\sum_{j=1}^N \sum_{l=0}^{n-1} e^{-\beta t_l} \E[\lvert  Y^{\theta,\Delta,(j)}_{t_l}-\yb^{\theta,\Delta,(j)}_{t_l} \rvert^2]\\
    &\leq C T (\lVert Q\rVert_F+\lVert R\rVert_F\lVert\theta\rVert_F^2)\sup_{l=0,\dots,n-1}\E[\lvert  \Qc^{\theta,\Delta}_{t_l}\rvert^2]\\
    &\leq \frac{\exp(C_1(\theta)) }{N} (\lVert Q\rVert_F+\lVert R\rVert_F\lVert\theta\rVert_F^2),
\end{align*}
and similarly
\begin{align*}
    \lvert (J_2^{(T),\Delta}-\bar\Jc_2)(\zeta)\rvert
    &\leq \frac{\exp(C_2(\zeta))}{N} (\lVert Q\rVert_F+\lVert R\rVert_F\lVert\zeta\rVert_F^2),
\end{align*}
up to a modification of $C_1(\theta)$ and $C_2(\zeta)$. Recalling that $\Rc(b)$ is compact, the right-hand side of the two previous inequalities can be uniformly bounded by a constant depending only upon $b$. The proof of \eqref{estimlemD5} is now complete.

\vspace{3mm}

\noindent \textit{Step 3:} 
After recalling that $\Theta_i\in\Rc(2b)$, it directly follows from \eqref{estimlemD5} that
$$
    \lvert (\bar\Jc^{\Delta,N}_{pop}-J^{(T),\Delta})(\Theta_i)\rvert\leq \frac{c_4(2b)}{N} 
$$
     and, since $\lVert U_i\rVert_F=\lVert V_i\rVert=r$, $\P-$ \emph{a.s}
    \begin{align*}
       \lVert ( \hat\nabla^{\Delta,N,pop}_{\theta} -\hat\nabla_\theta^{(T),\Delta}) J(\Theta)\rVert_F&\leq \frac{d}{r^2}\frac{1}{\tilde N}\sum_{i=1}^{\tilde N}\lvert (\bar\Jc^{\Delta,N}_{pop}-J^{(T),\Delta})(\Theta_i)\rvert\lVert U_i\rVert_F \; 
        \leq \;  \frac{d}{r}\frac{c_4(2b)}{N},  \\
          \lVert ( \hat\nabla^{\Delta,N,pop}_{\zeta} -\hat\nabla_\zeta^{(T),\Delta}) J(\Theta)\rVert_F&\leq \frac{d}{r^2}\frac{1}{\tilde N}\sum_{i=1}^{\tilde N}\lvert (\bar\Jc^{\Delta,N}_{pop}-J^{(T),\Delta})(\Theta_i)\rvert\lVert V_i\rVert_F \; \leq \;  \frac{d}{r}\frac{c_4(2b)}{N},  
    \end{align*}
which ends the proof.     
\ep

\vspace{5
mm}

\bibliographystyle{plain}
\bibliography{Publish}
\end{document}